\documentclass{article}
\usepackage{amsmath,amsthm,amssymb}

\title{The transmission problem with imperfect interfaces of small resistance\thanks{This work was supported by a KIAS Individual Grant (MG089001) at Korea Institute for Advanced Study.}}
\author{Shota Fukushima\thanks{Graduate School of Science and Technology, Gunma University, Gunma 376-8515, Japan. Email: \texttt{shota.fukushima@gunma-u.ac.jp}} \and Yong-Gwan Ji\thanks{School of Mathematics, Korea Institute for Advanced Study, Seoul 02455, S. Korea. Email: \texttt{ygji@kias.re.kr}} \and Hyeonbae Kang\thanks{Department of Mathematics and Institute of Applied Mathematics, Inha University, Incheon 22212, S. Korea. Email: \texttt{hbkang@inha.ac.kr}}}

\usepackage{mathrsfs,bm}
\usepackage{enumitem, mathtools,tablists}

\usepackage[labelfont=bf]{caption}

\usepackage{here}

\newcommand{\beq}{\begin{equation}}%
\newcommand{\eeq}{\end{equation}}%
\makeatletter
\newcommand{\leqnomode}{\tagsleft@true\let\veqno\@@leqno}
\makeatother

\newcommand{\jbk}[1]{\left\langle {#1} \right\rangle}

\newcommand{\e}{\mathrm{e}}
\newcommand{\iu}{i}

\newcommand{\df}{d}

\newcommand{\lap}{\Delta}

\newcommand{\nv}{\nu}
\newcommand{\fff}{\mathrm{I}}
\newcommand{\sff}{\mathrm{II}}
\DeclareMathOperator{\tr}{tr}
\DeclareMathOperator{\ad}{ad}

\DeclareMathOperator{\supp}{supp}
\DeclareMathOperator{\pv}{p.\!{}v.\!{}}
\DeclareMathOperator{\Ker}{Ker}
\DeclareMathOperator{\Ran}{Ran}
\DeclareMathOperator{\lspan}{span}
\DeclareMathOperator{\dom}{dom}

\newcommand{\p}{\partial}

\renewcommand{\Bbb}{\mathbb{B}}
\newcommand{\Cbb}{\mathbb{C}}

\newcommand{\Nbb}{\mathbb{N}}

\newcommand{\Rbb}{\mathbb{R}}

\newcommand{\Rscr}{\mathscr{R}}

\newcommand{\Acal}{\mathcal{A}}
\newcommand{\Bcal}{\mathcal{B}}

\newcommand{\Dcal}{\mathcal{D}}
\newcommand{\Ecal}{\mathcal{E}}
\newcommand{\Fcal}{\mathcal{F}}

\newcommand{\Hcal}{\mathcal{H}}

\newcommand{\Kcal}{\mathcal{K}}
\newcommand{\Lcal}{\mathcal{L}}
\newcommand{\Mcal}{\mathcal{M}}
\newcommand{\Ncal}{\mathcal{N}}

\newcommand{\Pcal}{\mathcal{P}}
\newcommand{\Qcal}{\mathcal{Q}}
\newcommand{\Rcal}{\mathcal{R}}
\newcommand{\Scal}{\mathcal{S}}
\newcommand{\Tcal}{\mathcal{T}}
\newcommand{\Ucal}{\mathcal{U}}

\newcommand{\Xcal}{\mathcal{X}}
\newcommand{\Ycal}{\mathcal{Y}}

\newcommand{\Ge}{\epsilon}

\newcommand{\Gvf}{\varphi}
\newcommand{\Gg}{\gamma}

\newcommand{\Gv}{\nu}

\newcommand{\Gt}{\theta}

\newcommand{\Gs}{\sigma}

\newcommand{\Gy}{\psi}

\newtheorem{prop}{Proposition}[section]
\newtheorem{theo}[prop]{Theorem}
\newtheorem{coro}[prop]{Corollary}
\newtheorem{lemm}[prop]{Lemma}

\theoremstyle{definition}
\newtheorem{defi}[prop]{Definition}

\newtheorem*{exam*}{Example}
\newtheorem*{rema*}{Remark}

\numberwithin{equation}{section}

\usepackage{xcolor}
\definecolor{blue800}{HTML}{0031d8}
\definecolor{myred}{HTML}{ce0000}

\usepackage{hyperref}
\hypersetup{%
colorlinks=true,
linkcolor=blue800,
urlcolor=blue800,
citecolor=blue800,
}
\usepackage{tikz}
\usetikzlibrary{math, cd, decorations.pathmorphing}

\newcommand{\SL}[1][\partial D]{\Scal_{#1}}
\newcommand{\DL}[1][\partial D]{\Dcal_{#1}}
\newcommand{\NP}[1][\partial D]{\Kcal_{#1}}

\DeclareMathOperator{\Capac}{\mathbf{Cap}}
\DeclareMathOperator{\capac}{cap}

\mathtoolsset{showonlyrefs=true}
\numberwithin{equation}{section}

\begin{document}
\maketitle

\begin{abstract}
We consider the transmission problem in presence of interfaces with imperfect bonding. The imperfect bonding condition is characterized by the positive resistance along the interface, which causes discontinuity of the potential across the interface while the flux is continuous. If the interface resistance is zero, then the interface is of perfect bonding, where both the potential and the flux of the solution are continuous across the interface. In this paper, we first construct using layer potentials the solution to the transmission problem with imperfect interfaces. We then prove that the solutions converge in various Sobolev spaces to the solution to the transmission problem with perfect interfaces as the interface resistance tends to zero. In particular, it is shown that the gradient of the solution converges in the uniform norm if the boundary is sufficiently regular.
\end{abstract}



\tableofcontents

\section{Introduction}

Let $D\subset \Rbb^d$ ($d\geq 2$) be a bounded open set with the compact Lipschitz boundary $\partial D$. We decompose $D$ into connected components as
\[
    D=\bigcup_{j=1}^N D_j.
\]
We denote by $D^+$ the exterior $\Rbb^d\setminus \overline{D}$ of $D$. For a function $u$ defined on $\Rbb^d\setminus \partial D$, we denote by $u|_-$ the boundary trace from the interior domain $D$ and by $u|_+$ that from the exterior domain if they exist. The subject of this paper is the following transmission problem for perfect conductors:
\begin{equation}\leqnomode
    \label{IPB}\tag{ImPB}
    \begin{dcases}
        \lap u=0 & \text{in } \Rbb^d\setminus \partial D, \\
        u|_+-\gamma \p_\nv u|_+=u|_-=\mathrm{const.} & \text{on } \partial D_j, \ j=1, \ldots, N,  \\
        \int_{\partial D_j} \p_\nv u|_+ \, \df \sigma=0 & \text{for } j=1, \ldots, N,  \\
        u (x)-h(x)=O(|x|^{-d+1}) & \text{as } |x|\to \infty,
    \end{dcases}
\end{equation}
where $h$ is a given harmonic function on $\Rbb^d$ and the constants in the second condition are not given, but determined by the third condition, and may differ depending on the index $j$ of the connected component of $D$. Here and throughout this paper, $\p_\nv u$ denotes the normal derivative where the normal vector directs outward with respect to $D$.

This problem is realized as the limit, as the conductivity $k$ of $D$ tends to $\infty$ (so yielding the name `the perfect conducting problem'), of the same problem with the second and third conditions replaced with the following two conditions:
\begin{equation}\label{imperfect}
u|_+-\gamma \p_\nv u|_+=u|_-, \quad \p_\nv u |_+ = k \p_\nv u |_- \quad\mbox{on } \p D.
\end{equation}
In application, the boundary $\partial D$ is interpreted as the membrane, the interface between different phases or bonding of composite materials. The interface condition on $\p D$ is said to be perfect if both the potential $u$ and the associated flux are continuous across $\p D$, and imperfect or non-ideal if either $u$ or the flux fails to be continuous across the interface. The condition \eqref{imperfect} is imperfect if $\gamma >0$ because the potential $u$ is discontinuous, while it is perfect if $\gamma =0$. If the solution $u$ is discontinuous but its associated flux is continuous across the interface as in \eqref{IPB}, the interface condition is said to be of low conductivity type since it can be realized by taking a limit of a core-shell structure as the thickness of the shell and the conductivity inside the shell tend to zero such that the ratio of them converges to some positive quantity. Actually, this limit of the ratio corresponds to $\gamma^{-1}$. We refer to \cite{Benveniste-Miloh99, Milton23} for more details of the derivation of the imperfect interface condition. This scheme is studied in a mathematically rigorous way in \cite{BCF80}. There, this scheme is called interior reinforcement. The parameter $\gamma$ (or $\gamma^{-1}$) is referred to by various terms in different fields. On one hand, the quantity $\gamma$ is called (electric/thermal) interface resistance, contact resistance or Kapitza resistance (see \cite{Persson22, Torquato-Rintoul95}). On the other hand, its reciprocal $\gamma^{-1}$ is referred to as interface conductance, permeability coefficient or transfer coefficient (see \cite{TWSB19, VLA16}). The solution $u$ is interpreted as a static temperature distribution, electric or magnetic potential \cite{Hashin01}. An elastic analogue of \eqref{IPB} is also studied in relation to continuum mechanics of composite materials \cite{Hashin02}.

There are also mathematical studies of the imperfect interface for periodically placed inclusions (\cite{CPR09} for two-dimensional case and \cite{DallaRiva-Musolino13, Pernin99} for general dimensional case).
In \cite{Kang-Li19}, weakly neutral inclusions in two-dimensional plane are constructed with imperfect bonding interface condition. Here, the weakly neutral inclusion means a pair of $D$ and $\gamma$, where $\gamma$ may vary depending on the point on the interface $\partial D$, such that the solution $u$ decays faster at $\infty$, namely, the right hand side of the fourth condition in \eqref{IPB} is $O(|x|^{-d})$.

Recent interest in \eqref{IPB} arose from the estimates of the stress, the gradient $|\nabla u|$ of the solution $u$. In the perfect bonding case, namely, when $\Gg=0$, the stress may blow up as the distance $\Ge$ between $D_1$ and $D_2$ (assuming that there are two inclusions) tends to $0$, depending on the given harmonic function $h$. If the stress blows up, it is proved that the blow-up rate is $\Ge^{-1/2}$ in two-dimensions \cite{AKLLL07, Yun07} (see also Figure \ref{fg:two:disks}) and $(\Ge |\ln \Ge|)^{-1}$ in three-dimensions \cite{BLY09}. There is extensive literature regarding the gradient estimates for which we refer to \cite{DLY24, DLY25, Ji-Kang23IMRN} and references therein. Quite recently the stress estimate problem has been considered for \eqref{IPB} with a positive $\Gg$ as an approximation of the core-shell structure which is reminiscent of biological cells. It is proved in \cite{FJKLfse} that, if $D\subset \Rbb^2$ is the union of two distinct disks with same radii and $h$ is a linear function whose gradient is parallel to the shortest segment connecting the two disks, then, for fixed interface resistance $\gamma>0$, the gradient of the solution $u$ remains finite independent of $\Ge$. This result is generalized in \cite{DYZ26} to general dimensions and inclusions of general shape.

A natural question arises: Denoting the solution to \eqref{IPB} by $u^\Gg$, does $u^\Gg$ converge to $u^0$ as the parameter $\gamma>0$ tends to zero. The aim of this paper is to investigate this question. In the course of investigation, we also prove that the problem \eqref{IPB} admits a unique solution in appropriate Sobolev spaces.

In order to state the results in a precise manner, we introduce some notation.
We begin with function spaces where the unique solution to \eqref{IPB} exists. For $s\in \Rbb$ and $p\in [1, \infty]$, let $H^{s, p}(\Rbb^d)$ be the usual Sobolev space (also referred to as Bessel potential space) with the norm  
\[
    \|u\|_{H^{s, p}(\Rbb^d)}:=\|\Fcal^{-1} [\omega_s \Fcal [u]]\|_{L^p (\Rbb^d)},
\]
where $\omega_s(\xi):=(1+|\xi|^2)^{s/2}$ and $\Fcal$ is the Fourier transform. If $p\in (1, \infty)$ and $k \in \Nbb_0$ ($\Nbb_0$ is the set of all nonnegative integers), $H^{k, p}(\Rbb^d)$ coincides with the space of all functions in $L^p (\Rbb^d)$ whose weak derivatives of order up to $k$ also belong to $L^p (\Rbb^d)$, endowed with the natural norm. It is well-known that the family $\{H^{s, p}(\Rbb^d)\}_{s\in \Rbb, p\in (1, \infty)}$ is a complex interpolation scale in the sense that 
\begin{equation}\label{eq:Sobolev:interpolation}
    [H^{s_0, p_0}(\Rbb^d), H^{s_1, p_1}(\Rbb^d)]_\theta=H^{s, p}(\Rbb^d)
\end{equation}
provided $s_0, s_1\in \Rbb$, $s_0\neq s_1$, $p_0, p_1\in (1, \infty)$, $s=(1-\theta)s_0+\theta s_1$, $1/p=(1-\theta)/p_0+\theta /p_1$ and $\theta \in (0, 1)$. For a bounded Lipschitz $D\subset \Rbb^d$ and $s\in [0, \infty)$, the fractional Sobolev space $H^{s, p}(D)$ is defined by 
\[
    H^{s, p}(D):=\{ f|_D \mid f\in H^{s, p}(\Rbb^d)\}
\]
endowed with the quotient norm 
\[
    \|u\|_{H^{s, p}(D)}:=\inf \{ \|f\|_{H^{s, p}(\Rbb^d)}\mid f\in H^{s, p}(\Rbb^d), \ f|_D=u\}.
\]
Then the following properties are well-known for bounded Lipschitz $D\subset \Rbb^d$ \cite{Bergh-Loefstroem76,Jerison-Kenig95}. 
\begin{itemize}
    \item If $p\in (1, \infty)$ and $k \in \Nbb_0$, $H^{k, p}(D)$ coincides with the space of all functions in $L^p (D)$ whose weak derivatives of order up to $k$ also belong to $L^p (D)$, endowed with the natural norm.
    \item $[ H^{s_0, p_0}(D), H^{s_1, p_1}(D)]_\theta=H^{s, p}(D)$ provided $s_0, s_1\geq 0$, $p_0, p_1 \in (1, \infty)$, $\theta\in (0, 1)$, $s=(1-\theta)s_0+\theta s_1$ and $1/p=(1-\theta)/p_0+\theta/p_1$, where $[\cdot, \cdot]_\theta$ denotes the complex interpolation with parameter $\theta\in (0, 1)$. 
\end{itemize}

Next, we introduce Besov spaces. First, for $s\in \Rbb$ and $p, q\in [1, \infty]$, the Besov space $B^{s, p}_q (\Rbb^d)$ is defined as the collection of all $u$ satisfying
\[
    \|u\|_{B^{s, p}_q (\Rbb^d)}:=\sum_{j=0}^\infty \|2^{js}\Fcal^{-1}[\psi_j\Fcal[u]]\|_{L^p(\Rbb^d)}^q<\infty,
\]
where $\{\psi_j\}_{j=0}^\infty$ is the Littlewood--Paley decomposition, namely,
\begin{align*}
    &\psi_j (\xi):=\psi (2^{-j+1}|\xi|) \quad (j\geq 1) \text{ for some } \psi\in C_{\mathrm{c}}^\infty (\Rbb) \text{ with } \supp \psi\subset \left(\frac{1}{2}, 2\right), \\
    &\psi_0(\xi)+\sum_{j=1}^\infty \psi_j(\xi)=1 \text{ for } \xi\in \Rbb^d.
\end{align*}
It is well-known that, if $s=k+\theta$ for some $k\in \Nbb_0$ and $\theta\in (0, 1]$, and $p, q\in [1, \infty]$, the space $B^{s, p}_q (\Rbb^d)$ is characterized by the norm
    \begin{align*}
        \|u\|_{B^{s, p}_q (\Rbb^d)}:=&\sum_{|\alpha|\leq k} \|\partial^\alpha u\|_{L^p (\Rbb^d)}+\sum_{|\alpha|=k} [\partial^\alpha u]_{B^{\theta, p}_q (\Rbb^d)},
    \end{align*}
    where the seminorm $[\cdot]_{B^{\theta, p}_q(D)}$ is defined by 
    \begin{equation}\label{eq:Besov:difference}
        [u]_{B^{\theta, p}_q(\Rbb^d)}:=
        \begin{cases}
            \displaystyle \left(\int_{\Rbb^d}\|u-u(\cdot +h)\|_{L^p}^q\frac{dh}{|h|^{d+\theta q}}\right)^{1/q} & \text{if } \theta\in (0, 1), \\[10pt]
            \displaystyle \left(\int_{\Rbb^d}\|u(\cdot-h)+u(\cdot+h)-2u\|_{L^p}^q\frac{dh}{|h|^{d+\theta q}}\right)^{1/q} & \text{if } \theta=1.
        \end{cases}
    \end{equation}
    (See \cite{Triebel92}.) Here and throughout this paper, we employ the multi-index notation
\[
    \partial^\alpha:=\partial_{x_1}^{\alpha_1}\cdots \partial_{x_d}^{\alpha_d}, \ |\alpha|:=\alpha_1+\cdots+\alpha_d \quad (\alpha=(\alpha_1\ldots, \alpha_d)\in \Nbb_0^d).
\]
In the case when $q=\infty$, the seminorm $[\cdot]_{B^{\theta, p}_\infty}$ is interpreted as 
    \[
        [u]_{B^{\theta, p}_\infty (\Rbb^d)}=
        \begin{cases}
            \displaystyle \sup_{h\neq 0} \frac{\|u-u(\cdot +h)\|_{L^p}}{|h|^\theta} & \text{if } \theta\in (0, 1), \\[10pt]
            \displaystyle \sup_{h\neq 0}\frac{\|u(\cdot-h)+u(\cdot+h)-2u\|_{L^p}}{|h|^\theta} & \text{if } \theta=1.
        \end{cases}
    \]
    In particular, for $\theta\in (0, 1)$, 
\[
    [u]_{B^{\theta, \infty}_\infty (\Rbb^d)}=\sup_{x\neq y} \frac{|u(x)-u(y)|}{|x-y|^{\theta}}
\]
is nothing but the H\"older seminorm. The space $B^{s, \infty}_\infty (\Rbb^d)$ is also referred to as the H\"older--Zygmund space. 

Next, similarly to the fractional Sobolev spaces on domains, for a bounded Lipschitz $D\subset \Rbb^d$ and $s\in (0, \infty)$, the Besov space $B^{s, p}_q (D)$ is defined by 
\[
    B^{s, p}_q (D):=\{ f|_D \mid f\in B^{s, p}_q (\Rbb^d)\}
\]
endowed with the quotient norm 
\[
    \|u\|_{B^{s, p}_q (D)}:=\inf \{ \|f\|_{B^{s, p}_q (\Rbb^d)}\mid f\in B^{s, p}_q (\Rbb^d), \ f|_D=u\}.
\]
Similar characterization to \eqref{eq:Besov:difference} for $B^{s, p}_q(D)$ in terms of difference of function inside $D$ is known \cite{Dispa03}.   

The family of Besov spaces has the following interpolation properties \cite{Bergh-Loefstroem76, Jerison-Kenig95}. In what follows, $(\cdot, \cdot)_{\theta, q}$ denotes the real interpolation.
\begin{itemize}
    \item $(B^{s_0, p}_{q_0} (\Rbb^d), B^{s_1, p}_{q_1} (\Rbb^d))_{\theta, r}=B^{s, p}_r (\Rbb^d)$ provided $s_0, s_1\in \Rbb$, $s_0\neq s_1$, $p, q_0, q_1, r\in [1, \infty]$, $s=(1-\theta)s_0+\theta s_1$ and $\theta\in (0, 1)$.
    \item For bounded Lipschitz $D\subset \Rbb^d$, 
\[
    ( B^{s_0, p}_{q_0}(D), B^{s_1, p}_{q_1}(D))_{\theta, r}=B^{s, p}_r (D)
\]
provided $s_0, s_1> 0$, $s_0\neq s_1$, $p, q_0, q_1, r\in [1, \infty]$, $s=(1-\theta)s_0+\theta s_1$ and $\theta\in (0, 1)$.
\item For bounded Lipschitz $D\subset \Rbb^d$, 
\[
    (L^p (D), H^{k, p}(D))_{\theta, q}=B^{\theta k, p}_q (D)
\]
provided $k\in \Nbb$, $p, q\in [1, \infty]$ and $\theta\in (0, 1)$.
\end{itemize}
We define $B^{s, p}(D):=B^{s, p}_p (D)$ for short. It is well-known that $B^{s,2}(\Rbb^d)=H^{s, 2}(\Rbb^d)$ with the norm equivalence for any $s\in \Rbb$. Thus we have $B^{s,2}(D)=H^{s, 2}(D)$ for any open $D\subset \Rbb^d$ and $s>0$ as well. In fact, the continuous embeddings 
    \begin{align*}
        &B^{s, p}(D) \hookrightarrow H^{s, p}(D) \hookrightarrow B^{s, p}_2 (D) \quad (p\in (1, 2]), \\
        &B^{s, p}_2 (D) \hookrightarrow H^{s, p}(D) \hookrightarrow B^{s, p} (D) \quad (p\in [2, \infty))
    \end{align*}
    are valid for $s>0$ \cite{Bergh-Loefstroem76}.

The following is the first main result of this paper. 

\begin{theo}
    \label{theo:L2:sol:approx}
    Let $D\subset \Rbb^d$ ($d \ge 3$) be a bounded open set with Lipschitz boundary. Then there exists $\Ge_*=\Ge_* (D)\in (0, 1]$ which only depends on $D$ (and will be introduced in Theorem \ref{theo:layer:Lip}) such that there exists a unique solution $u^\gamma$ to \eqref{IPB} such that 
    \[
        u^\gamma |_{U\setminus \overline{D}}\in \bigcap_{1<p\leq 2} B^{1+1/p, p}_2 (U\setminus \overline{D})\cap \bigcap_{2\leq p<\frac{2}{1-\Ge_*}} H^{1+1/p, p}(U\setminus \overline{D})
    \]
    for any open ball $U\subset \Rbb^d$ containing $\overline{D}$. 

    If $\partial D$ is $C^1$, then we can take $\Ge_*=1$. Moreover, $u^\gamma$ converges to $u^0$ as $\Gg \to 0$ in $B^{1/p, p}_2 (U\setminus \overline{D})$ for any $p\in (1, 2]$ and in $H^{1/p, p}(U\setminus \overline{D})$ for any $p\in [2, 2/(1-\Ge_*))$.
\end{theo}

The above result only includes the case of dimension higher than 2. It is because its proof heavily relies on mapping properties of layer potentials on various $L^p$-Sobolev spaces and such properties are not available in two dimensions if $\p D$ is merely Lipschitz to the best of our knowledge. In two dimensions, we restrict ourselves to $L^2$-Sobolev spaces and show that the solution uniquely exists in $H^{s+1/2, 2}(U\setminus \overline{D})$ for any $s \in (0,1)$.

We are particularly interested in convergence of the gradients in the uniform norm in relation to blow-up of the gradient of the solution to \eqref{IPB} with $\Gg=0$. Actually this work is motivated by this question as mentioned earlier.
In order to deal with this question, we require regularity of domains stronger than $C^1$, say $C^{k,\theta}$ boundary for some integer $k \ge 1$ and $\theta \in (0,1]$. Under this assumption, we are able to include the two-dimensional case since various mapping properties are known even for the two-dimensional case.

The following theorem yields quantitative estimates for convergence of $u^\gamma$ to $u^0$. 

\begin{theo}
    \label{theo:existence}
    Let $D\subset \Rbb^d$ ($d\geq 2$) be a bounded open set with $C^{k,\theta}$ boundary for some integer $k \ge 1$ and $\theta \in (0,1]$. Let $U$ be an open ball containing $\overline{D}$. The following holds.
\begin{enumerate}[label=\textnormal{(\roman*)}]
    \item The solution $u^\gamma$ to the problem \eqref{IPB} belongs to $B^{s+1/p, p}(U\setminus \overline{D})$ ($p\in (1, \infty]$) and $H^{s+1/p, p}(U\setminus \overline{D})$ ($p\in (1, \infty)$) for any $s\in (1, k+\theta)$.

    \item \label{enum:asympt}There exists a unique harmonic function $v_l$ in $D^+$ ($l=0, 1, \ldots, k-1$) such that $v_0=u^0$, $v_l|_{U\setminus \overline{D}}$ belongs to $B^{s+1/p-l, p}(U\setminus \overline{D})$ for all $p\in [1, \infty]$ and $H^{s+1/p-l, p}(U\setminus \overline{D})$ for all $p\in (1, \infty)$, with the property that, for any $s\in (k, k+\theta)$ and $K=0, 1, \ldots, k-1$, both 
    \[
        \left\|u^\gamma-\sum_{l=0}^K \gamma^l v_l\right\|_{H^{s+1/p-1-K, p}(U\setminus \overline{D})} \quad (p\in (1, \infty))
    \]
    and 
    \[
        \left\|u^\gamma-\sum_{l=0}^K \gamma^l v_l\right\|_{B^{s+1/p-1-K, p}(U\setminus \overline{D})} \quad (p\in (1, \infty])
    \]
    are of order $o(\gamma^K)$ as $\gamma\to +0$.
        \end{enumerate}
\end{theo}

As consequences of Theorem \ref{theo:existence}, we obtain two corollaries. The first one is regarding the convergence of the gradient of the solution.
\begin{coro}\label{coro:stress}
    Let $U\subset \Rbb^d$ be an open ball containing $\overline{D}$. 
    \begin{enumerate}[label=\textnormal{(\roman*)}]
        \item\label{1theta} If $\partial D$ is $C^{1, \theta}$-smooth for some $\theta\in (0, 1]$, then it holds that
        \begin{equation}\label{eq:stress:conv}
            \lim_{\gamma\to +0} \| \nabla u^\gamma-\nabla u^0\|_{L^p (U\setminus \overline{D})}=0
        \end{equation}
        for any $p\in (1, 1/(1-\theta))$.
        \item\label{2theta} If $\p D$ is $C^{2,\theta}$-smooth for some $\theta\in (0, 1]$, then it holds that
        \begin{equation}\label{eq:stress:conv2}
            \lim_{\gamma\to +0} \| \nabla u^\gamma-\nabla u^0\|_{L^\infty (U\setminus \overline{D})}=0.
        \end{equation}
    \end{enumerate}
\end{coro}

For instance, the assertion \ref{1theta} is proved from Theorem \ref{theo:existence} \ref{enum:asympt} with $k=1$ and $K=0$ as follows. Since $1<s<1+\theta$, we have $1/p < s+1/p-1 < \theta + 1/p$. Thus, if $p < 1/(1-\theta)$, then $1< \theta + 1/p$, and hence
$$
\| \nabla u^\gamma-\nabla u^0\|_{L^p (U\setminus \overline{D})} \le \left\|u^\gamma-u^0 \right\|_{B^{s+1/p-1, p}(U\setminus \overline{D})} \to 0
$$
as $\Gg \to 0$.

% We see from \eqref{eq:convergence} that
% $$
% |\nabla R^\gamma_0(x)| \le C_\Gg (1+|x|)^{-d}, \quad x \in \Rbb^d\setminus \overline{U},
% $$
% where $C_\Gg=o(1)$ as $\Gg \to 0$. Thus, $\| \nabla R^\gamma_0 \|_{L^p(\Rbb^d\setminus \overline{U})} \to 0$ as $\Gg \to 0$ for all $p \in (1, \infty]$. 

As mentioned earlier, if $D$ consists of two locally strictly convex domains whose distance is $\Ge$, then for some $h$ the quantity $\|\nabla u^0\|_{L^\infty (D^+)}$ may be arbitrarily large as $\Ge$ tends to $0$. Corollary \ref{coro:stress} shows that $\lim_{\gamma\to +0} \| \nabla u^\gamma\|_{L^\infty (D^+)}$ can be arbitrarily large. Figure \ref{fg:two:disks} describes schematically the current understanding on dependency on $\Gg$ and $\Ge$ of $\|\nabla u^\gamma\|_{L^\infty (D^+)}$ when $D$ consists of two disks ($d=2$). The left downward arrow shows the blow-up of $\nabla u^\gamma$ when $\Gg=0$; the top arrow directed toward left indicates convergence of $\nabla u^\gamma$ to $\nabla u^0$ as $\Gg$ tends to $0$ (Corollary \ref{coro:stress}); the right downward arrow indicates the finiteness of $\nabla u^\gamma$ regardless of $\Ge$ when $\Gg$ stays away from $0$ which is proved in \cite{DYZ26, FJKLfse}. Knowing these results, we have a natural question: how $\nabla u^\gamma$ behaves as both $\Gg$ and $\Ge$ tend to $0$? Quite recently, this question seems to be solved in \cite{DLZ2510} including the higher dimensional case: for two disks on the plane, $\|\nabla u^\gamma\|_{L^\infty (D^+)}$ behaves like $(\gamma+\Ge)^{-1/2}$. The remaining intriguing question is the convergence of $u^\gamma$ to $u^0$ independent of $\Ge$. The aforementioned paper \cite{DLZ2510} proves that $u^\gamma$ converges to $u^0$ weakly in $H^{1, 2}$ and strongly in $L^2$ as $\gamma \to +0$ (on bounded sets in $D^+$). 

%It would be quite interesting to investigate how $\nabla u^\gamma$ behaves as both $\Gg$ and $\Ge$ tend to $0$.

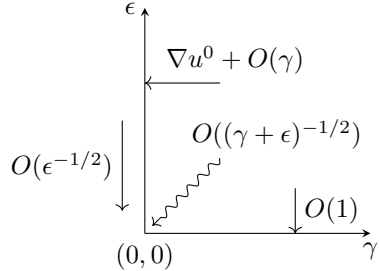
\begin{figure}[htb]
\centering
\begin{minipage}{0.5\textwidth}
    \begin{tikzpicture}[scale=1]
        \draw[->, >=stealth] (0, 0)--(3, 0);
        \draw[->, >=stealth] (0, 0)--(0, 3);
        \node (xaxis) at (3, 0) [below] {$\gamma$};
        \node (yaxis) at (0, 3) [left] {$\Ge$};
        \node (o) at (0, 0) [below] {$(0, 0)$};
        \draw[->] (-0.3, 1.5)--(-0.3, 0.3);
        \draw[->] (1, 2)--(0, 2);
        \draw[->] (2, 0.6)--(2, 0);
        \node (PB) at (-0.3, 0.9) [left] {$O(\Ge^{-1/2})$};
        \node (G0) at (1.2, 2) [above] {$\nabla u^0+O(\gamma)$};
        \node (E0) at (2, 0.3) [right] {$O(1)$};
        \node (GE) at (0.5, 1) [above right] {$O((\gamma+\Ge)^{-1/2})$};
        \draw[decorate, decoration={snake, amplitude=0.5mm, segment length=2mm}] (1, 1)--(0.2, 0.2);
        \draw[->] (0.2, 0.2)--(0.1, 0.1);
    \end{tikzpicture}
\end{minipage}
    \begin{minipage}{0.4\textwidth}
        
        \caption{$(\gamma, \Ge)$-dependency of the gradient estimates for two disks, where $\Ge$ denotes the distance between two disks.}
    \label{fg:two:disks}
    \end{minipage}
\end{figure}

The second corollary is an asymptotic expansion of the solution in terms of $\gamma$ on the smooth boundary: 

\begin{coro}
    \label{coro:asymptotic:exp}
    If $\partial D$ is $C^\infty$ in addition to the assumptions of Theorem \ref{theo:existence} for $D$, then the derivatives $\partial^\alpha v_k$ of the harmonic function $v_k$ in Theorem \ref{theo:existence} admits a continuous extension to $\overline{D^+}$ for any multi-index $\alpha$. Furthermore, the solution $u^\gamma$ admit the asymptotic expansion
    \[
        u^\gamma\sim \sum_{l=0}^\infty \gamma^l v_l \quad (\gamma\to +0)
    \]
    in the sense that
    \[
        \left\| \partial^\alpha \left(u^\gamma-\sum_{l=0}^K \gamma^l v_l\right)\right\|_{L^\infty (U\setminus \overline{D})}=O(\gamma^{K+1})
    \]
    as $\gamma\to +0$ for all $\alpha\in \Nbb_0^d$, $K\in \Nbb_0$ and any open ball $U\subset \Rbb^d$ including $\overline{D}$.
\end{coro}
The above corollary is also immediately proved by Theorem \ref{theo:existence} \ref{enum:asympt} since $v_l\in B^{s, \infty}(U\setminus \overline{D})$ for all $s\geq 1$.

We outline the organization of this paper. Section \ref{sect:Lipschitz} is devoted to the proof of Theorem \ref{theo:L2:sol:approx}. We employ the layer potential technique in an essential way for constructing the solution to \eqref{IPB}. So we collect basic facts on the layer potential operators in Section \ref{subs:layer:mapping}. Next, we investigate the Dirichlet-to-Neumann mapping for the exterior domain by using the layer potential operators in Section \ref{subs:dtn}. Finally, we associate a capacitance-type matrix with the imperfect interface problem \eqref{IPB}, which we call the resistive capacitance matrix and prove that it is invertible for any $\gamma>0$ in Section \ref{subs:capacitance}. Based on these preliminaries, we prove Theorem \ref{theo:L2:sol:approx} in Section \ref{subs:proof:Lip} by reducing the problem \eqref{IPB} to the boundary integral equation involving the Dirichlet-to-Neumann mapping. The invertibility of the resistive capacitance matrix plays a crucial role in solving the boundary integral equation. We also prove that the solution $u^\gamma$ to \eqref{IPB} belongs to $H^{s+1/p, p} (U\setminus \overline{D})\cap B^{s+1/p, p}(U\setminus \overline{D})$ for any open ball $U\subset \Rbb^d$ including $\overline{D}$ and appropriate $(s, p)\in (0, 1) \times (1, \infty)$, where $p$ can be arbitrarily large if $s$ is sufficiently small (Theorem \ref{theo:sol:membership}). Next, we prove Theorem \ref{theo:existence} in Section \ref{sect:smooth}. In particular, an expansion of the resolvent in terms of $\gamma$ (Proposition \ref{prop:perturbation}) plays a significant role in the proof of the main theorem. We end this paper with a short discussion (Section \ref{sect:discussion}). Appendix \ref{sect:pf:smoothing} provides a proof of smoothing properties of a certain boundary integral operator (Theorem \ref{theo:layer:Sobolev} \ref{enum:np:smoothing:1}), which is employed in Section \ref{sect:smooth}.

% \subsection*{Notations}

% Hereinafter, $D$ always denotes a bounded open subset with Lipschitz boundary. When $d=2$, we further assume that $\capac_D>1$.

%we make the following global assumption for any open subset $D\subset \Rbb^d$.
%\begin{assu}\label{assu:domain}
%    $D\subset \Rbb^d$ ($d\geq 2$) is a bounded open subset with compact Lipschitz boundary.  %\end{assu}

\section{Lipschitz case: Proof of Theorem \ref{theo:L2:sol:approx}}
\label{sect:Lipschitz}

In order to prove Theorem \ref{theo:L2:sol:approx}, we first recall basic facts on the layer potentials. We then introduce the Dirichlet-to-Neumann map associated with the exterior problem and construct a resistive variant of the capacitance matrix which plays a crucial role in this paper. Throughout this section, we assume that $\partial D$ is Lipschitz.

\subsection{Layer potential operators}\label{subs:layer:mapping}

We denote by $H^{s, p}(\partial D)$ and $B^{s, p}(\partial D)$ the fractional Sobolev space and the Besov space on $\partial D$ respectively, which are defined in a standard way using local graph representation of $\partial D$ and the delocalization by a partition of unity on $\partial D$. Since $\partial D$ is a closed manifold (whence it does not have a boundary), Sobolev and Besov spaces on it inherit the interpolation properties from $H^{s, p}(\Rbb^{d-1})$ and $B^{s, p}(\Rbb^{d-1})$ as follows: for $s_0, s_1\in [-1, 1]$, $s_0\neq s_1$, $s=(1-\theta)s_0+\theta s_1$, $1/p=(1-\theta)/p_0+\theta /p_1$ and $\theta\in (0, 1)$
\begin{itemize}
    \item $[H^{s_0, p_0}(\partial D), H^{s_1, p_1}(\partial D)]_\theta=H^{s, p}(\partial D)$ ($p_0, p_1\in (1, \infty)$),
    \item $(B^{s_0, p}(\partial D), B^{s_1, p}(\partial D))_{\theta, p}=B^{s, p}(\partial D)$ ($p\in [1, \infty]$).
\end{itemize}
If $\partial D$ is $C^{k, \theta}$, then the range of $s_0$ and $s_1$ extends to $(-k-\theta, k+\theta)$.

For a function space $\Xcal=H^{s, p}, B^{s, p}, L^p$, we add the tilde symbol to represent the subspace
\[
    \widetilde{\Xcal}(\partial D):=\left\{ \varphi\in \Xcal(\partial D) \,\middle|\, \int_{\partial D_j} \varphi\, \df \sigma=0, \   j=1, \ldots, N\right\}.
\]
As we only consider the fractional Sobolev spaces and Besov spaces with exponent not less than $-1$ throughout this paper, the spaces $\widetilde{\Xcal}(\partial D)$ becomes a closed subspace of $\Xcal (\partial D)$. 

Let $\Gamma (x)$ be the fundamental solution to the Laplace operator:
\begin{equation}
    \label{eq:fundamental:solution}
    \Gamma (x):=
    \begin{dcases}
        \frac{1}{2\pi}\log |x| & \text{if } d=2, \\
        -\frac{1}{(d-2)\omega_d |x|^{d-2}} & \text{if } d\geq 3,
    \end{dcases}
\end{equation}
where $\omega_d$ is the area of the unit sphere in $\Rbb^d$. Then, for an appropriate function $\varphi$ on $\partial D$, the single and double layer potentials of $\varphi$ are respectively defined by
\begin{equation}\label{eq:single:layer}
    {\SL}[\varphi](x):=\int_{\partial D} \Gamma (x-y)\varphi (y)\, \df \sigma (y) \quad (x\in \Rbb^d)
\end{equation}
and
\begin{align*}
    {\DL}[\varphi](x):=&\,\int_{\partial D} ((\nv_y\cdot \nabla_y)\Gamma (x-y))\varphi (y)\, \df \sigma (y) \\
    =&\,-\frac{1}{\omega_d}\int_{\partial D}\frac{(x-y)\cdot \nv_y}{|x-y|^d}\varphi (y)\, \df \sigma (y) \quad (x\in \Rbb^d\setminus \partial D),
\end{align*}
where $\nv_y$ is the outward unit normal vector at $y\in \partial D$. The Neumann--Poincar\'e operator (abbreviated to NP operator) is a boundary integral operator defined by
\[
    {\NP}[\varphi](x):=-\frac{1}{\omega_d}\pv\int_{\partial D}\frac{(x-y)\cdot \nv_y}{|x-y|^d}\varphi (y)\, \df \sigma (y) \quad (x\in \partial D)
\]
for an appropriate function $\varphi$ on $\partial D$. We denote the $L^2$-adjoint of $\NP$ by $\NP^*$, which is also called the NP operator.

Boundary integral operators such as the NP operator and the boundary value of single layer potential $\varphi\mapsto {\SL}[\varphi]|_{\partial D}$, which we still denote by ${\SL}$, have the following mapping properties. Although we only require the mapping properties between fractional Sobolev spaces for the proof of Theorem \ref{theo:L2:sol:approx}, we simultaneously work on the Besov spaces for later use. 

\begin{theo}\label{theo:layer:bdd:Lip}
    Let $D\subset \Rbb^d$ ($d\geq 2$) be a bounded open set with the Lipschitz boundary and $p\in (1, \infty)$. The following linear operators are bounded.
    \begin{enumerate}[label=\textnormal{(\alph*)}]
        \item \label{enum:sl:bdd:Lip} $\SL: H^{s-1, p}(\partial  D)\to H^{s, p}(\partial  D)$ for $s\in [0, 1]$ and $\SL: B^{s-1, p}(\partial  D)\to B^{s, p}(\partial  D)$ for $s\in (0, 1)$.
        \item \label{enum:NP:bdd:Lip}$\NP: H^{s, p}(\partial  D)\to H^{s, p}(\partial  D)$ for $s\in [0, 1]$ and $\NP: B^{s, p}(\partial  D)\to B^{s, p}(\partial  D)$ for $s\in (0, 1)$ 
        \item \label{enum:NPs:bdd:Lip}$\NP^*: H^{s-1, p}(\partial D)\to H^{s-1, p}(\partial D)$ for $s\in [0, 1]$ and $\NP^*: B^{s-1, p}(\partial D)\to B^{s-1, p}(\partial D)$ for $s\in (0, 1)$
    \end{enumerate}
\end{theo}

\begin{proof}
The boundedness of $\SL: L^p (\partial  D)\to H^{1, p}(\partial  D)=H^{1, p}(\partial D)$ is proved in \cite[Lemma 1.8]{Verchota84}. Then the assertion \ref{enum:sl:bdd:Lip} follows after applying standard arguments of the duality and the interpolation. The boundedness of $\NP: H^{k, p}(\partial  D)\to H^{k, p}(\partial  D)$ for $k=0, 1$ can be found in \cite[Theorem 5.1]{Mitrea94} and the interpolation again yields \ref{enum:NP:bdd:Lip}. \ref{enum:NPs:bdd:Lip} is proved by the duality.
\end{proof}

We also recall the invertibility and Fredholm property of the layer potentials. We need a notation: for $\Ge\in (0, 1]$, let $\Rscr_\Ge$ be the interior of the hexagon $\mathrm{OABCDE}$ where $\mathrm{O,A,B,C,D,E}$ are the planar points $(0, 0)$, $(\Ge, 0)$, $(1, (1-\Ge)/2)$, $(1, 1)$, $(1-\Ge, 1)$, $(0, (1+\Ge)/2)$, respectively, as shown in Figure \ref{fg:hexagon}. 

\begin{figure}[H]
        \centering
        \begin{minipage}{0.5\textwidth}
        \tikzmath{\a=0.5;}
        \begin{tikzpicture}[scale=1.2]
            \filldraw[fill=black!20, dotted] (0, 0) node [below left] {O}--(\a, 0)--(2, {1-\a})--(2, 2)--({2-\a}, 2)--(0, {1+\a})--cycle;
            \draw[->, >=stealth] (0, 0)--(2.5, 0) node [below] {$s$};
            \draw[->, >=stealth] (0, 0)--(0, 2.5) node [left] {$1/p$};
            \node (10) at (2, 0) [below] {$1$};
            \node (01) at (0, 2) [left] {$1$};
            \node (1/20-) at (\a, 0) [below] {$\Ge$};
            \node (1/20-) at ({2-\a}, 0) [below] {$1-\Ge$};
            \node (01/2-) at (0, {1-\a}) [left] {$(1-\Ge)/2$};
            \node (01/2+) at (0, {1+\a}) [left] {$(1+\Ge)/2$};
            \draw[dashed] (2, 0)--(2, 2)--(0, 2);
            \draw[dashed] ({2-\a}, 0)--({2-\a}, 2);
            \draw[dashed] (0, {1-\a})--(2, {1-\a});
            \coordinate (0) at (2, 1.9);
            \coordinate (1) at (1.1, 1.6);
            \coordinate (12) at (2, 1.6);
            \coordinate (2) at (1, {1.6-0.3*0.9/0.9});
            \coordinate (23) at (2, {1.6-0.3*0.9/0.9});
            \coordinate (3) at (1, {1.6-0.3*0.9/0.9*2});
            \node (R) at (1, 1) {$\Rscr_\Ge$};
        \end{tikzpicture}
    \end{minipage}
    \begin{minipage}{0.4\textwidth}
        
        \caption{The hexagon $\Rscr_\Ge$.}
        \label{fg:hexagon}
    \end{minipage}
    \end{figure}
The following theorem is well-known: the assertions \ref{enum:sl:inv:Lip}--\ref{enum:NP:inv:Lip} are proved in {\cite[Theorem 4.17]{Dahlberg-Kenig87} and \cite[Section 8]{FMM98}}, and \ref{enum:sl:inv:Lip:end}--\ref{enum:NP:inv:Lip:end} are proved in the proofs of \ref{enum:sl:inv:Lip}--\ref{enum:NP:inv:Lip}.

\begin{theo}\label{theo:layer:Lip}
    Let $D\subset \Rbb^d$ ($d\geq 2$) be a bounded open set with the Lipschitz boundary. If $d\geq 3$, then there exists $\Ge \in (0, 1]$ (depending on $D$) such that the following statements are valid for $(s, 1/p)\in \Rscr_\Ge$.
    \begin{enumerate}[label=\textnormal{(\alph*)}]
        \item\label{enum:sl:inv:Lip} $\SL: B^{s-1, p}(\partial  D)\to B^{s, p}(\partial  D)$ is an isomorphism.
        \item\label{enum:NP:Fredholm:Lip} $\pm 1/2I+\NP: B^{s, p}(\partial  D)\to B^{s, p}(\partial  D)$ and  $\, \pm 1/2I+\NP^*: B^{s-1, p}(\partial D)\to B^{s-1, p}(\partial D)$ are bounded Fredholm operators of index zero.
        \item\label{enum:NP:inv:Lip} $-1/2I+\NP^*: \widetilde{B}^{s-1, p}(\partial D) \to \widetilde{B}^{s-1, p}(\partial D)$ is an isomorphism.
    \end{enumerate}
    Let $\Ge_*=\Ge_*(D)$ be the supremum of such $\Ge$. Then the following assertions hold for any $p\in (1, 2/(1-\Ge_*))$. 
    \begin{enumerate}[label=\textnormal{(\alph*)}]
        \setcounter{enumi}{3}
        \item \label{enum:sl:inv:Lip:end} $\SL: L^p (\partial  D)\to H^{1, p} (\partial  D)$ is an isomorphism.
        \item\label{enum:NP:Fredholm:Lip:end} $\pm 1/2I+\NP: H^{1, p}(\partial  D)\to H^{1, p}(\partial  D)$ and  $\, \pm 1/2I+\NP^*: H^{-1, p}(\partial D)\to H^{-1, p}(\partial D)$ are bounded Fredholm operators of index zero.
        \item\label{enum:NP:inv:Lip:end} $-1/2I+\NP^*: \widetilde{L}^p (\partial D) \to \widetilde{L}^p(\partial D)$ is an isomorphism.
    \end{enumerate}
    
    If $\partial D$ is $C^{1, 0}$, then $\Ge_* (D)=1$.

    When $d=2$ and $\capac_D>1$, \ref{enum:sl:inv:Lip}--\ref{enum:NP:inv:Lip} hold for $p=2$ and $s\in (0, 1)$, and \ref{enum:sl:inv:Lip:end}--\ref{enum:NP:inv:Lip:end} for $p=2$.
\end{theo}

In the statement of the above theorem, $\capac_D$ denotes the logarithmic capacity. While referring to \cite[\S 2.3.5]{Ammari-Kang07} for its definition, we mention that the condition $\capac_D>1$ guarantees that the single layer potential is invertible. We emphasize that the condition $\capac_D>1$ is mild; if this condition fails to hold for some $D$, then we may dilate $D$ so that the condition holds (see \cite[4.10. Remark]{Verchota84}). 

Concerning \ref{enum:NP:inv:Lip} and \ref{enum:NP:inv:Lip:end}, we remind that the tilde symbol $\widetilde{\Xcal}(\partial D)$ stands for the subspace of a function space $\Xcal (\partial D)$ consisting of $\varphi\in \Xcal (\partial D)$ such that $\int_{\partial D_j} \varphi \, d\sigma=0$ for any connected component $D_j$ of $D$. 

In what follows, we use the notation 
\[
    \Rscr (D):=\Rscr_{\Ge*},
\]
where $\Ge_*=\Ge_*(D)$ is as defined in Theorem \ref{theo:layer:Lip}. We will also use the quantity $\Ge_*$ itself.

We will employ the following jump relations (see, for example, \cite{MMP94} for proofs):
\begin{align}
    \partial_\nv {\SL}[\varphi]|_\pm &= \left(\pm \frac{1}{2}I+\NP^*\right) [\varphi] \quad (\varphi\in B^{s-1, p}(\partial D)), \label{eq:sl:jump} \\
    {\DL}[\psi]|_\pm &=\left(\mp \frac{1}{2}I+\NP\right) [\psi] \quad (\psi\in B^{s, p}(\partial D)) \label{eq:dl:jump}
\end{align}
where $s\in (0, 1)$ and $p\in (1, \infty)$.

Let $\nu=(\nu_1, \ldots, \nu_d)$ be the unit normal vector on $\p D$ as before. For $1 \le j<k \le d$, let $\tau_{jk}= (0, \ldots, -\nu_k (\text{$j$-th position}), \ldots, \nu_j (\text{$k$-th position}), \ldots, 0)$. Then, $\tau_{jk}$ is tangential to $\p D$. So a tangential derivative on $\p D$ is defined by
\begin{equation}\label{900}
\partial_{\tau_{jk}}:= \tau_{jk} \cdot \nabla = \nv_j \partial_{x_k}-\nv_k \partial_{x_j}.
\end{equation}
Note that the following identity holds for $\varphi, \psi \in H^{1,p}(\p D)$:
\begin{equation}\label{eq:tan:dual}
    \int_{\partial D} \partial_{\tau_{jk}} \varphi \psi \, \df \sigma
    =\int_{\partial D} \varphi \partial_{\tau_{jk}} \psi \, \df \sigma.
\end{equation}
It is worth mentioning that this formula can be proved by extending $\varphi, \psi$ to $D$ and applying the divergence theorem. If $\nabla^{\tan}$ denotes the tangential gradient, namely,
$$
\nabla^{\tan} u(x)=\nabla u(x)-\partial_\nv u(x)\nv(x), \quad x \in \p D,
$$
then we have
\begin{equation}\label{910}
\nabla^{\tan}=-\left( \sum_{j=1}^d \nv_j \partial_{\tau_{1j}}, \ldots, \sum_{j=1}^d \nv_j \partial_{\tau_{dj}} \right).
\end{equation}

\begin{lemm}
\label{lemm:reduction:ds}
Let $D\subset \Rbb^d$ be a bounded open set with the Lipschitz boundary and let $\varphi\in C^{0, 1}(\partial D)$. Then 
\begin{equation}
    \label{eq:sl:dl:id}
        \partial_{x_i}{\DL}[\varphi](x)=\sum_{j=1}^d \partial_{x_j}{\SL}[\partial_{\tau_{ij}}\varphi](x)
\end{equation}
for $x\in D$. 
\end{lemm}

\begin{proof}
    We obtain through direct calculations
    \begin{align*}
        \partial_{x_i}{\DL}[\varphi](x)&=\sum_{j=1}^d \int_{\partial D} \nv_j (y)\partial_{x_i}\partial_{y_j}\Gamma (x-y) \varphi(y)\, \df \sigma (y) \\
        &=-\sum_{j=1}^d\int_{\partial D} \nv_j (y)\partial_{y_i}\partial_{y_j}\Gamma (x-y) \varphi(y)\, \df \sigma (y) \\
        &=-\sum_{j=1}^d\int_{\partial D} \nv_i (y)\partial_{y_j}^2\Gamma (x-y) \varphi(y)\, \df \sigma (y) \\
        &\quad +\sum_{j=1}^d\int_{\partial D} \partial_{\tau_{ij}}\partial_{y_j}\Gamma (x-y) \varphi(y)\, \df \sigma (y) .
    \end{align*}
    Since $\lap \Gamma (x-y)=0$ for $x\neq y$, it follows from \eqref{eq:tan:dual} that
    \begin{align*}
        \partial_{x_i}{\DL}[\varphi](x)&=\sum_{j=1}^d\int_{\partial D} \partial_{x_j}\Gamma (x-y) \partial_{\tau_{ij}}\varphi(y)\, \df \sigma (y) =\sum_{j=1}^d \partial_{x_j}{\SL}[\partial_{\tau_{ij}}\varphi](x). 
    \end{align*}
    This completes the proof. 
\end{proof}

We also need the following self-adjoint realization and an elementary spectral bound of NP operators. 

\begin{theo}
    \label{theo:NP:energy}
    Let $D\subset \Rbb^d$ ($d\geq 2$) be a bounded open set with the Lipschitz boundary. Then,
    \begin{enumerate}[label=\textnormal{(\alph*)}] %[label=\textnormal{(\roman*)}]
        \item \label{enum:Plemelj} the following Plemelj symmetrization principle holds for $p \in (1, \infty)$ and $\varphi\in H^{-1, p}(\partial D)$:
        \begin{equation}
            \label{eq:Plemelj}
            \SL\NP^*[\varphi]=\NP{\SL}[\varphi].
        \end{equation}

    \end{enumerate}

    If in addition we assume $\capac_D>1$ when $d=2$ (no further assumption is needed when $d\geq 3$), then the following holds.
    \begin{enumerate}[label=\textnormal{(\alph*)}]
        \setcounter{enumi}{1}
        \item \label{enum:inner:product}The bilinear form
        \begin{equation}\label{eq:inner:product}
        \jbk{\cdot, \cdot}_*:=-\jbk{\cdot, \SL^{-1}[\cdot]}
        \end{equation}
        is an inner product on $B^{1/2, 2}(\partial D)$ whose induced norm is equivalent to the natural norm on $B^{1/2, 2}(\partial D)$. Here $\jbk{\cdot, \cdot}$ is the dual pairing between $B^{1/2, 2}(\partial D)$ and $B^{-1/2, 2}(\partial D)$.
        \item \label{enum:NP:sa} The operator $\NP$ is a self-adjoint operator on $B^{1/2, 2}(\partial D)$ with the inner product \eqref{eq:inner:product} and its spectrum is included in $[-1/2, 1/2]$.
    \end{enumerate}
\end{theo}

\begin{proof}
Proofs of \ref{enum:Plemelj} and \ref{enum:inner:product} can be found in \cite{Steinbach-Wendland01}, \cite{KKLSY16, KPS07, Steinbach-Wendland01}, respectively. The proof of the self-adjointness-part of \ref{enum:NP:sa} is also found in aforementioned references, and that of the assertion on spectrum can be found in \cite{Chang-Lee08}.
\end{proof}

% with the convention $\jbk{a\varphi, b\psi}=a\overline{b}\jbk{\varphi, \psi}$ for all $a, b\in \Cbb$, $\varphi\in B^{1/2, 2}(\partial D)$ and $\psi\in B^{-1/2, 2}(\partial D)$

\subsection{Dirichlet-to-Neumann map}\label{subs:dtn}

Let $D\subset \Rbb^d$ ($d\geq 2$) be a bounded open set with the Lipschitz boundary and let $(s, 1/p)\in \Rscr_{\Ge_*( D)}$. For $\varphi\in B^{s, p}(\partial D)$, we define
\begin{equation}\label{eq:DtN}
    \Lambda_+[\varphi]=-\partial_\nv u \quad \text{on } \partial D,
\end{equation}
where $u: D^+\to \Rbb$ is the unique solution to the exterior Dirichlet problem
\begin{equation}\leqnomode
    \label{Dirichlet}\tag{D}
    \begin{cases}
        \lap u=0 & \text{in } D^+, \\
        u=\varphi & \text{on } \partial D, \\
        u(x)=-b\Gamma (x)+O(|x|^{-d+1}) & \text{as } |x|\to \infty \text{ for some } b\in \Rbb
    \end{cases}
\end{equation}
such that $u\in B^{s+1/p, p} (U\setminus \overline{D})\cap H^{s+1/p, p} (U\setminus \overline{D})$ for any open ball $U\subset \Rbb^d$ including $\overline{D}$, where $\Gamma (x)$ is the fundamental solution \eqref{eq:fundamental:solution}. The operator $\Lambda_+$ is called Dirichlet-to-Neumann map associated with the exterior domain $D^+$, or exterior DtN map for short. We put the minus sign in \eqref{eq:DtN} since the normal vector $\nu$ is outward to $D$ (inward to $D^+$).
Since the solution $u$ is given by $u(x)={\SL}[\SL^{-1} [\varphi]](x)$ for $x\in D^+$, the jump relation \eqref{eq:sl:jump} shows that
\begin{equation}\label{eq:DtN:potential}
\Lambda_+=-\left(\frac{1}{2}I+\NP^*\right) \SL^{-1} =-\SL^{-1} \left(\frac{1}{2}I+\NP\right).
\end{equation}
By the relation \eqref{eq:DtN:potential} and Theorem \ref{theo:layer:Lip}, we naturally consider $\Lambda_+$ as the bounded operators from $H^{1, p}(\partial D)$ into $L^p (\partial D)$ for $p\in (1, 2/(1-\Ge_*))$.

\begin{theo}
    \label{theo:inv}
    Let $D\subset \Rbb^d$ ($d\geq 2$) be a bounded open set with the Lipschitz boundary and $\gamma>0$. If $d \ge 3$, the operators 
    \[
        I+\gamma \Lambda_+: B^{s, p}(\partial D) \longrightarrow B^{s-1, p} (\partial  D) \quad ((s, 1/p)\in \Rscr (D))
    \]
    and
    \[
        I+\gamma \Lambda_+: H^{1, p}(\partial D) \longrightarrow L^p (\partial  D) \quad (p\in (1, 2/(1-\Ge_*)))
    \]
    are isomorphisms. 

    If $d=2$ and $\capac_D>1$, then the same conclusion holds for $p=2$.
\end{theo}

\begin{proof}
    Let $(s, 1/p)\in \Rscr (D)$ and $\gamma>0$ ($p=2$ if $d=2$). Since the imbedding $B^{s, p}(\p D)\to B^{s-1, p}(\p D)$ is compact, $\SL$ is compact on $B^{s, p}(\p D)$.  By Theorem \ref{theo:layer:Lip} \ref{enum:NP:Fredholm:Lip}, the operator $\gamma (1/2I+\NP)-\SL$ is a bounded Fredholm operator on $B^{s, p}(\partial D)$ of index zero. It then follows from the relation
    \begin{equation}\label{eq:rel1}
    I+\gamma \Lambda_+=-\SL^{-1}\left(\gamma\left(\frac{1}{2}I+\NP\right)-\SL\right),
    \end{equation}
    which is an immediate consequence of \eqref{eq:DtN:potential}, that $I+\gamma \Lambda_+$ is a Fredholm operator from $B^{s, p}(\partial D)$ into $B^{s-1, p}(\partial D)$ of index zero. Thus it suffices to prove that $(I+\gamma \Lambda_+)[\varphi]=0$ and $\varphi\in B^{s, p}(\partial D)$ imply $\varphi=0$.

    The simplest case is when $s=1/2$ and $p=2$. In fact, if $\varphi\in B^{1/2, 2}(\partial D)$ satisfies $(I+\gamma \Lambda_+)[\varphi]=0$, then, by Theorem \ref{theo:NP:energy}, it holds that
    \begin{align*}
        0&=-\jbk{\varphi, (I+\gamma \Lambda_+)[\varphi]}_{B^{1/2, 2}\text{-}B^{-1/2, 2}} \\
        &=\gamma \jbk{\left(\frac{1}{2}I+\NP\right)[\varphi], \varphi}_*+\| \varphi\|_{L^2}^2\geq \|\varphi\|_{L^2}^2.
    \end{align*}
    Thus $\varphi$ must be zero on $\partial D$.

    In order to prove the injectivity of $I+\gamma \Lambda_+: B^{s, p}(\partial D)\to B^{s-1, p}(\partial D)$, it suffices to prove that $\varphi\in L^p (\partial D)$ for some $p\in (1, 2]$ implies $\varphi=0$, since $B^{s, p}(\partial D) \subset L^p (\partial D)$ if $s>0$ and $L^p (\partial D)\subset L^2 (\partial D)$ if $p>2$. Thanks to the relation \eqref{eq:rel1}, we have
    \begin{equation}\label{eq:rel2}
        \left(\frac{1}{2}I+\NP\right)[\varphi]=\gamma^{-1}{\SL}[\varphi].
    \end{equation}
    By Theorem \ref{theo:layer:Lip} \ref{enum:sl:inv:Lip}, we have ${\SL}[\varphi]\in H^{1, p}(\partial D)$. Here, we emphasize that $\frac{1}{2}I+\NP$ may not be invertible if $D^+$ has a bounded component \cite{Verchota84}.

    We can take $(s_1, 1/p_1)\in \Rscr ( D)$ such that
    \[
        \begin{dcases}
            \frac{1}{2} \leq s_1 \leq 1, \\
            s_1-\frac{d-1}{p_1}\leq 1-\frac{d-1}{p}.
        \end{dcases}
    \]
    Then, by the Sobolev embedding theorem, we have $\varphi\in B^{s_1, p_1}(\partial D)$. Since
    \[
        \left(\frac{1}{2}I+\NP\right)[\varphi]=\gamma^{-1}{\SL}[\varphi]\in H^{1, p_1}(\partial D),
    \]
    the same argument proves that there exists $(s_2, 1/p_2)\in \Rscr ( D)$ such that
    \[
        \begin{dcases}
            \frac{1}{2}\leq s_2\leq 1, \\
            s_2-\frac{d-1}{p_2}\leq 1-\frac{d-1}{p_1}.
        \end{dcases}
    \]
    We can repeat this scheme until we have $\varphi\in B^{1/2, 2}(\partial  D)$, as shown in Figure \ref{fg:scheme}.
    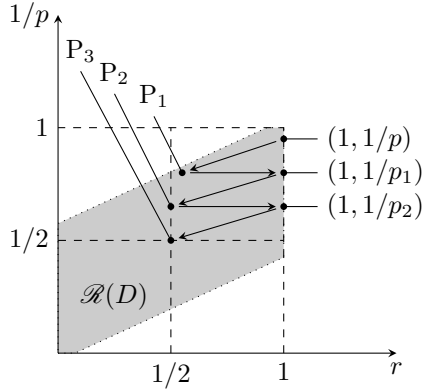
\begin{figure}[htb]
        \centering
        \begin{minipage}{0.5\textwidth}
        \tikzmath{\a=0.15;}
        \begin{tikzpicture}[scale=1.5]
            \filldraw[fill=black!20, dotted] (0, 0)--(\a, 0)--(2, {1-\a})--(2, 2)--({2-\a}, 2)--(0, {1+\a})--cycle;
            \draw[->, >=stealth] (0, 0)--(3, 0);
            \draw[->, >=stealth] (0, 0)--(0, 3);
            \node (x) at (3, 0) [below] {$r$};
            \node (y) at (0, 3) [left] {$1/p$};
            \node (10) at (2, 0) [below] {$1$};
            \node (01) at (0, 2) [left] {$1$};
            \node (1/20) at (1, 0) [below] {$1/2$};
            \node (01/2) at (0, 1) [left] {$1/2$};
            \draw[dashed] (2, 0)--(2, 2)--(0, 2);
            \draw[dashed] (1, 0)--(1, 2);
            \draw[dashed] (0, 1)--(2, 1);
            %\draw[thick] (2, {1-\a})--(2, 2);
            %\filldraw[fill=white] (2, {1-\a}) circle [radius=0.03];
            %\filldraw[fill=white] (2, 2) circle [radius=0.03];
            %\draw[thick] (0, {1+\a})--(0, 0);
            %\filldraw[fill=white] (0, {1+\a}) circle [radius=0.03];
            %\filldraw[fill=white] (0, 0) circle [radius=0.03];
            \coordinate (0) at (2, 1.9);
            \coordinate (1) at (1.1, 1.6);
            \coordinate (12) at (2, 1.6);
            \coordinate (2) at (1, {1.6-0.3*0.9/0.9});
            \coordinate (23) at (2, {1.6-0.3*0.9/0.9});
            \coordinate (3) at (1, {1.6-0.3*0.9/0.9*2});
            \fill[black] (0) circle [radius=0.03];
            \fill[black] (1) circle [radius=0.03];
            \fill[black] (12) circle [radius=0.03];
            \fill[black] (23) circle [radius=0.03];
            \fill[black] (3) circle [radius=0.03];
            \fill[black] (2) circle [radius=0.03];
            \draw[->, >=stealth, shorten <=3pt, shorten >=3pt] (0)--(1);
            \draw[->, >=stealth, shorten <=3pt, shorten >=3pt] (1)--(12);
            \draw[->, >=stealth, shorten <=3pt, shorten >=3pt] (12)--(2);
            \draw[->, >=stealth, shorten <=3pt, shorten >=3pt] (2)--(23);
            \draw[->, >=stealth, shorten <=3pt, shorten >=3pt] (23)--(3);
            \coordinate (rp0) at (2.3, 1.9);
            \coordinate (rp12) at (2.3, 1.6);
            \coordinate (rp23) at (2.3, {1.6-0.3*0.9/0.9});
            \coordinate (rp1) at (0.85, 2.1);
            \coordinate (rp2) at (0.5, 2.3);
            \coordinate (rp3) at (0.2, 2.5);
            \node (rp0n) at (rp0) [right] {$(1, 1/p)$};
            \node (rp12n) at (rp12) [right] {$(1, 1/p_1)$};
            \node (rp23n) at (rp23) [right] {$(1, 1/p_2)$};
            \node (rp1n) at (rp1) [above] {$\mathrm{P}_1$};
            \node (rp2n) at (rp2) [above] {$\mathrm{P}_2$};
            \node (rp3n) at (rp3) [above] {$\mathrm{P}_3$};
            \draw (0)--(rp0);
            \draw (12)--(rp12);
            \draw (23)--(rp23);
            \draw (1)--(rp1);
            \draw (2)--(rp2);
            \draw (3)--(rp3);
            \node (R) at (0.5, 0.5) {$\Rscr ( D)$};
        \end{tikzpicture}
    \end{minipage}
    \begin{minipage}{0.4\textwidth}
        
        \caption{Construction of the sequence $\mathrm{P}_j= (s_j, 1/p_j)$ when $p\in (1, 2]$.}
        \label{fg:scheme}
    \end{minipage}
    \end{figure}

    In conclusion, we can construct a sequence $\{(s_j, 1/p_j)\}_{j=1}^N$ such that $s_N\geq 1/2$, $p_N\geq 2$ and $p_{j+1}\geq p_j+\delta$ for some $\delta>0$ depending only on $s$, $p$, and $ D$ by the above procedure. Hence we obtain $\varphi=0$.

    The isomorphic property of $I+\gamma \Lambda_+: H^{1, p}(\partial D) \longrightarrow L^p (\partial  D)$ for $p\in (1, 2/(1-\Ge_*))$ is proved similarly by using Theorem \ref{theo:layer:Lip} \ref{enum:sl:inv:Lip:end}--\ref{enum:NP:inv:Lip:end}. 
    \end{proof}

    As a corollary of Theorem \ref{theo:inv}, we obtain the closedness of the exterior DtN map as an unbounded operator.

    \begin{coro}
        \label{coro:DtN:closed}
        For any $p\in (1, 2/(1-\Ge_*))$, the operator $\Lambda_+$ with domain $H^{1, p}(\partial D)$ is a closed operator on $L^p (\partial D)$. 
    \end{coro}

    \begin{proof}
        If $\varphi_j\in H^{1, p}(\partial D)$ converges to $\varphi$ in $L^p (\partial D)$ and $\Lambda_+[\varphi_j]$ converges to $\psi$ in $L^p (\partial D)$, then, since $\{(I+\Lambda_+)[\varphi_j]\}_{j=1}^\infty$ forms a Cauchy sequence in $L^p (\partial D)$, the sequence $\{\varphi_j\}_{j=1}^\infty$ is Cauchy in $H^{1, p}(\partial D)$ by Theorem \ref{theo:inv}. Thus $\varphi_j\to \varphi$ in $H^{1, p}(\partial D)$ and $(I+\Lambda_+)[\varphi]=\varphi+\psi$, which implies $\Lambda_+[\varphi]=\psi$.
    \end{proof}

    \subsection{Uniform resolvent estimate on \texorpdfstring{$L^p$}{Lp}-spaces}

Corollary \ref{coro:DtN:closed} implies the closedness of $I+\gamma \Lambda_+$ on $L^p (\partial D)$ for each $\gamma>0$. Thus the resolvent $(I+\gamma \Lambda_+)^{-1}$ is bounded on $L^p (\partial D)$ for $\gamma>0$.
Furthermore, the following uniform resolvent estimate is valid:

\begin{theo}
    \label{theo:unif:res:Lip}
    Let $D$ be a bounded open set with the Lipschitz boundary in $\Rbb^d$ with $d\geq 2$. 
    
    If $d \geq 3$, let $\Ge_*=\Ge_*(D)$ be the number defined in Theorem \ref{theo:layer:Lip}. For each $p\in (1, 2/(1-\Ge_*))$, it holds that
    \begin{equation}\label{eq:unif:res:Lip}
        \| (I+\gamma \Lambda_+)^{-1}\|_{L^p (\partial  D)\to L^p (\partial  D)}\leq 1
    \end{equation}
    for all $\gamma\geq 0$. In particular,
    \[
        \lim_{\gamma\to +0} \| (I+\gamma \Lambda_+)^{-1}[\varphi]-\varphi\|_{L^p (\partial D)}=0
    \]
    for any $\varphi\in L^p (\partial D)$.

    If $d=2$ and $\capac_D>1$, the same conclusion holds for $p=2$.
\end{theo}

We will prove Theorem \ref{theo:unif:res:Lip} by using a theory of semigroups acting on $L^p$-spaces. This idea naturally arises from the following well-known theorem \cite{Yosida74}: if a $C_0$-semigroup $\{ \Tcal (t)\}_{t\geq 0}$ acting on a Banach space $\Xcal$ satisfies $\|\Tcal (t)\|\leq M\e^{\omega t}$ for all $t\geq 0$ with the constants $M>0$ and $\omega\geq 0$ independent of $t\geq 0$, then the infinitesimal generator
    \begin{equation}\label{eq:infinitesimal}
        \Acal [\varphi]:=\lim_{t\to +0} \frac{\Tcal (t)[\varphi]-\varphi}{t} \quad \text{in } \Xcal,
    \end{equation}
whose domain is the linear subspace of $\Xcal$ consisting of all $\varphi\in \Xcal$ such that the right hand side of \eqref{eq:infinitesimal} exists, satisfies the following properties. 
\begin{itemize}
    \item $\Acal$ is a densely defined closed operator on $\Xcal$. 
    \item $\lambda I-\Acal$ is invertible for all $\lambda>\omega$ and $\|(\lambda I-\Acal)^{-n}\|\leq M(\lambda - \omega)^{-n}$ for all $\lambda>\omega$ and $n\in \Nbb$.
\end{itemize}
In fact, an analogue of Theorem \ref{theo:unif:res:Lip} for \textit{interior} DtN maps is essentially proved by \cite{terElst-Ouhabaz14} in terms of the associated semigroup (including the case when $d=2$). Our proof for the \textit{exterior} DtN maps follows that of \cite{terElst-Ouhabaz14} except that we require an additional argument in order to employ the Green theorem in exterior domains. In this step, however, we require the assumption $d\geq 3$. 

In what follows, we recall a few facts from semigroup theory for our use. 

\begin{defi}
    Let $X$ be a measure space and $L^2 (X)$ be the (real or complex) $L^2$-space on $X$. A bounded linear operator $\Tcal$ on $L^2 (X)$ is said to be $L^\infty$-contractive if $\varphi\in L^2 (X)\cap L^\infty (X)$, then $\Tcal [\varphi]\in L^\infty (X)$ and $\|\Tcal [\varphi]\|_{L^\infty}\leq \|\varphi\|_{L^\infty}$.

    An operator semigroup $\{ \Ucal (t)\}_{t\geq 0}$ (not necessarily $C_0$ or holomorphic in $t$) is said to be $L^\infty$-contractive if $\Ucal (t)$ is $L^\infty$-contractive for each $t\geq 0$.
\end{defi}

\begin{lemm}\label{lemm:contractive:strong}
    Let $X$ be a finite measure space. If a $C_0$-semigroup $\{\Ucal (t)\}_{t\geq 0}$ of self-adjoint operators on $L^2 (X)$ is $L^\infty$-contractive and satisfies $\|\Ucal (t)\|_{L^2\to L^2}\leq 1$ for all $t\geq 0$, then $\{\Ucal (t)\}_{t\geq 0}$ is a $C_0$-semigroup on $L^p (X)$ for any $p\in [1, \infty]$ and satisfies $\|\Ucal (t)\|_{L^p\to L^p}\leq 1$ for all $p\in [1, \infty]$ and $t\geq 0$. 
\end{lemm}

\begin{proof}
    By interpolation, we have $\|\Ucal (t)\|_{L^p\to L^p}\leq 1$ for $p\in [2, \infty]$ and $t\geq 0$. Since $\Ucal (t)$ is self-adjoint on $L^2 (X)$, a duality argument extends the estimate 
    \begin{equation}\label{eq:semigroup:estimate:g}
        \|\Ucal (t)\|_{L^p\to L^p}\leq 1
    \end{equation}
    for $p\in [1, \infty]$ and $t\geq 0$. Thus $\{\Ucal (t)\}_{t\geq 0}$ is an operator semigroup on $L^p (\partial D)$: $\Ucal (0)=I$ and $\Ucal (t)\Ucal (s)=\Ucal (t+s)$ on $L^p (\partial D)$ for any $t, s\geq 0$. 
    
    Next, we prove the strong continuity of the semigroup, i.e., $\|\Ucal (t)[\varphi]-\varphi\|_{L^p}\to 0$ as $t\to +0$ for any $\varphi\in L^p (X)$. For $p\in [1, 2]$, we can prove this convergence by approximating $L^p(X)$-function by $L^2(X)$-functions in the $L^p$-norm topology. When $p\in [2, \infty]$, we first prove that $\Ucal (t)[\varphi]\to \varphi$ as $t\to +0$ weakly in $L^p (X)$. We begin with the relation 
    \[
        \jbk{\Ucal (t)[\varphi], \psi}=\jbk{\varphi, \Ucal (t)[\psi]}
    \]
    for any $\varphi, \psi\in L^2 (X)$. By the inequality \eqref{eq:semigroup:estimate:g}, the above relation also holds for $\varphi\in L^p (X)$ and $\psi\in L^q (X)$, where $q=p/(p-1)$ is the H\"older conjugate of $p$. Now let $\varphi\in L^p (X)$ and $\psi\in L^q (X)$. Since $q=p/(p-1)\in [1, 2]$ by the assumption that $p\in [2, \infty]$, we can employ the fact that $\Ucal (t)[\psi]\to \psi$ strongly in $L^q(X)$ to obtain $\jbk{\Ucal (t)[\varphi], \psi}\to \jbk{\varphi, \psi}$ as $t\to +0$. Since $L^q (X)$ corresponds to the continuous dual of $L^p (X)$ thanks to the assumption $p\in [2, \infty]\, (\subset (1, \infty])$, we obtain the weak convergence $\Ucal (t)[\varphi]\to \varphi$ in $L^p (X)$ as $t\to +0$. Now we invoke a theorem in \cite[p.233 in Chapter IX]{Yosida74}: an operator semigroup $\{\Tcal (t)\}_{t\geq 0}$ on a Banach space $\Xcal$ is a $C_0$-semigroup if and only if $\lim_{t\to +0}\Tcal (t)[x]=x$ weakly in $\Xcal$ for each $x\in \Xcal$. This theorem proves that $\{\Ucal (t)\}_{t\geq 0}$ is a $C_0$-semigroup on $L^p (X)$ for any $p\in [1, \infty]$. 
\end{proof}

% \begin{defi}
%     Let $(X, \mu)$ be a measure space and $L^2 (X, \mu)$ be the \textit{real} $L^2$-space associated with $(X, \mu)$. A bounded linear operator $\Tcal$ on $L^2 (X, \mu)$ is said to be sub-Markovian if the following conditions are satisfied.
%     \begin{enumerate}[label=\textnormal{(\roman*)}]
%         \item (Positivity) If $\varphi\geq 0$, then $\Tcal [\varphi]\geq 0$.
%         \item ($L^\infty$-contractivity) If $\varphi\in L^2 (X, \mu)\cap L^\infty (X, \mu)$, then $\Tcal [\varphi]\in L^\infty (X, \mu)$ and $\|\Tcal [\varphi]\|_{L^\infty}\leq \|\varphi\|_{L^\infty}$.
%     \end{enumerate}

%     An operator semigroup $\{ \Tcal (t)\}_{t\geq 0}$ (not necessarily $C_0$ or holomorphic) is said to be sub-Markovian if $\Tcal (t)$ is sub-Markovian for each $t\geq 0$.
% \end{defi}

\begin{defi}\label{defi:coercive}
    Let $\Hcal$ be a \textit{real} Hilbert space (for simplicity). A bilinear form $b: \dom (b)\times \dom (b)\to \Rbb$ with the domain  $\dom (b)\subset \Hcal$ is a coercive closed form if the following statements are valid.
\begin{enumerate}[label=\textnormal{(\alph*)}]
    \item \label{enum:coercive:dense}$\dom (b)$ is dense in $\Hcal$.
    \item \label{enum:coercive:nonneg}(Non-negativity) $b(\varphi, \varphi)\geq 0$ for any $\varphi\in \dom (b)$.
    \item \label{enum:coercive:closed}(Closedness) Define $\|\varphi\|_b:=(b(\varphi, \varphi)+\|\varphi\|_\Hcal^2)^{1/2}$ for $\varphi\in \dom (b)$. Then $(\dom (b), \|\cdot\|_b)$ is complete.
    \item \label{enum:coercive:wsc}(Weak sector condition) There exists $C>0$ such that $|b(\varphi, \psi)|\leq C\|\varphi\|_b \|\psi\|_b$ for all $\varphi, \psi\in \dom (b)$.
\end{enumerate}
\end{defi}

It is well-known that, if $b$ is a coercive closed form, then the linear operator
\[
    \begin{cases}
        \dom (\Bcal)=\{ \varphi\in \dom (b) \mid \exists \psi\in \Hcal \text{ s.t. } \forall \eta\in \Hcal, \ b(\varphi, \eta)=\jbk{\psi, \eta}_\Hcal\}, \\
        \Bcal [\varphi]=\psi
    \end{cases}
\]
is a densely defined closed linear operator on $\Hcal$ such that the operator $-\Bcal$ generates a $C_0$-semigroup $\{ \e^{-t \Bcal}\}_{t\geq 0}$ (actually, holomorphic semigroup) on $\Hcal$. Furthermore, if $b$ is symmetric: $b(\varphi, \psi)=b(\psi, \varphi)$ for all $\varphi, \psi\in \dom (b)$, then $\Bcal$ is self-adjoint. (See \cite[Chapter IX]{Yosida74}.)

For a real-valued measurable function $\varphi: X\to \Rbb$, we define $\varphi^+(x):=\max \{ \varphi (x), 0\}$ and $(\varphi \wedge 1)(x):=\min \{\varphi (x), 1\}$.

The following Beurling--Deny-type criterion is known:

\begin{theo}[{\cite[Corollary 2.7]{Ouhabaz96}}]
    \label{theo:Dirichlet:form}
    Let $\Bcal$ be the linear operator associated with the coercive closed form $b$ on real $L^2 (X)$ for some $\sigma$-finite measure space $X$. Then the following two statements are equivalent.
    \begin{enumerate}[label=\textnormal{(\roman*)}]
        \item \label{enum:subMarkov}$\{\e^{-t\Bcal}\}_{t\geq 0}$ is $L^\infty$-contractive and $\e^{-t\Bcal}[\varphi]\geq 0$ if $\varphi\geq 0$.
        \item \label{enum:BD}$\dom (b)$ is closed under the operations $\varphi\mapsto \varphi^+, \varphi\wedge 1$, and
        \begin{equation}\label{eq:BD}
            b(\varphi, \varphi-(\varphi\wedge 1)^+)\geq 0
        \end{equation}
        for all $\varphi\in \dom (b)$.
    \end{enumerate}
\end{theo}
Operator semigroups satisfying \ref{enum:subMarkov} are known as sub-Markovian semigroups. 

We use Theorem \ref{theo:Dirichlet:form} to prove Theorem \ref{theo:unif:res:Lip}. We begin with the following lemma.

\begin{lemm}
    \label{lemm:DtN:bilinear}
    We define the bilinear form $b(\varphi, \psi):=\jbk{\Lambda_+[\varphi], \psi}$ for $\varphi, \psi\in \dom (b):=B^{1/2, 2}(\partial D)$. Then $b$ is a coercive closed form on $L^2 (\partial D)$. 
\end{lemm}

\begin{proof}
    We prove properties \ref{enum:coercive:dense}--\ref{enum:coercive:wsc} in Definition \ref{defi:coercive}. \ref{enum:coercive:dense} is immediate from the density of $B^{1/2, 2}(\partial D)$ in $L^2 (\partial D)$. By \eqref{eq:DtN:potential} and Theorem \ref{theo:NP:energy}, we obtain 
    \[
        b(\varphi, \varphi)=-\jbk{\left(\frac{1}{2}I+\NP^*\right)\SL^{-1}[\varphi], \varphi}=\jbk{\left(\frac{1}{2}I+\NP\right)[\varphi], \varphi}_*\geq 0
    \]
    for any $\varphi\in B^{1/2, 2}(\partial D)$. Thus \ref{enum:coercive:nonneg} is proved. 
    
    To prove \ref{enum:coercive:closed} and \ref{enum:coercive:wsc}, it suffices to prove that $\|\cdot\|_b$ is equivalent to the Besov (or Sobolev) norm on $B^{1/2, 2}(\partial D)$. It is easy to prove that 
    \[
        \|\varphi\|_b^2=\jbk{(I+\Lambda_+)[\varphi], \varphi}\leq \|(I+\Lambda_+)[\varphi]\|_{B^{-1/2, 2}}\|\varphi\|_{B^{1/2, 2}}\lesssim \|\varphi\|_{B^{1/2,2}}^2
    \]
    for any $\varphi\in B^{1/2, 2}(\partial D)$. We prove the reverse estimate by contradiction argument. Suppose that there exists a sequence $\{\varphi_j\}_{j=1}^\infty\subset B^{1/2, 2}(\partial D)$ such that $\|\varphi_j\|_*=1$ for all $j$ and $\|\varphi_j\|_b\to 0$ as $j\to \infty$. Here, $\|\cdot\|_*$ is the norm appeared in Theorem \ref{theo:NP:energy} \ref{enum:inner:product}, which is equivalent to the Besov norm on $B^{1/2, 2}(\partial D)$. We first note that the following orthogonal decomposition holds:
    \[
        B^{1/2, 2}(\partial D)=\Ker \left(\frac{1}{2}I+\NP\right)\oplus \Ker \left(\frac{1}{2}I+\NP\right)^\perp=:\Ncal \oplus \Mcal
    \]
    with respect to the inner product $\jbk{\cdot, \cdot}_*$ on $B^{1/2, 2}(\partial D)$. We decompose $\varphi_j=\varphi^0_j+\psi_j\in \Ncal \oplus \Mcal$ accordingly. Furthermore, since $1/2I+\NP$ is self-adjoint and Fredholm on $B^{1/2, 2}(\partial D)$ (Theorem \ref{theo:layer:Lip} \ref{enum:NP:Fredholm:Lip}), it is invertible on $\Mcal$. In particular, there exists $\delta>0$ such that 
    \[
        \jbk{\left(\frac{1}{2}I+\NP\right)[\psi], \psi}_*\geq \delta \|\psi\|_*^2
    \]
    for all $\psi\in \Mcal$. Thus we have 
    \begin{align*}
        \|\varphi_j\|_{L^2}^2+\|\psi_j\|_*^2
        &\lesssim \|\varphi_j\|_{L^2}^2+\jbk{\left(\frac{1}{2}I+\NP\right)[\psi_j], \psi_j}_* \\
        &=\|\varphi_j\|_{L^2}^2+\jbk{\left(\frac{1}{2}I+\NP\right)[\varphi_j], \varphi_j}_* \\
        &=\|\varphi_j\|_b^2\to 0 \quad (j\to \infty).
    \end{align*}
    Hence $\lim_{j\to \infty}\varphi_j=0$ in $L^2 (\partial D)$ and $\lim_{j\to \infty} \psi_j=0$ in $B^{1/2, 2}(\partial D)$. Then $\varphi^0_j=\varphi_j-\psi_j\to 0$ ($j\to \infty$) in $L^2 (\partial D)$. Since $\Ncal=\Ker (1/2I+\NP)$ is finite-dimensional, the norms $\|\cdot\|_{L^2}$ and $\|\cdot\|_*$ are equivalent on $\Ncal$. Hence, the convergence $\lim_{j\to \infty} \varphi^0_j=0$ holds in the strong topology of $B^{1/2, 2}(\partial D)$. Thus we have $\varphi_j\to 0$ in $B^{1/2, 2}(\partial D)$, which contradicts to $\|\varphi_j\|_*=1$. 
\end{proof}

The next step is to prove the $L^\infty$-contractivity of the semigroup generated by $-\Lambda_+$. 

\begin{lemm}
    \label{lemm:DtN:contractive}
    The semigroup $\{\Ucal (t)=\e^{-t\Lambda_+}\}_{t\geq 0}$ generated by $-\Lambda_+: B^{1/2, 2}(\partial D)\to B^{-1/2, 2}(\partial D)$ is $L^\infty$-contractive. (In fact, it is sub-Markovian.)
\end{lemm}

\begin{proof}
    It suffices to prove that $b$ satisfies the condition \ref{enum:BD} in Theorem \ref{theo:Dirichlet:form}. We follow the argument in \cite[Theorem 2.3]{Ouhabaz96}. We first prove that $\varphi\in \dom (b)=B^{1/2, 2}(\partial D)$ implies $\varphi^+, \varphi\wedge 1\in B^{1/2, 2}(\partial D)$. Take an extension operator $\Ecal: B^{1/2, 2}(\partial D)\to H^{1, 2}(D^+)$, that is, a bounded operator satisfying $\Ecal[\cdot]|_+=I$, such that $\supp \Ecal [\varphi]$ is contained in some compact set independent of $\varphi\in B^{1/2, 2}(\partial D)$. Then $(\Ecal [\varphi])^+\in H^{1, 2}(D^+)$ and 
    \begin{equation}\label{eq:truncate:der}
        \nabla (\Ecal [\varphi])^+ (x)=
        \begin{cases}
            \nabla \Ecal [\varphi](x) & \text{if } \Ecal [\varphi](x)>0, \\
            0 & \text{if } \Ecal [\varphi](x)\leq 0.
        \end{cases}
    \end{equation}
    (See \cite[Lemma 7.6]{Gilbarg-Trudinger01}.) Thus, by the trace theorem, we have $\varphi^+=(\Ecal [\varphi])^+|_+\in B^{1/2, 2}(\partial D)$. $\varphi\wedge 1\in B^{1/2, 2}(\partial D)$ is similarly proved by using $\varphi\wedge 1=\varphi-(\varphi-1)^+$. 
    
    We prove \eqref{eq:BD} next. Since
    \[
        b(\varphi, \varphi-(\varphi\wedge 1)^+)=b((\varphi\wedge 1)^+, \varphi-(\varphi\wedge 1)^+)+b(\varphi-(\varphi\wedge 1)^+, \varphi-(\varphi\wedge 1)^+),
    \]
    it suffices to prove that
    \[
        b((\varphi\wedge 1)^+, \varphi-(\varphi\wedge 1)^+)\geq 0
    \]
    for all $\varphi\in B^{1/2, 2}(\partial D)$.

    Let $u$ be the unique harmonic function on $D^+$ such that $u|_+=\varphi$ and $u=O(|x|^{-d+2})$ as $|x|\to \infty$. Analogously, let $u_1$ be the unique harmonic function on $D^+$ such that $u_1|_+=(\varphi\wedge 1)^+$ and $u_1=O(|x|^{-d+2})$ as $|x|\to \infty$. We define $u_0=(u\wedge 1)^+-u_1$, which satisfies $u_0|_+=0$ and $|\nabla u_0|=O(|x|^{-d+1})$ as $|x|\to \infty$ by \eqref{eq:truncate:der}.

    Since $d\geq 3$ and $|\nabla u|, |\nabla u_1|=O(|x|^{-d+1})$ as $|x|\to \infty$, we can apply the Green theorem in the exterior domain to obtain
    \begin{align*}
        &b ((\varphi\wedge 1)^+, \varphi-(\varphi\wedge 1)^+) \\
        &=b(u_1|_+, (u-u_1)|_+)
        =\jbk{\nabla u_1, \nabla (u-u_1)}_{L^2 (D^+)} \\
        &=\jbk{\nabla (u_0+u_1), \nabla (u-u_1)}_{L^2 (D^+)}-\jbk{\nabla u_0, \nabla (u-u_1)}_{L^2 (D^+)} \\
        &=\underbrace{\jbk{\nabla (u\wedge 1)^+, \nabla (u-(u\wedge 1)^+)}_{L^2 (D^+)}}_{=0}+\|\nabla u_0\|_{L^2 (D^+)}^2+\jbk{\nabla u_1, \nabla u_0}_{L^2(D^+)} \\
        &\quad -\jbk{\nabla u_0, \nabla (u-u_1)}_{L^2 (D^+)} \\
        &=\|\nabla u_0\|_{L^2 (D^+)}^2+\jbk{\nabla u_0, \nabla (-u+2u_1)}_{L^2(D^+)}.
    \end{align*}
    Here we used the fact that the supports of $\nabla (u\wedge 1)^+$ and $\nabla (u-(u\wedge 1)^+)$ are disjoint due to \eqref{eq:truncate:der}.

    We once again employ the Green theorem in the exterior domain to obtain
    \begin{align*}
        \jbk{\nabla u_0, \nabla (-u+2u_1)}_{L^2 (D^+)}
        &=-\jbk{ u_0|_+, \partial_\nv (-u+2u_0)|_+}_{\partial D}=0.
    \end{align*}
    Thus we obtain
    \[
        b ((\varphi\wedge 1)^+, \varphi-(\varphi\wedge 1)^+)
        =\|\nabla u_0\|_{L^2 (D^+)}^2\geq 0
    \]
    as desired.
\end{proof}

\begin{proof}[Proof of Theorem \ref{theo:unif:res:Lip}]
    By Lemma \ref{lemm:DtN:contractive}, the operator semigroup $\{\Ucal (t)=\e^{-t\Lambda_+}\}_{t\geq 0}$ is $L^\infty$-contractive. Furthermore, $\Ucal(t)$ is self-adjoint on $L^2 (\partial D)$ since $\Lambda_+$ is self-adjoint on $L^2 (\partial D)$. Thus, by Lemma \ref{lemm:contractive:strong}, $\{\Ucal (t)\}_{t\geq 0}$ is a $C_0$-semigroup on $L^p (\partial D)$ satisfying $\|\Ucal (t)\|_{L^p\to L^p}\leq 1$ for all $p\in [1, \infty]$ and $t\geq 0$. Then we obtain the infinitesimal generator
    \begin{equation*}
        %\label{eq:infinitesimal}
        \Acal_p [\varphi]:=\lim_{t\to +0} \frac{\Ucal (t)[\varphi]-\varphi}{t} \quad \text{in } L^p (\partial D).
    \end{equation*}
    Since $\|\Ucal (t)\|_{L^p\to L^p}\leq 1$, $\Acal_p$ satisfies the estimate $\|(\lambda -\Acal_p)^{-1}\|_{L^p \to L^p }\leq \lambda^{-1}$ for any $\lambda>0$. We substitute $\lambda=\gamma^{-1}$ and obtain
    \begin{equation}\label{eq:uniform:Ap}
        \|(I -\gamma\Acal_p)^{-1}\|_{L^p\to L^p}\leq 1
    \end{equation}
    for any $\gamma>0$. 

    According to \eqref{eq:uniform:Ap}, it suffices to prove that $(I-\gamma \Acal_p)^{-1}[\varphi]=(I+\gamma \Lambda_+)^{-1}[\varphi]$ for any $p\in (1, 2/(1-\Ge_*))$ and $\varphi\in L^p (\partial D)$ in order to complete the proof. The case when $p=2$ is obvious from the definition of infinitesimal generator of $C_0$-semigroup. 
    
    We first consider the extension to the case when $p\in (1, 2]$. We prove that the domain of $\Acal_p$ contains $H^{1, p}(\partial D)$ and is identical to $-\Lambda_+$ there. Let $\psi\in H^{1, p}(\partial D)$ and approximate it by the sequence $\psi_j\in H^{1, 2}(\partial D)$ in $H^{1, p}(\partial D)$-topology. Then, since $\Acal_p [\psi_j]=-\Lambda_+ [\psi_j]\to -\Lambda_+[\psi]$ in $L^p (\partial D)$ as $j\to \infty$ by Theorem \ref{theo:layer:Lip} and \eqref{eq:DtN:potential}. Thus, the closedness of $\Acal_p$ on $L^p (\partial D)$ implies that $\psi$ lies in the domain of $\Acal_p$ and $\Acal_p [\psi]=-\Lambda_+ [\psi]$. Now, we set $\psi=(I+\gamma \Lambda_+)^{-1}[\varphi]\in H^{1, p}(\partial D)$ for $\varphi\in L^p (\partial D)$. Then $\psi\in H^{1, p}(\partial D)$ by Theorem \ref{theo:inv} and we have 
    \[
        (I-\gamma \Acal_p)(I+\gamma \Lambda_+)^{-1}[\varphi]=(I+\gamma \Lambda_+)(I+\gamma \Lambda_+)^{-1}[\varphi]=\varphi. 
    \]
    Hence we have $(I-\gamma \Acal_p)^{-1}[\varphi]=(I+\gamma \Lambda_+)^{-1}[\varphi]$. 

    Next, we prove the case when $p\in (2, 2/(1-\Ge_*))$. Let $\varphi\in L^p (\partial D)$ lie in the domain of $\Acal_p$. Then, by definition of infinitesimal generator and $\Acal_2[\psi]=-\Lambda_+[\psi]$ for $\psi\in H^{1, 2}(\partial D)$, we obtain 
    \begin{align*}
        \jbk{\Acal_p [\varphi], \psi}&=\lim_{t\to +0}\frac{\jbk{\Ucal (t)[\varphi], \psi}-\jbk{\varphi, \psi}}{t}=\lim_{t\to +0}\frac{\jbk{\varphi, \Ucal (t)[\psi]}-\jbk{\varphi, \psi}}{t} \\
        &=-\jbk{\varphi, \Lambda_+[\psi]}. 
    \end{align*}
    Thus,  
    \begin{align*}
        \jbk{(I-\gamma \Acal_p)[\varphi], \psi}=\jbk{\varphi, (I+\gamma \Lambda_+)[\psi]}
    \end{align*}
    for any $\psi\in H^{1, 2}(\partial D)$. Now we replace $\psi\in H^{1, 2}(\partial D)$ to $(I+\gamma \Lambda_+)^{-1}[\psi]$ with $\psi\in L^2 (\partial D)$. Then we have 
    \begin{align*}
        \jbk{(I+\gamma \Lambda_+)^{-1}(I-\gamma \Acal_p)[\varphi], \psi}
        &=\jbk{(I-\gamma \Acal_p)[\varphi], (I+\gamma \Lambda_+)^{-1}[\psi]} \\
        &=\jbk{\varphi, (I+\gamma \Lambda_+)(I+\gamma \Lambda_+)^{-1}[\psi]} \\
        &=\jbk{\varphi, \psi}.
    \end{align*}
    Since $\psi\in L^2 (\partial D)$ is arbitrary and $L^2 (\partial D)$ is dense in $L^{p/(p-1)}(\partial D)$, we have 
    \[
        (I+\gamma \Lambda_+)^{-1}(I-\gamma \Acal_p)[\varphi]=\varphi
    \]
    for any $\varphi\in L^p (\partial D)$ which lies in the domain of $\Acal_p$. Thus $(I-\gamma \Acal_p)^{-1}=(I+\gamma \Lambda_+)^{-1}$ on $L^p (\partial D)$. This completes the proof. 
\end{proof}

\subsection{Resistive capacitance matrix}\label{subs:capacitance}

For any subset $S\subset \Rbb^d$, $1_S$ denotes the indicator function of $S$. For each connected component $D_j$ of $D$, we set
\beq\label{ej}
    e_j:=-\SL^{-1}[1_{\partial D_j}]\in B^{-1/2, 2}(\partial D)
\eeq
on $\partial D$. By virtue of Theorem \ref{theo:inv}, we can define for $\gamma \ge 0$ the matrix
\begin{equation}\label{eq:matrix}
    \Capac^\gamma_D=(C^\gamma_{ij})_{i, j=1}^N:=\left(\int_{\partial D_i}(I+\gamma \Lambda_+)^{-1}[e_j]\, \df \sigma\right)_{i, j=1}^N.
\end{equation}
We call $\Capac_D^\gamma$ the resistive capacitance matrix since $\Capac_D^0$ is known as a capacitance matrix (see \cite{Smolic-Klajn21} and references therein) and $\gamma$ represents the resistance as mentioned before. It is known that if $D$ is three- or higher-dimensional, then $\Capac_D^0$ is always positive; in two dimensions, it is (well-defined and) positive if and only if $\capac_D>1$ (see \cite[Lemma 1]{Feppon-Ammari22}). It is helpful to remind that the condition $\capac_D>1$ is a standing hypothesis of this paper in two dimensions.

By virtue of Theorem \ref{theo:unif:res:Lip}, we obtain the following proposition.
\begin{prop}
    \label{prop:rcap:conti}
    Let $D$ be a bounded open set with the Lipschitz boundary in $\Rbb^d$ $(d\geq 2)$. We further assume $\capac_D>1$ when $d=2$. Then, $|C^\gamma_{ij}-C^0_{ij}|\to 0$ as $\gamma\to +0$.
\end{prop}

We now prove invertibility of $\Capac^\gamma_D$ for $\gamma >0$ which plays a crucial role in the proof of Theorem \ref{theo:existence}. For that purpose we consider the following exterior Robin boundary value problem with $\gamma >0$: for any $f\in B^{-1/2, 2}(\partial D)$,
    \begin{equation}\leqnomode
    \label{R}\tag{R}
    \begin{dcases}
    \lap u=0 & \text{in } D^+, \\
    u|_+ - \Gg \p_\Gv u|_+=f & \text{on } \partial D,  \\
    u (x)=-b\Gamma (x)+O(|x|^{-d+1}) & \text{as } |x|\to \infty \text{ for some } b\in \Rbb.
    \end{dcases}
    \end{equation}
    In view of Theorem \ref{theo:inv}, the problem \eqref{R} is equivalent to the exterior Dirichlet problem \eqref{Dirichlet} with $\varphi=(I+\gamma \Lambda_+)^{-1}[f] \in B^{1/2, 2}(\p D)$, and hence the problem \eqref{R} admits a unique solution such that $u|_{U\setminus \overline{D}}\in H^{1, 2} (U\setminus \overline{D})$ for any open ball $U\subset \Rbb^d$ including $\overline{D}$.

    Let $\varphi_0\in B^{-1/2, 2}(\partial D)$ be the unique function such that ${\SL}[\varphi_0]$ is constant on $\partial D$ and $\int_{\partial D}\varphi_0\, \df \sigma=1$ (see for example \cite[p.39]{Ammari-Kang07} for existence of such a function). Then the solution $u$ to \eqref{R} admits the representation 
    \begin{equation}\label{eq:Robin:sol:explicit}
        u={\SL}[-b\varphi_0+\psi] \quad\text{in } D^+,
    \end{equation}
    with some $\psi \in B^{-1/2, 2} (\p D)$ satisfying $\int_{\partial D}\psi\, \df \sigma=0$. In fact, we just set $\psi: = \SL^{-1}[(u+ b{\SL}[\varphi_0])|_+]$. Then, we have
    $$
    {\SL}[\psi]= u + b {\SL}[\varphi_0] \quad\text{in } D^+
    $$
    and hence ${\SL}[\psi] = O(|x|^{-d+1})$ as $|x|\to \infty$, or equivalently $\int_{\partial D}\psi\, \df \sigma=0$.

\begin{theo}\label{positive:definite}
The resistive capacitance matrix $\Capac^\gamma_D$ is real, symmetric and positive-definite for all $\Gg \geq 0$ and $d \geq 2$. 	
\end{theo}

\begin{proof}
    It is enough to deal with the case when $\Gg >0$. For $j=1,2, \ldots, N$, let $u_j$ be the unique solution to \eqref{R} for $f = 1_{\p D_j}$. We set $\varphi_j:=\partial_\nv u_j |_+\in B^{-1/2, 2}(\partial D)$ so that $\Lambda_+[u_j|_+]=-\varphi_j$. Let $e_j$ be the function defined in \eqref{ej}. Since ${\SL}[e_j]=-1_{\partial D_j}$ as a function on $\p D$, ${\SL}[e_j]$ as a function on $\Rbb^d$ is constant on each connected component $D_j$ of $D$. Thus, we infer from the jump formula \eqref{eq:sl:jump} that $\NP^*[e_j]=(1/2)e_j$. We then infer from the representation \eqref{eq:DtN:potential} of $\Lambda_+$ that $\Lambda_+[1_{\partial D_j}]=e_j$. In addition, by the second condition of \eqref{R}, we have $\gamma \varphi_j=u_j|_+-1_{\partial D_j}\in B^{1/2, 2}(\partial D)$. Since $\gamma >0$, we have $\Lambda_+[\varphi_j]\in B^{-1/2, 2}(\partial D)$ and
    \begin{equation}\label{eq:phi:DtN:higher}
        (I+\gamma \Lambda_+)[\varphi_j]
        =\varphi_j+\Lambda_+[u_j|_+-1_{\partial D_j}]=-e_j.
    \end{equation}
    It then follows that
    \begin{align}
        C^\gamma_{ij}&=\int_{\partial D_i} (I+\gamma \Lambda_+)^{-1}[e_j]\, \df \sigma
        =-\int_{\partial D_i} \varphi_j \, \df \sigma \notag \\
        &= -\int_{\p D}  \varphi_j  (u_i|_+ - \Gg \Gvf_i)  \, \df \sigma \notag \\
        &= -\int_{\p D} \left. \p_\Gv u_j \right|_+ u_i|_+ \,\df \sigma + \Gg \int_{\p D} \Gvf_j \Gvf_i \, \df \sigma. \label{eq:Cap:symmetric:d-1}
    \end{align}

    Let $b_j$ be the constant appearing in the third line in \eqref{R}, namely, the constant such that $u_j (x)+b_j\Gamma (x)=O(|x|^{-d+1})$ as $|x|\to \infty$. Then, by \eqref{eq:Robin:sol:explicit}, there is $\psi_j \in B^{-1/2, 2} (\partial D)$ with $\int_{\partial D}\psi_j \, \df \sigma=0$ such that
    $$
        u_j={\SL}[-b_j\varphi_0+\psi_j].
    $$
    It thus follows from the jump formula \eqref{eq:sl:jump} that
    \beq\label{1000}
        \int_{\p D} \left. \p_\Gv u_j \right|_+ u_i|_+ \,\df \sigma
        =\jbk{\left(\frac{1}{2}I+\NP^*\right) [-b_j\varphi_0+\psi_j], {\SL}[-b_i\varphi_0+\psi_i]} .
    \eeq
    Let $c_0$ be the constant such that ${\SL}[\varphi_0]= c_0$ on $\partial D$.
    Since $\Kcal_{\p D}^* [\Gvf_0] = (1/2)\varphi_0$ and $\Scal_{\p D}$ is self-adjoint with respect to $L^2$-inner product, we have
		\begin{align*}
		\jbk{\left(\frac{1}{2}I+\NP^*\right)[-b_j\Gvf_0], \Scal_{\p D}[\Gy_i]} &= \jbk{-b_j \Scal_{\p D} [\Gvf_0], \Gy_i} \\
		&= -b_j c_0\jbk{1, \Gy_i} \\
		&= 0.
		\end{align*}
Since ${\NP}[1]=1/2$ (see \cite[Theorem 4.1]{Mitrea97}), we have
    	\begin{align*}
    	\jbk{\left(\frac{1}{2}I+\NP^*\right)[\Gy_j], \Scal_{\p D}[-b_i \Gvf_0]} &= \jbk{\Gy_j, -b_i c_0 \left(\frac{1}{2}I+\NP\right) [1]} \\
    	&= -b_i c_0 \jbk{\Gy_j, 1 } \\
    	&= 0.
    	\end{align*}
    It then follows from \eqref{1000} and \eqref{eq:Plemelj} that
    \begin{align*}
        \int_{\p D} \left. \p_\Gv u_j \right|_+ u_i|_+ \,\df \sigma
        &=-b_j b_i \|\varphi_0\|_*^2
        -\jbk{\left(\frac{1}{2}I+\NP^*\right)[\psi_j], {\SL}[\psi_i]} \\
        &=-b_j b_i \|\varphi_0\|_*^2
        -\jbk{\left(\frac{1}{2}I+\NP\right){\SL}[\psi_j], {\SL}[\psi_i]}_*,
    \end{align*}
    where $\jbk{\cdot, \cdot}_*$ is the inner product defined in \eqref{eq:inner:product}.
    Thus we have
    \begin{equation}
        \label{eq:Cap_symmetric:d}
        C^\gamma_{ij}=b_j b_i \|\varphi_0\|_*^2
        +\jbk{\left(\frac{1}{2}I+\NP\right){\SL}[\psi_j], {\SL}[\psi_i]}_*+\Gg \int_{\p D} \Gvf_j \Gvf_i \, \df \sigma.
    \end{equation}
    This formula immediately shows that $\Capac^\gamma_D$ is symmetric since $\NP$ is self-adjoint with respect to  $\jbk{\cdot, \cdot}_*$.

    In order to prove the positivity of $\Capac^\gamma_D$, let $c = (c_1, \ldots, c_N)^T \in \Cbb^N$, and let
    $$
    b := (b_1, \ldots, b_N)^T, \quad \Gy := \sum_{k=1}^{N} c_k \Gy_k, \quad \Gvf := \sum_{k=1}^{N} c_k \Gvf_k, \quad u := \sum_{k=1}^{N} c_k u_k.
    $$
    We infer from \eqref{eq:Cap_symmetric:d} that
    \begin{equation}\label{three}
    c \cdot \Capac^\gamma_D c =|c\cdot b|^2 \|\varphi_0\|_*^2
    +\jbk{\left(\frac{1}{2}I+\NP\right){\SL}[\psi], {\SL}[\psi]}_*+\Gg \| \Gvf\|_{L^2 (\partial D)}^2.
    \end{equation}
    The second term on the right-hand side of the above identity is non-negative by Theorem \ref{theo:NP:energy} \ref{enum:NP:sa}, and hence
    \begin{equation*}
    c \cdot \Capac^\gamma_D c \geq 0.
    \end{equation*}

    Suppose that $c\cdot \Capac^\gamma_D c=0$. Then each term on the right-hand side of \eqref{three} is zero.
    Since $c\cdot b=0$, we have
    \[
        u=\sum_{k=1}^N c_k{\SL}[-b_k \varphi_0+\psi_k]={\SL}[c\cdot b \varphi_0+\psi]={\SL}[\psi] \quad \text{on } D^+.
    \]
    Since $\int_{\partial D}\psi\, \df \sigma=0$, we have $u(x)=O(|x|^{-d+1})$ and $|\nabla u(x)|=O(|x|^{-d})$ as $|x|\to \infty$. Thus, by Green theorem in the exterior domain, we obtain
    \[
        \| \nabla u\|_{L^2 (D^+)}^2 =-\jbk{\partial_\nv u|_+, u|_+} =\jbk{\left(\frac{1}{2}I+\NP\right){\SL}[\psi], {\SL}[\psi]}_* =0.
    \]
    Hence $u$ is constant in each connected component of $D^+$. In particular, we have $\p_\Gv u|_+=0$ on $\p D$. By the second condition of \eqref{R}, we have
    \begin{equation}\label{eq:u:span}
        u|_+=\sum_{k=1}^N c_k 1_{\partial D_k}.
    \end{equation}

    Let $D_j^{+}$, $j=1, \ldots, N^+$, be the bounded connected components of $D^+$. Then $u$ is constant on each $\p D_j^+$. Moreover, since $u(x)=O(|x|^{-d+1})$ as $|x|\to \infty$, $u=0$ in the unbounded connected component of $D^+$. Thus we have
    $$
    	u|_+=\sum_{j=1}^{N^{+}} c_j^{+} 1_{\partial D_j^{+}}
    $$
    for some constants $c_j^+$, $j=1, \ldots, N^+$. This together with \eqref{eq:u:span} yields
    $$
    \sum_{k=1}^N c_k 1_{\partial D_k} = \sum_{j=1}^{N^{+}} c_j^{+} 1_{\partial D_j^{+}}.
    $$
    It is known that the set ${\{1_{\partial D_k}\}}_{k=1}^{N} \cup {\{1_{\partial D_j^{+}}\}}_{j=1}^{N^+}$ is linearly independent
    (see \cite[(2.5)]{Mitrea97}), we infer that $c=0$.
\end{proof}

We conclude this subsection by showing an invariance of $\Capac^\gamma_D$ (when it is a scalar) under certain excisions of the domain. Although we will not use it in proving Theorem \ref{theo:L2:sol:approx}, we include it in this paper since it is of independent interest.

\begin{prop}
    \label{prop:included}
    Let $D, E\subset \Rbb^d$ ($d\geq 2$) be bounded connected open sets with compact Lipschitz boundaries such that
    \beq\label{4000}
    E\subset D \quad\mbox{and}\quad \partial D\subset \partial E.
    \eeq
    We further assume $\capac_D>1$ and $\capac_E>1$ when $d=2$. Then
    \beq
    \Capac^\gamma_D=\Capac^\gamma_E
    \eeq
    for all $\gamma\geq 0$.
\end{prop}

\begin{proof}
    We denote the exterior DtN maps for $D$ and $E$ by $\Lambda^D_+$ and $\Lambda^E_+$ respectively. We set
    $\psi:=-(I+\gamma\Lambda^D_+)^{-1}\SL^{-1}[1_{\partial D}]$.
    Then, by the jump relation \eqref{eq:dl:jump} and the Plemelj symmetrization principle \eqref{eq:Plemelj}, we obtain
    \begin{align*}
        1_{\partial D}
        &=-\SL (I+\gamma\Lambda^D_+)[\psi]
        =-{\SL}[\psi]+\gamma \left(\frac{1}{2}I+\NP\right)[\psi] \\
        &=(-{\SL}[\psi]+\gamma {\DL}[\psi])|_-
    \end{align*}
    on $\partial D$. Then, by the uniqueness of the solution to the Dirichlet problem, it holds that
    \[
        -{\SL}[\psi]+\gamma {\DL}[\psi]=1
    \]
    in $D$, and hence in $E$. On the other hand, since
    \begin{align*}
        -\SL[\partial E] (I+\gamma\Lambda^E_+)[\psi 1_{\partial D}]
        &=\left(-\SL[\partial E]+\gamma \left(\frac{1}{2}I+\NP[\partial E]\right)\right)[\psi 1_{\partial D}] \\
        &=\begin{cases}
            (-{\SL}[\psi]+\gamma {\DL}[\psi])|_- & \text{on } \partial D, \\
            -{\SL}[\psi]+\gamma {\DL}[\psi] & \text{on } D\cap \partial E
        \end{cases} \\
        &=1
    \end{align*}
    on $\partial E$ by the assumptions, we have
    \[
        -(I+\gamma\Lambda^E_+)^{-1}\SL[\partial E]^{-1}[1_{\partial E}]=\psi 1_{\partial D}=-(I+\gamma\Lambda^D_+)^{-1}\SL^{-1}[1_{\partial D}]1_{\partial D}.
    \]
    We integrate the above equality on $\partial E$ to obtain the conclusion.
\end{proof}

    A typical example of pairs of domains satisfying \eqref{4000} is an annulus and a ball, and their resistive capacitance matrices are the same according to Proposition \ref{prop:included}. They are scalars and can be computed easily. Let $D\subset \Rbb^d$ ($d\geq 2$) be a ball of radius $R>0$. We further assume that $R<1$ when $d=2$ so that $\capac_D>1$. If $d\geq 3$, a direct calculation by using \eqref{eq:Plemelj} yields
    \begin{align*}
        -\SL (I+\gamma\Lambda_+)[1_{\partial D}]
        =\left(-\SL+\gamma \left(\frac{1}{2}I+\NP\right)\right)[1_{\partial D}]
        =\frac{R^{d-2}}{d-2}+\gamma
    \end{align*}
    on $\partial D$. Thus, we have
    \begin{align*}
        -(I+\gamma\Lambda_+)^{-1} \SL^{-1} [1_{\partial D}]
        =\frac{d-2}{R^{d-2}+(d-2)\gamma},
    \end{align*}
    and hence
    \[
         \Capac^\gamma_D=\frac{(d-2) \omega_d R^{d-1}}{R^{d-2}+(d-2)\gamma},
    \]
    where $\omega_d$ is the surface area of the unit sphere in $\Rbb^d$. If $d=2$, a similar calculation proves
    \begin{equation}\label{eq:capac:disk}
        \Capac^\gamma_D=\frac{2\pi R}{-R\log R+\gamma}.
    \end{equation}

\subsection{Proof of Theorem \ref{theo:L2:sol:approx}}\label{subs:proof:Lip}

We begin with the following representation formula.
\begin{lemm}
    \label{lemm:sol:rep}
    Let $D \subset \Rbb^d$ $(d\geq 2)$ be a bounded open set with the Lipschitz boundary and $U$ be an open ball containing $\overline{D}$. If the solution $u^\gamma$ to \eqref{IPB} with $\gamma>0$ belongs to $H^{s+1/p, p} (U\setminus \overline{D})$ for some $(s, 1/p)\in \Rscr (D)$,  then there exists $\varphi^\gamma\in \widetilde{B}^{s, p}(\partial D)$ such that
    \begin{equation}
        \label{eq:solution:try}
        u^\gamma (x) = h(x)+{\SL}[\varphi^\gamma](x)-\gamma {\DL}[\varphi^\gamma](x), \quad x\in \Rbb^d\setminus \partial D.
    \end{equation}
\end{lemm}

\begin{proof}
Let $\varphi^\gamma:=\partial_\nv u^\gamma|_+$. Since $\partial (U \setminus \overline{D})$ is Lipschitz, $s\in (0, 1)$ and $p, q\in (1, \infty)$ (and $(s, 1/p)\in \Rscr (D)$), the trace from $H^{s+1/p, p} (U\setminus \overline{D})$ into $B^{s, p}(\partial (U \setminus \overline{D}))$ is bounded (cf. \cite{Bergh-Loefstroem76}). Since $\Gg >0$, it follows from the second condition of \eqref{IPB} that $\varphi^\gamma\in B^{s, p}(\partial D)$, and from the third  condition that $\varphi^\gamma\in \widetilde{B}^{s, p}(\partial D)$.

Let $w=u^\gamma-h-{\SL}[\varphi^\gamma]+\gamma {\DL}[\varphi^\gamma]$ in $\Rbb^d\setminus \partial D$. Then, by jump formulas \eqref{eq:sl:jump} and \eqref{eq:dl:jump}, $w$ is the solution to
the transmission problem
\begin{equation}\leqnomode\label{T}\tag{T}
    \begin{dcases}
        \lap w=0 & \text{in } \Rbb^d\setminus \partial D, \\
        \p_\nv w|_+- \p_\nv w|_-=0 & \text{on } \partial D,  \\
        w|_+-w|_-=0 & \text{on } \partial D,  \\
        w (x)=O(|x|^{-d+1}) & \text{as } |x|\to \infty.
    \end{dcases}
\end{equation}
We emphasize that $\p_\nv {\DL}[\varphi^\gamma]$ does not have a jump across $\p D$. We then immediately infer from \eqref{T} that $\nabla w=0$ and hence $w=0$ in $\Rbb^d$.
\end{proof}

We seek a solution $u^\Gg$ to \eqref{IPB} in the form of \eqref{eq:solution:try} which amounts to deriving an integral equation for $\varphi^\gamma$. If $\int_{\p D_j} \varphi^\Gg \, \df \Gs =0$ for all $j=1, \ldots, N$, then $u^\Gg$ defined by \eqref{eq:solution:try} satisfies all the conditions of \eqref{IPB} except the one on the second line. We now derive an integral equation for $\varphi^\Gg$ so that the second condition is satisfied. It suffices to consider the condition that $u^\Gg$ locally constant in $D$ since if this condition is satisfied, then the other condition, namely, $u^\gamma|_+-\gamma \p_\nv u^\gamma|_+=u^\gamma|_-$ is automatically satisfied. In fact, if \eqref{eq:solution:try} holds and $u^\gamma$ is locally constant in $D$, then we have
\begin{align*}
    \partial_\nv u^\gamma|_+&=\partial_\nv u^\gamma|_+-\partial_\nv u^\gamma|_-=\varphi^\gamma,
\end{align*}
hence
\begin{align*}
    u^\gamma|_+-\gamma \partial_\nv u^\gamma |_+
    &=u^\gamma|_--\gamma ({\DL}[\varphi^\gamma]|_+-{\DL}[\varphi^\gamma]|_-)-\gamma \partial_\nv u^\gamma|_+ \\
    &=u^\gamma|_-+\gamma \varphi^\gamma-\gamma \partial_\nv u^\gamma|_+=u^\gamma|_-.
\end{align*}

We now derive the condition for $u^\Gg$ to be locally constant in $D$.
We first observe that the following formula holds for $\psi\in B^{s, p} (\partial D)$, $(s, 1/p)\in \Rscr (D)$:
\begin{equation}
    \label{eq:normal:DL}
    {\DL}[\psi](x)=
        -\SL \Lambda_+ [\psi](x), \quad x\in D,
\end{equation}
where $\Lambda_+$ is the exterior DtN map defined by \eqref{eq:DtN}.
In fact, one can see from the jump formula \eqref{eq:dl:jump} for the double layer potential and Plemelj's formula \eqref{eq:Plemelj} that
$$
{\DL}[\psi]|_- = \left(\frac{1}{2}I+\NP\right)[\psi] =
        \SL \left(\frac{1}{2}I+\NP^*\right)\SL^{-1}[\psi].
$$
The formula \eqref{eq:normal:DL} follows from \eqref{eq:DtN:potential}.

Substituting \eqref{eq:normal:DL} into \eqref{eq:solution:try} yields
\begin{equation}\label{eq:sol:DtN}
u^\gamma(x) = h(x) + {\SL}(I + \Gg \Lambda_+ ) [\varphi^\gamma](x), \quad x \in D.
\end{equation}
Since $u^\Gg$ is locally constant in $D$, we have $\p_\nu u^\Gg|_-=0$ on $\p D$, which amounts to
\begin{equation}
    \label{eq:imperfect:integral}
        \left(\frac{1}{2}I-\NP^*\right)(I+\gamma \Lambda_+) [\varphi^\gamma]=\p_\nv h \quad\mbox{on } \p D.
\end{equation}
This is the integral equation for $\varphi^\Gg$ such that the function $u^\Gg$ defined by \eqref{eq:solution:try} becomes the solution to \eqref{IPB}. If $\Gg=0$, the equation becomes
\begin{equation}
    \label{eq:phi0:Lip}
        \left(\frac{1}{2}I-\NP^*\right) [\varphi^0]=\p_\nv h \quad\mbox{on } \p D,
\end{equation}
and the solution takes the form
\begin{equation}
        \label{eq:solution:try:zero}
        u^0 (x) = h(x)+{\SL}[\varphi^0](x),
    \end{equation}
which is well-known integral representation of the solution to \eqref{IPB} with $\Gg=0$ (see, for example, \cite{Ammari-Kang07}).

\begin{lemm}
    \label{lemm:bie:solvable:Lip}
    Let $D=\bigcup_{i=1}^N D_i \subset \Rbb^d$ ($d\geq 2$) be a bounded open set with the Lipschitz boundary and let $\gamma>0$.

    If $d \geq 3$, then the equation \eqref{eq:imperfect:integral} has a unique solution $\varphi^\gamma$ in 
    \[
        \varphi^\gamma \in \bigcap_{\substack{\Ge_*<s< 1 \\ 1<p<2/(s-\Ge_*)}} \widetilde{B}^{s, p}(\partial D) \cap \bigcap_{1<p<2/(1-\Ge_*)} \widetilde{W}^{1, p}(\partial D).
    \]

    Furthermore, as $\gamma\to +0$, $\varphi^\gamma$ converges to $\varphi^0$ in $L^p (\partial D)$ for any $p\in (1, 2/(1-\Ge_*))$.

    The same conclusion holds for $p=2$ when $d=2$ and $\capac_D>1$.
\end{lemm}

\begin{proof}
    We only give proofs for the three-dimensional case since those for the two-dimensional case are exactly the same. 

    We first prove the existence of solution $\varphi^\gamma$ in $\widetilde{B}^{s, p} (\partial D)$ for $(s, 1/p)\in \Rscr (D)$. Thanks to the fact $\Ker ( -1/2I+\NP)= \lspan \{ 1_{\p D_j} \}_{j=1}^N$ and the relation $\SL^{-1}\NP=\NP^*\SL^{-1}$ (by \eqref{eq:Plemelj}), we see that
    \begin{equation}\label{eq:1/2-K:ker}
        \Ker \left( -\frac{1}{2}I+\NP^*\right)
        =\lspan \{ e_j \}_{j=1}^N,
    \end{equation}
    where $e_j$ is defined by \eqref{ej}.

    Since $\p D$ is Lipschitz, the function $\p_\nv h$ belongs to $L^\infty (\partial D)$ and $\int_{\p D} \p_\nv h=0$. Thus $\p_\nv h \in  \widetilde{B}^{s-1, p} (\partial D)$ (since $s \le 1$).  It thus follows from Theorem \ref{theo:layer:Lip} \ref{enum:NP:inv:Lip} that there is a unique $\varphi^0\in  \widetilde{B}^{s-1, p} (\partial D)$ satisfying \eqref{eq:phi0:Lip}.
Thus, the equation \eqref{eq:imperfect:integral} is reduced to
\begin{equation}
    \label{eq:phi:dtn}
    (I+\gamma \Lambda_+)[\varphi^\gamma]=\varphi^0+\sum_{j=1}^N a^\gamma_j e_j
\end{equation}
for some $a^\gamma =(a^\gamma_j)_{j=1}^N\in \Rbb^{N}$. Since $e_j\in B^{s-1, p} (\partial D)$ by Theorem \ref{theo:layer:Lip} and $I+\gamma\Lambda_+: B^{s, p}(\partial D)\to B^{s-1, p}(\partial D)$ is invertible for all $\gamma>0$ by Theorem \ref{theo:inv}, the equation \eqref{eq:phi:dtn} becomes
\begin{equation}\label{eq:phig:sol:Lip}
    \varphi^\gamma=(I+\gamma \Lambda_+)^{-1}[\varphi^0]+\sum_{j=1}^N a^\gamma_j (I+\gamma \Lambda_+)^{-1}[e_j].
\end{equation}
We emphasize that $\varphi^\gamma\in B^{s, p}(\partial D)$.

In order to find $\varphi^\gamma$ in $\widetilde{B}^{s, p}(\partial D)$, it suffices to find $a^\gamma=(a^\gamma_j)_{j=1}^N$ such that
$$
    \int_{\partial D_i}(I+\gamma \Lambda_+)^{-1}[\varphi^0]\, \df \sigma+\sum_{j=1}^N C^\gamma_{ij} a^\gamma_j =0 ,
$$
or equivalently
\begin{equation}\label{eq:constraint}
    \int_{\partial D_i}(I+\gamma \Lambda_+)^{-1}[\varphi^0]\, \df \sigma + (\Capac^\gamma_D a^\gamma)_i =0 \quad  i=1, \ldots, N,
\end{equation}
where $\Capac^\gamma_D=(C^\gamma_{ij})_{i, j=1}^N$ is the resistive capacitance matrix defined by \eqref{eq:matrix}. The equation \eqref{eq:constraint} is solvable for any $\gamma>0$ by Theorem \ref{positive:definite}. Thus we find a solution $\varphi^\gamma$ in the space 
\[
    \bigcap_{(s, 1/p)\in \Rscr (D)} \widetilde{B}^{s, p}(\partial D)=\bigcap_{\substack{\Ge_*<s< 1 \\ 1<p<2/(s-\Ge_*)}} \widetilde{B}^{s, p}(\partial D),
\]
where the above equality follows from $\widetilde{B}^{s_0, p}(\partial D)\subset \widetilde{B}^{s_1, p}(\partial D)$ for $0<s_1\leq s_0<1$. 

Existence of $\varphi^\gamma$ in $H^{1, p}(\partial D)$ for any $p\in (1, 2/(1-\Ge_*))$ is proved similarly. 

In order to prove the uniqueness of the solution in $\widetilde{B}^{s, p}(\partial D)$ for some $(s, 1/p)\in \Rscr (D)$, suppose that $h=0$. Then $\varphi^0=0$, and the equation \eqref{eq:constraint} becomes
\[
    \Capac^\gamma_D a^\gamma=0.
\]
Since $\Capac^\gamma_D$ is invertible, we obtain $a^\gamma_j=0$. Thus $\varphi^\gamma=0$.

We now prove $\lim_{\gamma\to +0}\varphi^\gamma=\varphi^0$ in $L^p (\partial D)$. By Theorem \ref{theo:unif:res:Lip}, we have
\begin{equation}
    \label{eq:potential:exp:wip:Lip}
    \lim_{\gamma\to +0} (I+\gamma \Lambda_+)^{-1}[\varphi^0]=\varphi^0 \text{ in } L^p (\partial D).
\end{equation}
Denoting $(\Capac^\gamma_D)^{-1}=(C_\gamma^{ij})_{i, j=1}^N$, it follows from Proposition \eqref{prop:rcap:conti} that
\[
    \lim_{\gamma\to +0} C_\gamma^{ij}=C_0^{ij}.
\]
Thus we obtain
\begin{equation}\label{eq:a:lim:Lip}
\begin{aligned}
    a^\gamma_i&=-\sum_{j=1}^N C^{ij}_\gamma \int_{\partial D_j}(I+\gamma \Lambda_+)^{-1}[\varphi^0]\, \df \sigma
    \to -\sum_{j=1}^N C_0^{ij} \int_{\partial D_j}\varphi^0\, \df \sigma=0
\end{aligned}
\end{equation}
as $\gamma\to +0$. Combining \eqref{eq:phig:sol:Lip}, \eqref{eq:potential:exp:wip:Lip} and \eqref{eq:a:lim:Lip}, we obtain the desired conclusion.
\end{proof}

In the course of proof above, we obtain the following result which is worth to be formulated as a theorem. Actually,     a combination of \eqref{eq:sol:DtN}, \eqref{eq:phi:dtn} and \eqref{eq:a:lim:Lip} yields it since $u^0= h+{\SL}[\varphi^0]$.

\begin{theo}
    \label{theo:Lp:sol:g}
    Let $D=\bigcup_{i=1}^N D_i \subset \Rbb^d$ ($d\geq 2$) be a bounded open set with Lipschitz boundary and let $\varphi^0$ be the solution to \eqref{eq:imperfect:integral} with $\Gg=0$. The locally constant function $u^\gamma|_D$ and $a^\gamma=(a^\gamma_i)_{i=1}^N$ in \eqref{eq:constraint} satisfies the relation
    \begin{equation}
        \label{eq:const}
        u^\gamma|_{D_i}-u^0|_{D_i}=-a^\gamma_i=\sum_{j=1}^N C_\gamma^{ij} \int_{\partial D_j} (I+\gamma \Lambda_+)^{-1} [\varphi^0]\df \sigma,
    \end{equation}
    where $(C_\gamma^{ij})_{i, j=1}^N=(\Capac^\gamma_D)^{-1}$.
\end{theo}

In order to lift the convergence of $\varphi^\gamma$ in Lemma \ref{lemm:bie:solvable:Lip} to that of $u^\gamma$, we employ the following mapping properties of layer potentials, whose proofs can be found in \cite[Theorems 7.4, 8.5 and 8.7]{Mitrea-Taylor00}.

\begin{theo}
    \label{theo:layer:mapping:Lip}
    Let $D\subset \Rbb^d$ ($d\geq 3$) be a bounded open set with the Lipschitz boundary. Then the following linear operators are bounded.
    \begin{enumerate}[label=\textnormal{(\alph*)}]
        \item \label{enum:sl:int:Lip:pot} $\SL: H^{s-1, p} (\partial D)\to H^{s+1/p, p} (D)$ for $(s, p)\in [0, 1]\times [2, \infty)$;
        \item \label{enum:dl:int:Lip:pot} $\DL: H^{s, p} (\partial D)\to H^{s+1/p, p}(D)$ for $(s, p)\in [0, 1] \times [2, \infty)$.
        \item \label{enum:sl:int:Lip:BBesov:2}$\SL: H^{s-1, p} (\partial D)\to B^{s+1/p, p}_2 (D)$ for $(s, p)\in [0, 1]\times (1, 2]$;
        \item \label{enum:dl:int:Lip:BBesov:2}$\DL: H^{s, p} (\partial D)\to B^{s+1/p, p}_2 (D)$ for $(s, p)\in [0, 1]\times (1, 2]$.
    \end{enumerate}
\end{theo}

We are now ready to prove Theorem \ref{theo:L2:sol:approx}.

\begin{proof}[Proof of Theorem \ref{theo:L2:sol:approx}]
To prove the uniqueness of the solution, suppose that the solution $u^\gamma$ is a solution to \eqref{IPB} with $h=0$ which belongs to $H^{s+1/p, p} (U\setminus \overline{D})$ for some $(s, 1/p)\in \Rscr (D)$. Then, by Lemma \ref{lemm:sol:rep}, $u^\gamma$ is represented as \eqref{eq:solution:try} for some $\varphi^\gamma\in \widetilde{B}^{s, p}(\partial D)$ with $h=0$. However, Lemma \ref{lemm:bie:solvable:Lip} asserts $\varphi^\gamma=0$. Thus we have $u^\gamma=0$, and hence the solution is unique.

    Next, we deal with the existence. Let $\varphi^\gamma$ be the unique solution to \eqref{eq:imperfect:integral}, whose solvability is guaranteed by Lemma \ref{lemm:bie:solvable:Lip}. Then the function $u^\gamma$ determined by \eqref{eq:solution:try} solves the problem \eqref{IPB}. It follows from Theorem \ref{theo:layer:mapping:Lip} \ref{enum:sl:int:Lip:pot} and \ref{enum:dl:int:Lip:pot} that $u^\gamma|_{U\setminus \overline{D}}\in H^{s+1/p, p} (U\setminus \overline{D})$ for $p\in [2, 2/(1-\Ge_*))$. This proves the existence part.

By Lemma \ref{lemm:bie:solvable:Lip}, $\varphi^\gamma$ converges to $\varphi^0$ in $L^p(\p D)$ for $1< p < 2/(1-\Ge_*)$, as $\gamma\to +0$. Since $u^0=h+{\SL}[\varphi^0]$, we infer that $u^\gamma$ converges to $u^0$ as $\Gg \to 0$ in $B^{1/p, p}_2 (U\setminus \overline{D})$ for any $p\in (1, 2]$ and in $H^{1/p, p} (U\setminus \overline{D})$ for any $p\in [2, 2/(1-\Ge_*))$ by Theorem \ref{theo:layer:mapping:Lip}. 
 \end{proof}

We conclude this section by proving the small resistance limit near infinity, which is also valid in the two-dimensional case.

\begin{theo}
    \label{theo:conv:infinity}
    If $D\subset \Rbb^d$ ($d\geq 2$) is a bounded open set with Lipschitz boundary, then
    \[
        \|(1+|x|)^{d-1+|\alpha|}\partial^\alpha (u^\gamma-u^0)\|_{L^\infty (\Rbb^d\setminus \overline{U})}\to 0
    \]
    as $\gamma\to +0$ for any multi-index $\alpha\in \Nbb_0^d$ and any open ball $U\subset \Rbb^d$ containing $\overline{D}$.
\end{theo}

\begin{proof}
    It is easily proved that the following mappings are bounded for any multi-index $\alpha\in \Nbb_0^d$ and any open ball $U\subset \Rbb^d$ containing $\overline{D}$: 
    \begin{equation}
        \label{eq:layer:potential:end}
\begin{aligned}
        &(1+|x|)^{d-1+|\alpha|}\partial^\alpha {\SL}[\cdot]|_{\Rbb^d\setminus \overline{U}}: \widetilde{L}^2 (\partial D)\to L^\infty (\Rbb^d\setminus \overline{U}), \\
        &(1+|x|)^{d-1+|\alpha|}\partial^\alpha {\DL}[\cdot]|_{\Rbb^d\setminus \overline{U}}: L^2(\partial D)\to L^\infty (\Rbb^d\setminus \overline{U}).
    \end{aligned}
    \end{equation}
    Since $\lim_{\gamma\to +0}\varphi^\gamma=\varphi^0$ in $L^2 (\partial D)$ by Lemma \ref{lemm:bie:solvable:Lip}, where $D$ is dilated when $d=2$ in a way that $\capac_D>1$, the above mapping properties for layer potentials prove the conclusion. 
\end{proof}

\subsection{Regularity and integrability of the solution}

The following theorem is known \cite{FMM98}.

\begin{theo}
    \label{theo:layer:mapping:Besov}
    Let $D\subset \Rbb^d$ ($d\geq 2$) be a bounded open set with Lipschitz boundary. Then, for any $s\in (0, 1)$, the following linear operators are bounded. 
    \begin{enumerate}[label=\textnormal{(\alph*)}]
        \item \label{enum:sl:Besov:Besov}${\SL}: B^{s-1, p} (\partial D) \longrightarrow B^{s+1/p, p} (D)$ $(p \in [1, \infty])$;
        \item \label{enum:dl:Besov:Besov}${\DL}: B^{s, p} (\partial D) \longrightarrow B^{s+1/p, p} (D)$ $(p \in [1, \infty])$.
        \item \label{enum:sl:p:Besov:Bessel}${\SL}: B^{s-1, p} (\partial D) \longrightarrow H^{s+1/p, p} (D)$ $(p \in (1, \infty))$;
        \item \label{enum:dl:p:Besov:Bessel}${\DL}: B^{s, p} (\partial D) \longrightarrow H^{s+1/p, p} (D)$ $(p \in (1, \infty))$.
    \end{enumerate}
\end{theo}

Combining the above theorem with the solution representation \eqref{eq:solution:try} and Lemma \ref{lemm:bie:solvable:Lip}, we immediately obtain the following membership of the solution to \eqref{IPB}. 

\begin{theo}
    \label{theo:sol:membership}
    If $D\subset \Rbb^d$ ($d\geq 3$) is a bounded open set with Lipschitz boundary, then the solution $u^\gamma$ to \eqref{IPB} satisfies 
    \[
        u^\gamma\in \bigcap_{\substack{\Ge_*<s< 1 \\ 1<p<2/(s-\Ge_*)}} (B^{s+1/p, p}(U\setminus \overline{D})\cap H^{s+1/p, p}(U\setminus \overline{D}))
    \]
    for any $\gamma>0$ and any open ball $U\subset \Rbb^d$ including $\overline{D}$.
\end{theo}

\section{Smooth case: Proof of Theorem \ref{theo:existence}}\label{sect:smooth}

\subsection{Layer potentials on smooth boundaries}

We recall or prove mapping properties of the layer potentials when the boundary is smooth, say $C^{k, \theta}$ for some $k\in \Nbb$ and $\theta\in (0, 1]$. 

We first prove identities involving derivatives of the layer potential.

\begin{lemm}
    \label{lemm:reduction}
    Let $D\subset \Rbb^d$ be a bounded open set with the Lipschitz boundary and let $\varphi\in C^{0, 1}(\partial D)$. Then
    \begin{align}
        \label{eq:red:sl}
        \partial_{x_i}{\SL}[\varphi](x)&={\SL}\left[\Lambda_+ [\nv_i \varphi]+\sum_{j=1}^d \partial_{\tau_{ij}}(\nv_j \varphi)\right](x)
    \end{align}
    for $x\in D$.
\end{lemm}

\begin{proof}
    We first observe from \eqref{910} that the following holds:
    $$
    \p_{y_i}\Gamma (x-y)=(\partial_{\nv_y}\Gamma (x-y))\nv_i(y)-\sum_j \nv_j (y)\partial_{\tau_{ij}(y)} \Gamma (x-y).
    $$
Thus we have
    \begin{align}
        \partial_{x_i} {\SL}[\varphi](x)
        &=-\int_{\partial D} \partial_{y_i} \Gamma (x-y)\varphi (y)\, \df \sigma (y) \nonumber\\
        &=-{\DL}[\nv_i \varphi](x)+\sum_{j=1}^d {\SL}[\partial_{\tau_{ij}}(\nv_j \varphi)](x). \nonumber
    \end{align}
One can see from the jump formula \eqref{eq:sl:jump} and \eqref{eq:DtN:potential} that ${\DL}=- {\SL} \Lambda_+$ in $D$. Thus \eqref{eq:red:sl} follows.
\end{proof}

Assume that $D\subset \Rbb^d$ is a bounded open set with compact $C^{k, \theta}$-boundary $\partial D$ for some $k\in \Nbb$ and $\theta\in (0, 1]$. We recall the following properties of boundary integral operators.

\begin{theo}
    \label{theo:layer:Sobolev}
    Assume that $\partial D$ is $C^{k, \theta}$ for some $k\in \Nbb$ and $\theta\in (0, 1]$. Set $\overline{s}:=k+ \theta$. Then the following holds for any $p \in (1,\infty)$. If $d=2$, then we further assume that $\capac_D>1$. 
    \begin{enumerate}[label=\textnormal{(\roman*)}]
        \item \label{enum:sl:invertible} The single layer potential $\SL$ is an isomorphism from $B^{s-1, p}(\partial D)$ onto $B^{s, p}(\partial D)$ for any $s\in (-\overline{s}+1, \overline{s})$.
        \item \label{enum:np:smoothing}If $\overline{s}\in (1, 2]$ and $\delta\in [0, \theta)$, then the mappings
        \begin{align*}
            &\NP: B^{s, p}(\partial D)\longrightarrow B^{s+\delta, p}(\partial D), \\
            &\NP^*: B^{s-1, p}(\partial D)\longrightarrow B^{s-1+\delta, p}(\partial D)
        \end{align*}
        are bounded for any $s\in (-\overline{s}+1, \overline{s}-\delta)$.
        \item \label{enum:np:smoothing:1}If $\overline{s}>2$, then the mappings
        \begin{align*}
            &\NP: B^{s, p}(\partial D)\longrightarrow B^{s+1, p}(\partial D), \\
            &\NP^*: B^{s-1, p}(\partial D)\longrightarrow B^{s, p}(\partial D)
        \end{align*}
        are bounded for any $s\in (-\overline{s}+1, \overline{s}-1)$.
        \item \label{enum:np:Fredholm} The mapping
        \begin{align*}
            \frac{1}{2}I+\NP^*: B^{s-1, p}(\partial D)\longrightarrow B^{s-1, p}(\partial D)
        \end{align*}
        is a bounded Fredholm operator of index zero for any $s\in (-\overline{s}+1, \overline{s})$.
        \item \label{enum:np:inv}The mapping
        \[
            -\frac{1}{2}I+\NP^*: \widetilde{B}^{s-1, p}(\partial D) \longrightarrow \widetilde{B}^{s-1, p}(\partial D)
        \]
        is an isomorphism for any $s\in (-\overline{s}+1, \overline{s})$.
    \end{enumerate}
\end{theo}

\begin{proof}
\ref{enum:sl:invertible} is proved in \cite[Theorem 1.1]{Mazya-Shaposhnikova05} for $s\in (1, \overline{s})\setminus \Nbb$ and \ref{enum:np:smoothing} is proved in \cite[\S 3.4]{Taylor00} for $s\in (1, \overline{s}-\delta)$. We can extend the range of $s$ to that in \ref{enum:sl:invertible} and \ref{enum:np:smoothing} by utilizing the duality and the interpolation. It is likely that \ref{enum:np:smoothing:1} is also known. However, since the authors failed to find an appropriate reference, a proof is given in Appendix \ref{sect:pf:smoothing} for readers' convenience. \ref{enum:np:Fredholm} is an immediate consequence of \ref{enum:np:smoothing} since it implies that $\NP^*$ is compact on $B^{s-1, p}(\partial D)$. 

\ref{enum:np:inv} is proved as follows. We have the topological direct sum decomposition 
\[
    B^{s-1, p}(\partial D)=\lspan \{e_j\}_{j=1}^N \oplus \widetilde{B}^{s-1, p}(\partial D), \quad \psi=\sum_{j=1}^N a_j e_j+\left(\psi-\sum_{j=1}^N a_j e_j\right).
\]
Here $e_j=-\SL^{-1}[1_{\partial D_j}]$ and $a_j\in \Rbb$ is determined by the linear equation 
\[
    \int_{\partial D_j}\psi\, \df\sigma-\sum_{j=1}^N C^0_{ij}a_j=0,
\]
where $(C^0_{ij})_{i, j=1}^N$ is the capacitance matrix, which is invertible. (See Section \ref{subs:capacitance}.) Since 
\[
    \Ker \left(-\frac{1}{2}I+\NP^*: B^{s-1, p}(\partial D)\to B^{s-1, p}(\partial D)\right)=\lspan \{e_j\}_{j=1}^N
\]
and
\[
    \Ran \left(-\frac{1}{2}I+\NP^*: B^{s-1, p}(\partial D)\to B^{s-1, p}(\partial D)\right)\subset \widetilde{B}^{s-1, p}(\partial D),
\]
\ref{enum:np:Fredholm} immediately proves \ref{enum:np:inv}.
\end{proof} 

By Theorem \ref{theo:layer:Sobolev} \ref{enum:sl:invertible} and \ref{enum:np:smoothing}, we immediately obtain the following mapping property of the exterior DtN mapping: 

\begin{coro}\label{coro:DtN:smooth}
    Assume that $\partial D$ is $C^{k, \theta}$ for some $k\in \Nbb$ and $\theta\in (0, 1]$. Set $\overline{s}:=k+ \theta$. Then the exterior DtN map $\Lambda_+=-(1/2I+\NP^*)\SL^{-1}$ is a bounded operator from $B^{s, p}(\partial D)$ into $B^{s-1, p}(\partial D)$ for any $s\in (-\overline{s}+1, \overline{s})$ and $p \in (1,\infty)$.
\end{coro}

\begin{theo}
    \label{theo:layer:mapping:smooth}
    Suppose that $\partial D$ is $C^{k, \theta}$ for some integer $k \in \Nbb$ and $\theta \in (0,1]$. Then, for any $s\in (0, k+\theta)$, the following linear operators are bounded. 
    \begin{enumerate}[label=\textnormal{(\alph*)}]
        \item \label{enum:sl:p}${\SL}: B^{s-1, p} (\partial D) \longrightarrow B^{s+1/p, p} (D)$ $(p \in [1, \infty])$;
        \item \label{enum:dl:p}${\DL}: B^{s, p} (\partial D) \longrightarrow B^{s+1/p, p} (D)$ $(p \in [1, \infty])$.
        \item \label{enum:sl:p:Bessel}${\SL}: B^{s-1, p} (\partial D) \longrightarrow H^{s+1/p, p} (D)$ $(p \in (1, \infty))$;
        \item \label{enum:dl:p:Bessel}${\DL}: B^{s, p} (\partial D) \longrightarrow H^{s+1/p, p} (D)$ $(p \in (1, \infty))$.
    \end{enumerate}
\end{theo}

\begin{proof}
    Proofs of \ref{enum:sl:p} and \ref{enum:dl:p} for the case when $p=\infty$ can be found in \cite{Dondi-LdC17, Miranda70} and the references therein. So we deal with the case when $p < \infty$ in what follows. We fix a decimal part $\beta\in (0, \theta)$ and prove the assertions \ref{enum:sl:p}--\ref{enum:dl:p:Bessel} for $s=m+\beta$ by induction in $m=0, \ldots, k$. The rest case, $k-1+\beta<s<k$ for example, follows from the interpolation. 
    
    The cases when $m=0$ are proved in Theorem \ref{theo:layer:mapping:Besov}. Suppose that the conclusion holds for $s=m+\beta$ for some $m=0, \ldots, k-1$. 

    We first prove \ref{enum:sl:p} for $s=m+1+\beta$. Let $\psi\in B^{m+\beta, p}(\partial D)$. We define the operator $\Pcal_j: B^{m+\beta, p}(\partial D)\to B^{m-1+\beta, p}(\partial D)$ by
    \begin{align*}
     \Pcal_j  [\psi]:=\Lambda_+ [\nv_i  \psi]+\sum_{j=1}^d \partial_{\tau_{ij}}(\nv_j  \psi),
    \end{align*}
    which already appeared in the right hand side of \eqref{eq:red:sl}. The mapping property of $\Pcal_j$ is proved as follows. Since $\partial D$ is $C^{k, \theta}$, $\nv$ has $C^{k-1, \theta}$-regularity. Thus the multiplication operator $\psi\mapsto \nv_j\psi$ is bounded on $B^{m+\beta, p}(\partial D)$. By \eqref{eq:DtN:potential} and Theorem \ref{theo:layer:Sobolev} \ref{enum:sl:invertible}--\ref{enum:np:smoothing:1} (in particular the boundedness of $\NP^*$ on $B^{m+\beta, p}(\partial D)$ and $\NP$ on $B^{m+1+\beta, p}(\partial D)$), we have the boundedness of $\Lambda_+: B^{m+\beta, p}(\partial D)\to B^{m-1+\beta, p}(\partial D)$. On the other hand, since the tangential vector $\tau_{jk}$ has the $C^{k-1, \theta}$-regularity as well, we have the boundedness of $\partial_{\tau_{jk}}: B^{m+\beta, p}(\partial D)\to B^{m-1+\beta, p}(\partial D)$. Hence $\Pcal_j$ has the desired mapping property. 

    Now, we obtain from \eqref{eq:red:sl} and the induction hypothesis that
    \begin{equation}\label{eq:intertwining}
        \begin{aligned}
            \|\partial_{x_j} {\SL} [\psi]\|_{B^{m+\beta+1/p, p}(D)}
            &=\|{\SL}\Pcal_j[\psi]\|_{B^{m+\beta+1/p, p}(D)} \\
            &\lesssim \|\Pcal_j [\psi]\|_{B^{m+\beta-1, p}(\partial D)} \\
            &\lesssim \|\psi\|_{B^{m+\beta, p}(\partial D)}.
        \end{aligned}
    \end{equation}
    Combining the above estimate with 
    \begin{equation}\label{eq:sl:trivial}
        \|{\SL}[\psi]\|_{L^p (D)}\leq \|{\SL}[\psi]\|_{B^{\beta+1/p, p} (D)}\lesssim \|\psi\|_{B^{\beta-1, p}(\partial D)}\leq \|\psi\|_{B^{m+\beta, p}(\partial D)},
    \end{equation}
    which immediately follows from the $m=0$ case, we obtain 
    \[
        \|{\SL}[\psi]\|_{B^{m+1+\beta+1/p, p}(D)}\lesssim \|\psi\|_{B^{m+\beta, p}(\partial D)}
        \] 
    as desired. 
    
    Next, we prove \ref{enum:dl:p} for $s=m+1+\beta$. Let $\varphi\in B^{m+1+\beta, p}(\partial D)$. We have from Lemmas \ref{lemm:reduction:ds}, \ref{lemm:reduction}, and the assertion \ref{enum:sl:p} that
    \begin{align*}
            \|\partial_{x_j} {\DL}[\varphi]\|_{B^{m+\beta+1/p, p}(D)}
            &\leq \sum_{k=1}^{d} \|\partial_{x_k}{\SL}[\partial_{\tau_{jk}}\varphi]\|_{B^{m+\beta+1/p, p}(D)} \\
            &\lesssim \sum_{k=1}^{d}\|{\SL}[\partial_{\tau_{jk}}\varphi]\|_{B^{m+1+\beta+1/p, p}(D)} \\
            &\lesssim \sum_{k=1}^{d}\|\partial_{\tau_{jk}}\varphi\|_{B^{m+\beta, p}(\partial D)}
            \lesssim \|\varphi\|_{B^{m+1+\beta, p}(\partial D)}.
    \end{align*}
    Combining the above estimate with $\|{\DL}[\varphi]\|_{L^p(D)}\lesssim \|\varphi\|_{B^{m+\beta}(\partial D)}$, which is proved similarly to \eqref{eq:sl:trivial}, we obtain the desired estimate. 

    We can similarly prove \ref{enum:sl:p:Bessel} and \ref{enum:dl:p:Bessel}.
\end{proof}

\subsection{Uniform resolvent estimate on Besov spaces}

On smooth boundaries, we can extend Theorem \ref{theo:inv} for Sobolev spaces of higher exponent: 

\begin{theo}
    \label{theo:DtN:iso:smooth}
    If $\partial D$ is $C^{k, \theta}$ for some $k\geq 1$ and $\theta\in (0, 1]$, then the mapping $I+\gamma \Lambda_+: B^{s, p}(\partial D)\to B^{s-1, p}(\partial D)$ is an isomorphism for any $\gamma>0$, $s\in (-k-\theta+1, k+\theta)$ and $p\in (1, \infty)$. 
\end{theo}

\begin{proof}
    It immediately follows from Theorem \ref{theo:layer:Sobolev} (and a similar argument to Theorem \ref{theo:inv}) that $I+\gamma \Lambda_+: B^{s, p}(\partial D)\to B^{s-1, p}(\partial D)$ is a bounded Fredholm operator of index zero for any $s\in (-k-\theta+1, k+\theta)$. Thus it suffices to prove that $\varphi\in B^{s, p}(\partial D)$ for some $s\in (-k-\theta+1, k+\theta)$ and $(I+\gamma\Lambda_+)[\varphi]=0$ imply $\varphi=0$. 
    
    When $s\in [0, k+\theta)$, we just recall from the proof of Theorem \ref{theo:inv} that $(I+\gamma \Lambda_+)[\varphi]=0$ and $\varphi \in L^p (\partial D)$ for some $p\in (1, \infty)$ imply $\varphi=0$, where we used the assumption that $\partial D$ is smoother than $C^1$. Thus $I+\gamma \Lambda_+: B^{s, p}(\partial D)\to B^{s-1, p}(\partial D)$ is an isomorphism if $s\in [0, k+\theta)$ and $p\in (1, \infty)$. On the other hand, when $s\in (-k-\theta+1, 0)$, the injectivity of $I+\gamma \Lambda_+: B^{s, p}(\partial D)\to B^{s-1, p}(\partial D)$ is proved as follows. Assume that $\varphi\in B^{s, p}(\partial D)$ satisfies $(I+\gamma \Lambda_+)[\varphi]=0$. Then, for any $\psi\in B^{1-s, q}(\partial D)$, where $q=p(p-1)$ is the H\"older conjugate of $p$, we have 
    \[
        0=\jbk{(I+\gamma \Lambda_+)[\varphi], \psi}=\jbk{\varphi, (I+\gamma \Lambda_+)[\psi]}.
    \]
    Since we already proved that $I+\gamma \Lambda_+: B^{1-s, q}(\partial D)\to B^{-s, q}(\partial D)$ is an isomorphism, we can replace $\psi\in B^{1-s, q}(\partial D)$ with any image $(I+\gamma \Lambda_+)^{-1}[\widetilde{\psi}]$ of $\widetilde{\psi}\in B^{-s, q}(\partial D)$ by the inverse $(I+\gamma \Lambda_+)^{-1}: B^{-s, q}(\partial D)\to B^{1-s, q}(\partial D)$. Thus we have $\langle\varphi, \widetilde{\psi}\rangle=0$ for any $\widetilde{\psi}\in B^{-s, q}(\partial D)$. Hence $\varphi=0$. 
\end{proof}

We obtain the following uniform resolvent estimate on higher-order Sobolev spaces for the exterior DtN map analogous to Theorem \ref{theo:unif:res:Lip}:

\begin{theo}
    \label{theo:unif:resolvent:smooth}
    Suppose that $\partial D$ is $C^{k, \theta}$ for some $k\in \Nbb$ and $\theta\in (0, 1]$. Then the estimate
    \begin{equation}\label{eq:unif:resolvent:init}
    \limsup_{\gamma\to +0}\|(I+\gamma \Lambda_+)^{-1}\|_{B^{s-1, p}(\partial D)\to B^{s-1, p}(\partial D)}<\infty
\end{equation}
is valid for all $s\in (-k-\theta+1, k+\theta)$ and $p\in (1, \infty)$. 
\end{theo}

\begin{proof}
    The analogous estimate to \eqref{eq:unif:resolvent:init} is proved in \cite[Theorem 6.10]{Escher-Seiler08} for the \textit{interior} DtN map and $s$ in the range $(-\theta, 1+\theta)$ if $k=1$ and $(-1, 2)$ if $k\geq 2$. Since exactly the same reasoning is applicable to the exterior DtN map, the estimate \eqref{eq:unif:resolvent:init}  is valid for $s$ in the range. To extend the range of $s$ to $(-k-\theta+1, k+\theta)$, we need to deal with the case when $k \geq 2$ and $s \in (-k-\Gt+1, -1] \cup [2, k+\Gt)$.

Suppose $k \geq 2$ and $s\in [2,k+\Gt)\cap [2, 3)$. Let $\varphi\in B^{s-1, p}(\partial D)$ and $\psi^\gamma=(I+\gamma \Lambda_+)^{-1}[\varphi]$. Then, we have from Theorem \ref{theo:layer:Sobolev} \ref{enum:sl:invertible}
\begin{align}
    \|\psi^\gamma\|_{B^{s-1, p}(\partial D)}
    &\lesssim \|\SL^{-1}(I+\gamma \Lambda_+)^{-1}[\varphi]\|_{B^{s-2, p}(\partial D)} \nonumber\\
    &\leq \|(I+\gamma \Lambda_+)^{-1}\SL^{-1}[\varphi]\|_{B^{s-2, p}(\partial D)} \nonumber \\
    &\qquad +\|[\SL^{-1}, (I+\gamma \Lambda_+)^{-1}][\varphi]\|_{B^{s-2, p}(\partial D)} \label{650}.
\end{align}
Here, $[ \, , \, ]$ denotes the commutator. By \eqref{eq:Plemelj} and \eqref{eq:DtN:potential}, we have
\[
    [I+\gamma \Lambda_+, \SL^{-1}]=-\gamma \SL^{-1}(\NP-\NP^*)\SL^{-1},
\]
and hence
\begin{align*}
    [\SL^{-1}, (I+\gamma \Lambda_+)^{-1}][\varphi]
    &=\gamma (I+\gamma \Lambda_+)^{-1}\SL^{-1}(\NP^*-\NP)\SL^{-1}(I+\gamma \Lambda_+)^{-1}[\varphi] \\
    &=\gamma (I+\gamma \Lambda_+)^{-1}\SL^{-1}(\NP^*-\NP)\SL^{-1}[\psi^\gamma].
\end{align*}
Observe that all the mappings in the following diagram are bounded regardless of $\Gg$ (small); the first and third by Theorem \ref{theo:layer:Sobolev} \ref{enum:sl:invertible}, the second by Theorem \ref{theo:layer:Sobolev} \ref{enum:np:smoothing:1}, and the last one by the known result mentioned at the beginning of this proof since $s-2 \in [0,1)$:
\begin{equation*}\label{eq:diag:2}
\begin{tikzcd}
    B^{s-1, p}(\partial D) \arrow[r, "\SL^{-1}"] & B^{s-2, p}(\partial D) \arrow[r, "\NP^*-\NP"]  &[15pt] B^{s-2, p}(\partial D) \\
    \phantom{B^{s-2, p}(\partial D)} \arrow[r, "\SL^{-1}"] & B^{s-2, p}(\partial D) \arrow[r, "(I+\gamma\Lambda_+)^{-1}"] & B^{s-2, p}(\partial D).
\end{tikzcd}
\end{equation*}
Thus we have
\begin{equation*}
    \label{eq:commutator:est}
    \|[\SL^{-1}, (I+\gamma \Lambda_+)^{-1}][\varphi]\|_{B^{s-2, p}(\partial D)}\lesssim \gamma \|\psi^\gamma\|_{B^{s-1, p}(\partial D)} .
\end{equation*}
It then follows from \eqref{650} that
\begin{align*}
    \|\psi^\gamma\|_{B^{s-1, p}(\partial D)}
    &\lesssim \|\SL^{-1}[\varphi]\|_{B^{s-2, p}(\partial D)}+\gamma \|\psi^\gamma\|_{B^{s-1, p}(\partial D)} \\
    &\lesssim \|\varphi\|_{B^{s-1, p}(\partial D)}+\gamma \|\psi^\gamma\|_{B^{s-1, p}(\partial D)}.
\end{align*}
For sufficiently small $\gamma$, the second term in the right hand side is absorbed to the left hand side. Thus we obtain
\begin{equation}
    \label{eq:unif:resolvent:2}
    \|\psi^\gamma\|_{B^{s-1, p}(\partial D)}\lesssim \|\varphi\|_{B^{s-1, p}(\partial D)}.
\end{equation}

So far we proved \eqref{eq:unif:resolvent:init} when $s\in [0, k+\theta)\cap [2,3)$, in particular, if $s \in [0,3)$ if $k+\Gt > 3$. So we can apply the same argument to the case when $s\in [0, k+\theta)\cap [3,4)$ if $k+\Gt > 3$.
By iterating this procedure, we obtain \eqref{eq:unif:resolvent:2} for $s\in [0, k+\theta)$. The case when $s \in (-k-\theta+1, -1]$ follows from the duality and the symmetry of $\Lambda_+$ in the sense that $\jbk{\Lambda_+[\varphi], \psi}=\jbk{\varphi, \Lambda_+ [\psi]}$ for all $\varphi\in B^{s, p}(\partial D)$ and $\psi\in B^{-s+1, p/(p-1)}(\partial D)$.
\end{proof}

\subsection{Expansion of operators on the boundary}

The following proposition plays a crucial role in proving Theorem \ref{theo:existence}.

\begin{prop}\label{prop:perturbation}
Assume that $\partial D$ is $C^{k, \theta}$ for some $k\in \Nbb$ and $\theta\in (0, 1]$, and let $s\in (k, k+\theta)$ and $p\in (1, \infty)$. It holds  for any $\varphi \in B^{s-1, p}(\partial D)$ and $K=0, 1, \ldots, k-1$ that
    \begin{align}
        (I+\gamma \Lambda_+)^{-1}[\varphi]=\sum_{l=0}^K (-\gamma\Lambda_+)^l [\varphi]+o (\gamma^K) \quad \text{in } B^{s-1-K, p}(\partial D) \label{eq:yosida}
    \end{align}
    as $\gamma\to +0$.
\end{prop}

\begin{proof}
Since
\begin{equation}\label{idK}
    (I+\gamma \Lambda_+)^{-1}-\sum_{l=0}^K (-\gamma\Lambda_+)^l = (I+\gamma \Lambda_+)^{-1}(-\gamma\Lambda_+)^K -(-\gamma\Lambda_+)^K
\end{equation}
    for $K=0, 1, \ldots, k-1$, we have for any $\varphi\in B^{s-1, p}(\partial D)$
    \begin{align*}
        \left\| (I+\gamma \Lambda_+)^{-1}[\varphi]-\sum_{l=0}^K (-\gamma\Lambda_+)^l [\varphi]\right\|
        = \gamma^K \left\| (I+\gamma \Lambda_+)^{-1}\Lambda_+^K [\varphi]-\Lambda_+^K [\varphi] \right\| .
    \end{align*}
    Here and in what follows in this proof, the norm is the $B^{s-1-K, p}(\partial D)$-norm. Let $C\in (0, \infty)$ be the left hand side of \eqref{eq:unif:resolvent:init} (in Theorem \ref{theo:unif:resolvent:smooth}) with $s$ replaced by $s-K$.

    If $\psi \in B^{s, p}(\partial D)$, then we have
    \begin{align*}
    & \| (I+\gamma \Lambda_+)^{-1}\Lambda_+^K [\varphi]-\Lambda_+^K [\varphi]\| \\
    & \le \| (I+\gamma \Lambda_+)^{-1}(\Lambda_+^K [\varphi]-\psi)\| + \| (I+\gamma \Lambda_+)^{-1} [\psi] -\psi \| + \| \psi-\Lambda_+^K [\varphi]\| \\
        & \le (C+1) \| \psi-\Lambda_+^K [\varphi]\| + \| (I+\gamma \Lambda_+)^{-1} [\psi] -\psi \| \\
        &= (C+1) \| \psi-\Lambda_+^K [\varphi]\| + \gamma \| (I+\gamma \Lambda_+)^{-1} \Lambda_+ [\psi] \|.
    \end{align*}
    Since $\Lambda_+[\psi]\in B^{s-1, p}(\partial D)$ (Corollary \ref{coro:DtN:smooth}), we infer from Theorem \ref{theo:unif:resolvent:smooth} that
    \begin{align*}
        \limsup_{\gamma\to +0}\gamma^{-K}\left\| (I+\gamma \Lambda_+)^{-1}[\varphi]-\sum_{l=0}^K (-\gamma\Lambda_+)^l [\varphi]\right\| \lesssim \|\psi - \Lambda_+^K [\varphi] \|
    \end{align*}
    for any $\psi \in B^{s, p}(\partial D)$. Since $B^{s, p}(\partial D)$ is dense in $B^{s-1-K, p}(\partial D)$, we have
    \begin{align*}
        \limsup_{\gamma\to +0}\gamma^{-K}\left\| (I+\gamma \Lambda_+)^{-1}[\varphi]-\sum_{l=0}^K (-\gamma\Lambda_+)^l [\varphi]\right\| =0,
    \end{align*}
    which yields \eqref{eq:yosida}.
\end{proof}

We can also expand the resistive capacitance matrix.

\begin{prop}\label{lemm:matrix:inv}
    If $\partial D$ is $C^{k, \theta}$ for some $k\in \Nbb$ and $\theta\in (0, 1]$, then
    \[
        C^\gamma_{ij}=\sum_{l=0}^{k-1} (-\gamma)^l \int_{\partial D_i} \Lambda_+^l [e_j]\, \df\sigma+o(\gamma^{k-1})
    \]
    as $\gamma\to +0$.
\end{prop}

\begin{proof}
    Since $1_{\partial D_j}\in B^{k, 2}(\partial D)$, we have $e_j=-\SL^{-1}[1_{\partial D_j}]\in B^{k-1, 2} (\partial D)$ from Theorem \ref{theo:layer:Sobolev} \ref{enum:sl:invertible}. Thus, by the Cauchy-Schwarz inequality, we obtain
    \begin{align*}
        &\left|C^\gamma_{ij}-\sum_{l=0}^{k-1} (-\gamma)^l \int_{\partial D_i} \Lambda_+^l [e_j]\, \df\sigma\right| \\
        &\leq \int_{\partial D_i} \left|(I+\gamma \Lambda_+)^{-1}[e_j]-\sum_{l=0}^{k-1} (-\gamma\Lambda_+)^l [e_j]\right|\, \df \sigma \\
        &\leq |\partial D_i|^{1/2}\left\|(I+\gamma \Lambda_+)^{-1}[e_j]-\sum_{l=0}^{k-1} (-\gamma\Lambda_+)^l [e_j]\right\|_{L^2 (\partial D_i)},
    \end{align*}
    where $|\partial D_i|$ is the surface area of $\partial D_i$. Now we use Proposition \ref{prop:perturbation} to obtain the desired asymptotics as $\gamma\to +0$.
\end{proof}

\subsection{Proof of Theorem \ref{theo:existence}}\label{sect:general}

The following lemma is used for proof of Theorem \ref{theo:existence}. 

\begin{lemm}
    \label{lemm:bie:solvable}
    Assume that $\partial D$ is $C^{k, \theta}$ for some $k\in \Nbb$ and $\theta\in (0, 1]$, and let $s\in (k, k+\theta)$ and $p\in (1, \infty)$. The equation \eqref{eq:imperfect:integral} has a unique solution $\varphi^\gamma$ in $\widetilde{B}^{s, p}(\partial D)$ for any $\gamma>0$.

    Furthermore, there exists $\psi_l\in \widetilde{B}^{s-1-l,p}(\partial D)$ ($l=0, 1, \ldots, k-1$) with $\psi_0=\varphi^0$ such that
    \begin{equation}
        \label{eq:potential:limit}
        \varphi^\gamma=\sum_{l=0}^K \gamma^l \psi_l+o(\gamma^K) \quad \text{in } B^{s-1-K, p} (\partial D)
    \end{equation}
    as $\gamma\to +0$ for any $K=0, 1, \ldots, k-1$.
\end{lemm}

\begin{proof}
Since the proof for uniqueness and existence of $\varphi^\gamma$ is analogous to Lemma \ref{lemm:bie:solvable:Lip}, we only give a proof of \eqref{eq:potential:limit}. We expand $(I+\gamma\Lambda_+)^{-1}$ in \eqref{eq:phig:sol:Lip} according to Proposition \ref{prop:perturbation} to obtain
\begin{equation}
    \label{eq:potential:exp:wip}
    \varphi^\gamma=\sum_{l=0}^K (-\gamma\Lambda_+)^l \left[\varphi^0+\sum_{i=1}^N a^\gamma_i e_i\right]+o(\gamma^K) \quad \text{in } B^{s-1-K, p}(\partial D)
\end{equation}
as $\gamma\to +0$ for any $K=0, 1, \ldots, k-1$.
We denote $(\Capac^\gamma_D)^{-1}=(C_\gamma^{ij})_{i, j=1}^N$. By Proposition \ref{lemm:matrix:inv}, we can find $(A^{ij}_l)_{i, j=1}^N\in \Rbb^{N\times N}$ ($l=0, 1, \ldots, k-1$) such that
\[
    C^{ij}_\gamma=\sum_{l=0}^{k-1} \gamma^l A^{ij}_l +o(\gamma^{k-1}) \quad (\gamma\to +0).
\]
By \eqref{eq:constraint} and Proposition \ref{prop:perturbation}, we obtain
\begin{align*}
    a^\gamma_i&=-\sum_{j=1}^N C^{ij}_\gamma \int_{\partial D_j}(I+\gamma \Lambda_+)^{-1}[\varphi^0]\, \df \sigma \\
    &=-\sum_{j=1}^N\sum_{l, m=0}^{k-1} \gamma^{l+m} A^{ij}_l \int_{\partial D_j}(-\Lambda_+)^m [\varphi^0]\, \df \sigma+o(\gamma^{K})
\end{align*}
as $\gamma\to +0$. Combining this expansion with \eqref{eq:potential:exp:wip}, we obtain the desired expansion \eqref{eq:potential:limit} with the obvious definition of $\psi_l$:
\begin{equation}\label{eq:psik}
\begin{aligned}
    \psi_l:=&\,(-\Lambda_+)^l [\varphi^0] \\
    & \quad -\sum_{\substack{m, n\geq 0 \\ l+m\leq l}} \sum_{i, j=1}^N A^{ij}_{l-m-n} \left(\int_{\partial D_j} (-\Lambda_+)^n [\varphi^0]\, \df \sigma \right)(-\Lambda_+)^m [e_i] ,
\end{aligned}
\end{equation}
which is easily shown to belong to $B^{s-1-l, p}(\partial D)$.
\end{proof}

Now we prove Theorem \ref{theo:existence}.

\begin{proof}[Proof of Theorem \ref{theo:existence}]
    The proof of the uniqueness when $1<p<\infty$ is same as in Theorem \ref{theo:Lp:sol:g}. The uniqueness when $p=\infty$ is immediately proved by $H^{1, \infty}(U\setminus \overline{D})\subset H^{1, p}(U\setminus \overline{D})$ for any $1<p<\infty$ and any open ball $U\subset \Rbb^d$ including $\overline{D}$.

    Next, we prove \ref{enum:asympt}. We first deal with the case when $1<p<\infty$. Let $\varphi^\gamma\in \widetilde{B}^{s, p}(\partial D)$ be the unique solution to \eqref{eq:imperfect:integral}. Then the function $u^\gamma$ determined by \eqref{eq:solution:try} solves the problem \eqref{IPB} and it is proved by Theorem \ref{theo:layer:mapping:smooth} that $u^\gamma|_{U\setminus \overline{D}}\in B^{s+1/p, p}(U\setminus \overline{D})\cap H^{s+1/p, p}(U\setminus \overline{D})$ for any open ball $U\subset \Rbb^d$ including $\overline{D}$.

    Let $\psi_l$ be the functions which appeared in \eqref{eq:potential:limit}. Set
    \[
        \rho^\gamma_K:=\varphi^\gamma-\sum_{l=0}^K \gamma^l \psi_l.
    \]
    It then follows from \eqref{eq:solution:try} that
    \begin{equation}\label{6100}
        u^\gamma =\sum_{l=0}^K \gamma^l v_l-\gamma^{K+1}{\DL}[\psi_K]+{\SL}[\rho^\gamma_K]-\gamma {\DL}[\rho^\gamma_K]
    \end{equation}
    where
    \[
        v_0:=u^0= {\SL}[\psi_0]+h, \quad v_l:={\SL}[\psi_l]-{\DL}[\psi_{l-1}]
    \]
    for $l=1, 2, \ldots, K$.

    By \eqref{eq:potential:limit}, we have
    \beq\label{6001}
    \| \rho^\gamma_K \|_{B^{s+1/p-2-K, p}(\partial D)}=o(\gamma^K).
    \eeq
    Since $\psi_l\in \widetilde{B}^{s-1-l,p}(\partial D)$, it follows from Theorem \ref{theo:layer:mapping:smooth} \ref{enum:sl:p} and \ref{enum:dl:p} that
    \beq
    v_l \in B^{s+1/p-l, p}(U\setminus \overline{D})\cap H^{s+1/p-l, p}(U\setminus \overline{D})
    \eeq
    for any open ball $U\subset \Rbb^d$ including $\overline{D}$. Let $R^\gamma_K:=u^\gamma-\sum_{l=0}^K \gamma^l v_l$. Since 
    \begin{equation}
        \label{eq:remainder}
        R^\gamma_K= -\gamma^{K+1}{\DL}[\psi_K]+{\SL}[\rho^\gamma_K]-\gamma {\DL}[\rho^\gamma_K],
    \end{equation}
    it can be proved using \eqref{6001}, the fact that $\psi_K\in \widetilde{B}^{s-1-K,p}(\partial D)$, and Theorem \ref{theo:layer:mapping:smooth} that
    $$
    \left\|R^\gamma_K\right\|_{B^{s+1/p-1-K, p}(U\setminus \overline{D})}+\left\|R^\gamma_K\right\|_{H^{s+1/p-1-K, p}(U\setminus \overline{D})}=o(\gamma^K).
    $$

    When $p=\infty$, we take $r\in (k, k+\theta)$ and $q\in (1, \infty)$ such that $s<r-(d-1)/q<k+\theta$. Let $\varphi^\gamma \in \widetilde{B}^{r, q}(\partial D)$ be the solution to \eqref{eq:imperfect:integral}. By the Sobolev embedding $B^{r, q}(\partial D)\hookrightarrow B^{r-(d-1)/q, \infty}(\partial D)$, we obtain $\varphi^\gamma \in \widetilde{B}^{r-(d-1)/q, \infty} (\partial D)$. Hence, by Theorem \ref{theo:layer:mapping:smooth} \ref{enum:sl:p} and \ref{enum:dl:p}, we obtain $u^\gamma-h\in B^{r-(d-1)/q, \infty} (U\setminus \overline{D})\subset H^{k, \infty}(U\setminus \overline{D})$ for any open ball $U$ containing $\overline{D}$. By Lemma \ref{lemm:bie:solvable} and the same Sobolev embedding as before, it holds that
    \[
        \left\|\varphi^\gamma-\sum_{l=0}^K \gamma^l \psi_l\right\|_{B^{r-1-K-(d-1)/q, \infty}(\partial D)}=o(\gamma^K)
    \]
    as $\gamma\to +0$. Then we obtain 
    \[
        \left\|u^\gamma-\sum_{l=0}^K \gamma^l v_l\right\|_{B^{s-1-K, \infty}(U\setminus \overline{D})}=o(\gamma^K) \quad (\gamma\to +0)
    \]
    from Theorem \ref{theo:layer:mapping:smooth} and Lemma \ref{lemm:bie:solvable}.
\end{proof}

\begin{theo}
    \label{theo:smooth:infinity}
    For any open ball $U\subset \Rbb^d$, $K=0, 1\, \ldots, k-1$ and multi-index $\alpha\in \Nbb_0^d$, we have 
    \begin{equation}\label{eq:asymptotic:exp:infinity}
        \left\|(1+|x|)^{d+|\alpha|-1}\partial^\alpha \left(u^\gamma -\sum_{l=0}^K \gamma^l v_l\right)\right\|_{L^\infty (\Rbb^d\setminus \overline{U})}=o(\gamma^K)
    \end{equation}
    as $\gamma\to +0$. 
\end{theo}

\begin{proof}
    Let $\psi_l$, $\rho^\gamma_K$ and $R^\gamma_K$ be the functions which appeared in the proof of Theorem \ref{theo:existence}. By Lemma \ref{lemm:bie:solvable}, $\int_{\partial D}\psi_l \, \df \sigma=\int_{\partial D}\rho^\gamma_K\, \df \sigma=0$. We combine this property with \eqref{eq:layer:potential:end} to obtain
    \[
    \|(1+|x|)^{d+|\alpha|-1}\partial^\alpha R^\gamma_K\|_{L^\infty (\Rbb^d\setminus \overline{U})}=o(\gamma^K).
    \]
    This proves \eqref{eq:asymptotic:exp:infinity}.
\end{proof}

%So far we completed the proof of the assertion \ref{enum:convergence}. The assertion \ref{enum:convergence:smooth} follows from $u^\gamma-h\in H^{r-(d-1)/q, \infty} (U\setminus \partial D)$ with $r-(d-1)/q>s$.

One also obtains the subprincipal term $v_1$ in Theorem \ref{theo:existence} when $\partial D$ is smoother than $C^2$:
\begin{theo}\label{theo:subprincipal}
    If $\partial D$ is $C^{2, \theta}$ for some $\theta\in (0, 1]$, then
    \begin{equation}\label{eq:subprincipal}
    v_1=
    \begin{dcases}
        {\SL}[\widetilde{\varphi}]-\sum_{i, j=1}^N C_0^{ij}\left(\int_{\partial D_j} \widetilde{\varphi}\, \df\sigma\right){\SL}[e_i] & \text{in } D^+, \\
        \sum_{i, j=1}^N C_0^{ij}\left(\int_{\partial D_j} \widetilde{\varphi}\, \df\sigma\right)1_{D_i} & \text{in } D,
    \end{dcases}
\end{equation}
where $\widetilde{\varphi}:=\SL^{-1}[\varphi^0]$.
\end{theo}

\begin{proof}
    The potential $\psi_1$ in \eqref{eq:potential:limit} can be calculated as
\begin{equation}
    \label{eq:psi01}
    \begin{aligned}
        \psi_1&=-\Lambda_+[\varphi^0]+\sum_{i, j=1}^N C_0^{ij} \left(\int_{\partial D_j} \Lambda_+ [\varphi^0]\, \df\sigma\right)e_i
    \end{aligned}
\end{equation}
by using the explicit formula \eqref{eq:psik} and $\int_{\partial D_j}\varphi^0\, \df \sigma=0$. By the Gauss divergence theorem, we obtain
\begin{align*}
    \int_{\partial D_j} \Lambda_+ [\varphi^0]\, \df\sigma
    &=-\int_{\partial D_j} \left(\frac{1}{2}I+\NP^*\right) \SL^{-1}[\varphi^0]\, \df\sigma \\
    &=-\int_{\partial D_j} \SL^{-1}[\varphi^0]\, \df\sigma -\int_{\partial D_j} \left(-\frac{1}{2}I+\NP^*\right) \SL^{-1}[\varphi^0]\, \df\sigma \\
    &=-\int_{\partial D_j} \widetilde{\varphi}\, \df\sigma -\int_{\partial D_j} \partial_\nv \SL^{-1}[\varphi^0]|_-\, \df\sigma \\
    &=-\int_{\partial D_j} \widetilde{\varphi}\, \df\sigma.
\end{align*}
Then we substitute $\varphi^\gamma =\varphi^0+\gamma \psi_1+o(\gamma)$ to \eqref{eq:solution:try} and obtain
\[
    v_1=-\SL\Lambda_+[\varphi^0]-{\DL}[\varphi^0]-\sum_{i, j=1}^N C_0^{ij}\left(\int_{\partial D_j} \widetilde{\varphi}\, \df\sigma\right){\SL}[e_i]
\]
on $\Rbb^d\setminus \partial D$. Since
\[
    {\DL}[\varphi^0]-\SL\left(-\frac{1}{2}I+\NP^*\right)[\widetilde{\varphi}]={\SL}[\widetilde{\varphi}]1_D
\]
on $\Rbb^d\setminus \partial D$ and ${\SL}[e_i]=-1_{D_i}$ on $D$, we obtain the desired formula \eqref{eq:subprincipal}.
\end{proof}

% We have obtained the following theorem in the process of proving Theorem \ref{theo:existence}.

% \begin{theo}
%     \label{theo:additional}
%     If $\partial D$ is $C^{k, \theta}$ for some $k\geq 2$ and $\theta\in (0, 1]$, then $u^\gamma|_{D^+}$ with the boundary value $u^\gamma|_+$ belongs to $C^{k, \theta-\delta} (\overline{D^+})$ for all $\gamma>0$ and $\delta\in (0, \theta)$.
% \end{theo}

\section{Discussion}\label{sect:discussion}

In this paper, we proved the unique existence of the solution $u^\gamma$ to \eqref{IPB} with the regularity $B^{s+1/p, p}\cap H^{s+1/p}$ near $\partial D$ and asymptotics as $\gamma\to +0$ in $B^{s+1/p-1, p}(U \setminus \overline{D})\cap H^{s+1/p-1, p}(U \setminus \overline{D})$ for any open ball $U$ containing $\overline{D}$. Now it is natural to ask whether $u^\gamma$ converges to $u^0$, that is, the solution to \eqref{IPB} with $\Gg=0$, in $B^{s+1/p, p}(U \setminus \overline{D})\cap H^{s+1/p, p}(U \setminus \overline{D})$. 

% If it is affirmative, then it must hold that
% \[
%     (I+\gamma\Lambda_+) [\varphi^\gamma]\to \varphi^0 \quad \text{in } B^{\mu-1-1/p, p}(\partial D)
% \]
% as $\gamma\to +0$ \textcolor[rgb]{1.00,0.00,0.00}{(what is $\mu$?)}. The convergence in $B^{s-1, p}(\partial D)$ is proved by using Proposition \ref{prop:perturbation} \textcolor[rgb]{1.00,0.00,0.00}{(what does this mean?)}. However, since we are asking the convergence in $B^{s, p}(\partial D)$, we may need more precise analysis. 

We suspect that $u^\gamma$ may fail to converge to $u^0$ in $B^{s+1/p, p}(U\setminus\overline{D}) \cap H^{s+1/p, p}(U \setminus \overline{D})$. If it is the case, the problem \eqref{IPB} is a singular perturbation of \eqref{IPB} with $\Gg=0$ where $\Gg$ is the perturbation parameter. We expect that this problem is solved in the future studies.

\appendix

\section{Proof of Theorem \ref{theo:layer:Sobolev} \ref{enum:np:smoothing:1}}\label{sect:pf:smoothing}

Throughout this section, we denote by $\Bbb (\Xcal, \Ycal)$ the space of all bounded linear operators from a normed space $\Xcal$ into another normed space $\Ycal$. 

\subsection{Localization of NP operator}

We decompose the NP operator $\NP^*$ on $\partial D\subset \Rbb^d$ into several operators with different smoothing properties. The decomposition will be done stepwise. This subsection presents the first step to decompose $\NP^*$ into diagonal and off-diagonal terms. Fix a finite collection $\{ \Phi_\iota: V_\iota \to \partial D\}_{\iota=1}^N$ of local parametrization $\Phi_\iota: V_\iota \to \partial D$ with the $C^{k, \theta}$-smoothness such that $\bigcup_{\iota=1}^N \Phi_\iota (V_\iota)=\partial D$, where $V_\iota\in \Rbb^{d-1}$ is bounded open set in $\Rbb^{d-1}$. We take a partition of unity $\{\kappa_\iota\}_{\iota=1}^N\subset C^{k, \theta}(\partial D)$ on $\partial D$ subordinated to the open covering $\{\Phi_\iota (V_\iota)\}_{\iota=1}^N$. We also take functions $\chi_\iota\in C^{k, \theta}(\partial D)$ such that $\supp \chi_\iota\subset \Phi_\iota (V_\iota)$ and $\chi_\iota=1$ near $\supp \kappa_\iota$. Then we define
\begin{align}\label{eq:Ldiag}
    \Lcal_{\mathrm{diag}}[f]&:=\sum_{\iota=1}^N \kappa_\iota \NP^* [\chi_\iota f], \\
    \Lcal_{\mathrm{off}}&:=\NP^*-\Lcal_{\mathrm{diag}}. \notag
\end{align}

\begin{lemm}
    \label{lemm:off:diag}
    If $\partial D$ is $C^{k, \theta}$ for some $k\in \Nbb$ and $\theta\in [0, 1]$, then the integral kernel of $\Lcal_{\mathrm{off}}$ belongs to $C^{k-1, \theta}(\partial D\times \partial D)$.

    In particular, $\Lcal_\mathrm{off} \in \Bbb (L^1 (\partial D), C^{k-1, \theta}(\partial D))$.
\end{lemm}

\begin{proof}
    Since
    \[
        \Lcal_{\mathrm{off}}[f]=\sum_{\iota=1}^N \kappa_\iota \NP^* [(1-\chi_\iota)f]
    \]
    and $\supp \kappa_\iota \cap \supp (1-\chi_\iota)=\emptyset$, the integral kernel of $\Lcal_{\mathrm{off}}$ vanishes near the diagonal. Thus, the conclusion immediately follows from the $C^{k-1, \theta}$-regularity of the integral kernel of $\NP^*$ away from the diagonal.
\end{proof}

Let
\begin{equation}
    \label{eq:Q:defi}
    \begin{aligned}
    \Qcal_\iota [h](t):=&\,\kappa_\iota (\Phi_\iota (t))\NP^* [\chi_\iota (h\circ \Phi_\iota^{-1})](\Phi_\iota (t))
\end{aligned}
\end{equation}
so that 
\[
    \Lcal_{\mathrm{diag}}[f](x)=\sum_{V_\iota \ni x} \Qcal_\iota [f\circ \Phi_\iota](\Phi_\iota^{-1}(x))
\]
for $x\in \partial D$. We immediately see that 
\[
    \Qcal_\iota [h](t)=\frac{1}{\omega_d}\int_{V_\iota} Q_\iota (t, s)h(s)\,\df s,
\]
where 
\begin{align*}
        Q_\iota (t, s):=\frac{\nv_{\Phi_\iota (t)}\cdot (\Phi_\iota (t)-\Phi_\iota (s))}{|\Phi_\iota (t)-\Phi_\iota (s)|^d}\kappa_\iota (\Phi_\iota (t))\chi_\iota (\Phi_\iota (s))\sqrt{\det \fff_\iota (s)}.
\end{align*}

Since $Q_\iota$ is compactly supported in $V_\iota\times V_\iota$, we naturally define $Q_\iota=0$ outside $V_\iota \times V_\iota$. If $\partial D$ is $C^{1, \theta}$ for some $\theta\in (0, 1]$, then we have 
\[
    |Q_\iota (t, s)|\lesssim \frac{1}{|t-s|^{d-1-\theta}}
\]
for $(t, s)\in \Rbb^{d-1}$, and hence $\Qcal_\iota$ maps $C_\mathrm{c}^\infty (\Rbb^{d-1})$ into $L_\mathrm{loc}^1 (\Rbb^{d-1})$. 

\subsection{Weakly singular integral operators}

In order to estimate the mapping properties of $\Qcal_\iota$, we consider the following class of weakly singular integral operators. 

\begin{defi}\label{defi:wSIO}
    For $\theta>0$ and $\delta\in \Rbb$, we define the class $\Sigma^\theta_\delta$ of weakly singular integral operators as the class of integral operators
    \begin{equation}\label{eq:SIO:int}
        \Tcal [f](x)=\int_{\Rbb^n} T(x, y)f(y)\,\df y
    \end{equation}
    with the integral kernel $T(x, y)\in L^1 (\Rbb^n \times \Rbb^n)$ satisfying the following conditions:
    \begin{enumerate}[label=\textnormal{(\alph*)}]
        \item \label{enum:wSIO:supp}$T(x, y)$ is compactly supported. 
        \item \label{enum:wSIO:int}$\displaystyle |T(x, y)|\lesssim \frac{1}{|x-y|^{n-\theta}}$ whenever $x\neq y$;
        \item \label{enum:wSIO:der}$\displaystyle |T(x, y)-T(z, y)|\lesssim \frac{|x-z|^{\theta+\delta}}{|x-y|^{n+\delta}}$ whenever $2|x-z|<|x-y|$.
    \end{enumerate}
\end{defi}

Here, keep in mind that the space dimension $n$ is for that of the boundary $\partial D$, whence we will employ theorem in this subsection with $n=d-1$ later.

\begin{theo}\label{theo:SIO:bdd:Sobolev:0}
    If $\Tcal\in \Sigma^\theta_\delta$ for some $\theta\in (0, 1)$ and $\delta\geq 0$, then 
    \[
        \Tcal \in \Bbb (L^p (\Rbb^n), B^{\theta-\Ge, p} (\Rbb^n))
    \]
    for any $1<p<\infty$ and $\Ge>0$.
\end{theo}

\begin{proof}
    Take an open ball $U\subset \Rbb^{d-1}$ such that $\supp T\subset U\times U$. Then, by the property \ref{enum:wSIO:int} for $\Tcal$, we have 
    \begin{align*}
        \sup_{x\in \Rbb^n}\int_{\Rbb^n}|T(x, y)|\,\df  y
        \lesssim \sup_{x\in U}\int_U \frac{\df  y}{|x-y|^{n-\theta}}<\infty,
    \end{align*}
    and similarly
    \[
        \sup_{y\in \Rbb^n}\int_{\Rbb^n} |T(x, y)|\,\df x<\infty.
    \]
    Thus, by the Schur test, we obtain 
    \begin{equation}\label{eq:wSIO:Lp}
        \|\Tcal [f]\|_{L^p(\Rbb^n)}\lesssim \|f\|_{L^p(\Rbb^n)}.
    \end{equation}

    Next, we prove the estimate
    \begin{equation}\label{eq:Besov:seminorm:part}
        \int_{U\times U}\frac{|\Tcal [f](x)-\Tcal [f](z)|^{p}}{|x-z|^{n+(\theta-\Ge)p}}\, \df x \df z\lesssim \|f\|_{L^p(\Rbb^n)}
    \end{equation}
    for $f\in L^p (\Rbb^n)$ and sufficiently small $\Ge>0$. For $x, z\in \Rbb^n$, we set
    \begin{equation}\label{eq:Axz}
        A_{x, z}:=\{ y\in \Rbb^n \mid 2|x-z|< |x-y|\}.
    \end{equation}
    Then, since $\supp T\subset U\times U$, we have 
    \begin{align*}
        |\Tcal [f](x)-\Tcal [f](z)|
        &\leq \int_{U\cap A_{x, z}}|T(x, y)-T(z, y)||f (y)|\, \df y \\
        &\quad+\int_{U\setminus A_{x, z}} |T(x, y)-T(z, y)||f (y)|\, \df y \\
        &=:I_1+I_2,
    \end{align*}
    and hence
    \begin{equation}
        \label{eq:dTf:0}
        |\Tcal [f](x)-\Tcal [f](z)|^p\lesssim I_1^p+I_2^p.
    \end{equation}
    
    Let $q=p/(p-1)$ be the H\"older conjugate of $p$. Then, by the H\"older inequality, $I_1^p$ is estimated as
    \begin{align*}
        I_1^p &\lesssim |x-z|^{(\theta+\delta)p}\left(\int_{U\cap A_{x, z}} \frac{|f(y)|}{|x-y|^{n+\delta}}\,\df y\right)^p \\
        &\leq |x-z|^{(\theta+\delta)p} \left(\int_{U\cap A_{x, z}}\frac{\df y}{|x-y|^{n+\Ge q/p}}\right)^{p/q}\int_{U\cap A_{x,z}}\frac{|f(y)|^p}{|x-y|^{n-\Ge+p\delta}}\,\df y \\
        &\lesssim |x-z|^{(\theta+\delta)p-\Ge} \int_{U\cap A_{x,z}}\frac{|f(y)|^p}{|x-y|^{n-\Ge +p\delta}}\,\df y.
    \end{align*}
    Note that, if $y\in \Rbb^n\setminus A_{x, z}$, then $|z-y|\leq 3|x-z|$. Thus we have  
    \begin{align*}
        I_2^p&\leq \left(\int_{U\setminus A_{x, z}} (|T(x, y)|+|T(z, y)|)|f(y)|\,\df y\right)^p \\
        &\lesssim \left(\int_{U\setminus A_{x, z}} \frac{|f(y)|}{|x-y|^{n-\theta}}\,\df y\right)^p+\left(\int_{\substack{y\in U \\ |z-y|\leq 3|x-z|}} \frac{|f(y)|}{|z-y|^{n-\theta}}\,\df y\right)^p \\
        &\lesssim \left(\int_{U\setminus A_{x, z}} \frac{\df y}{|x-y|^{n-\Ge q/p}}\right)^{p/q}\int_{U\setminus A_{x, z}} \frac{|f(y)|^p}{|x-y|^{n-\theta p+\Ge}}\,\df y \\
        &\quad +\left(\int_{\substack{y\in U \\ |z-y|\leq 3|x-z|}} \frac{\df y}{|z-y|^{n-\Ge q/p}}\right)^{p/q} \int_{\substack{y\in U \\ |z-y|\leq 3|x-z|}} \frac{|f(y)|^p}{|z-y|^{n-\theta p+\Ge}}\,\df y \\
        &\lesssim |x-z|^\Ge\int_{U\setminus A_{x, z}} \frac{|f(y)|^p}{|x-y|^{n-\theta p+\Ge}}\,\df y \\
        &\quad +|x-z|^\Ge \int_{\substack{y\in U \\ |z-y|\leq 3|x-z|}} \frac{|f(y)|^p}{|z-y|^{n-\theta p+\Ge}}\,\df y.
    \end{align*}
    We then have
    \begin{align*}
        &\int_{U\times U}\frac{|\Tcal [f](x)-\Tcal [f](z)|^p}{|x-z|^{n+(\theta-\Ge) p}}\,\df x\df z \\
        &\lesssim \int_U \df y\, |f(y)|^p \int_U \df x \int_{\substack{z\in U \\ 2|x-z|<|x-y|}} \frac{\df z}{|x-z|^{n-\Ge p-\delta p+\Ge}} \\
        &\quad +\int_U \df y\, |f(y)|^p\int_U \frac{\df x}{|x-y|^{n-\theta p+\Ge}} \int_{\substack{ z\in U \\ |x-z|\geq 2|x-y|}}\frac{\df z}{|x-z|^{n+(\theta-\Ge)p-\Ge}} \\
        &\quad +\int_U \df y \, |f(y)|^p \int_U \frac{\df z}{|z-y|^{n-\theta p+\Ge}} \int_{\substack{x\in U \\ |z-y|\leq 3|x-z|}} \frac{\df x}{|x-z|^{n+(\theta-\Ge)p-\Ge}} \\
        &\lesssim \|f\|_{L^p}^p
        +\int_U \df y\, |f(y)|^p\int_U \frac{\df x}{|x-y|^{n-\Ge p}}
        +\int_U \df y \, |f(y)|^p \int_U \frac{\df z}{|z-y|^{n-\Ge p}} \\
        &\lesssim \|f\|_{L^p}^p.
    \end{align*}
    This completes the proof of \eqref{eq:Besov:seminorm:part}. 

    Recall from \eqref{eq:Besov:difference} that
    \begin{equation}\label{eq:Besov:seminorm:full}
        [\Tcal [f]]_{B^{\theta-\Ge, p}(\Rbb^n)}^p=\int_{\Rbb^n\times \Rbb^n}\frac{|\Tcal [f](x)-\Tcal [f](z)|^{p}}{|x-z|^{n+(\theta-\Ge)p}}\, \df x \df z.
    \end{equation}
    We denote the integrand in the right hand side of \eqref{eq:Besov:seminorm:full} by $A(x, z)$.
    Since $\supp T\subset U\times U$, we obtain 
    \begin{align*}
        &[\Tcal [f]]_{B^{\theta-\Ge, p}(\Rbb^n)}^p \\
        &=\int_{U\times U}A(x, z)\, \df x \df z
        +\int_{U\times U^c} A(x, z)\, \df x \df z+\int_{U^c\times U}A(x, z)\, \df x \df z \\
        &=:J_1+J_2+J_3.
    \end{align*}
    It remains to prove that $J_a$ ($a=1, 2, 3$) are bounded by $\|f\|_{L^p (\Rbb^n)}$. The estimate for $J_1$ is already proved in \eqref{eq:Besov:seminorm:part}. To estimate $J_2$ and $J_3$, we note that 
    \[
        J_2=\int_{U\times U^c}\frac{|\Tcal [f](x)|^{p}}{|x-z|^{n+(\theta-\Ge)p}}\, \df x \df z, \quad 
        J_3=\int_{U^c\times U}\frac{|\Tcal [f](z)|^{p}}{|x-z|^{n+(\theta-\Ge)p}}\, \df x \df z,
    \]
    which follow from $\supp T\subset U\times U$. 

    Since the distance between $\pi (\supp T)$ and $U^c$ is positive, where $\pi (x, y):=x$ ($(x, y)\in \Rbb^n\times \Rbb^n$) is the projection with respect to the first component, we have $|x-z|\geq c$ on $(U\times U^c)\cup (U^c\times U)$ for some constant $c>0$. Thus we obtain from \eqref{eq:wSIO:Lp} that  
    \begin{align*}
        J_2\leq \int_U \df x\, |\Tcal [f](x)|^p \int_{|x-z|\geq c}\frac{\df z}{|x-z|^{n+(\theta-\Ge)p}}\lesssim \int_U |\Tcal [f](x)|^p\, \df x\lesssim \|f\|_{L^p(\Rbb^n)}^p
    \end{align*}
    and similarly $J_3\lesssim \|f\|_{L^p(\Rbb^n)}^p$. This completes the proof. 
\end{proof}

\subsection{Mapping property on lower-order Besov spaces}

We now assume that $\partial D$ is $C^{2, \theta}$ for some $\theta\in (0, 1]$. In order to decompose the operator $\Qcal_\iota$ into two parts with different smoothing properties, let
\[
    \begin{aligned}
        \fff_\iota (t)&:=\left(\frac{\partial\Phi_\iota}{\partial t_j}(t)\cdot \frac{\partial\Phi_\iota}{\partial t_k} (t)\right)_{j, k=1}^{d-1}, \\
        \sff_\iota(t)&:=\left(\nv_{\Phi_\iota (t)}\cdot \frac{\partial^2 \Phi_\iota}{\partial t_j\partial t_k}(t)\right)_{j, k=1}^{d-1}
    \end{aligned}
\]
be, respectively, the first fundamental form and the second fundamental form on $\partial D$ at $\Phi_\iota (t)$ (locally represented by the parametrization $\Phi_\iota$). The first fundamental form induces the natural norm
    \[
        |v|_{\fff_\iota (t)}:=\sqrt{v\cdot \fff_\iota (t)v}
    \]
    for $v=(v_j)_{j=1}^{d-1} \in \Rbb^{d-1}$.

    Then, we define
    \begin{equation}
        \label{eq:NP:decomp}
        \begin{aligned}
            \Pcal_\iota [h](t):=
            \frac{\sqrt{\det \fff_\iota (t)}}{2\omega_d} \kappa_\iota (\Phi_\iota (t))\int_{V_\iota} \frac{(t-s)\cdot \sff_\iota (t)(t-s)}{|t-s|_{\fff_\iota (t)}^{d}}\chi_\iota(\Phi_\iota (s))h(s)\,\df s
        \end{aligned}
    \end{equation}
    for $t\in V_\iota$ and $h: \Rbb^{d-1}\to \Cbb$. Since $\supp \Pcal_\iota [h]\subset V_\iota$, we naturally extend $\Pcal_\iota [h]=0$ outside $V_\iota$. We also define the remainder term
    \begin{equation}\label{eq:R:op}
        \Rcal_\iota:=\Qcal_\iota- \Pcal_\iota.
    \end{equation}

    When $d=2$, the operator $\Pcal_\iota$ is more concisely expressed as
    \begin{equation}
        \label{eq:P:2d}
        \Pcal_\iota [h](t)=
            \frac{\kappa_\iota (\Phi_\iota (t))\sff_\iota (t)}{4\pi\sqrt{\fff_\iota (t)}}\int_{V_\iota} \chi_\iota(\Phi_\iota (s))h(s)\,\df s.
    \end{equation}
    Thus we have 
        \begin{align*}
            \| \Pcal_\iota [h]\|_{B^{\theta, p}(\Rbb^{d-1})}
            &=\left|\int_{V_\iota} \chi_\iota(\Phi_\iota (s))h(s)\,\df s\right|\left\|\frac{(\kappa_\iota \circ \Phi_\iota)\sff_\iota}{4\pi\sqrt{\fff_\iota}}\right\|_{B^{\theta, p}(\Rbb^{d-1})} \\
            &\lesssim \|h\|_{B^{-2-\theta, p}(\Rbb^{d-1})}.
        \end{align*}

    On the other hand, when $d\geq 3$, the operator $\Pcal_\iota$ is represented as $\Pcal_\iota [h]=a_\iota (t, -\iu \partial_t)[(\chi_\iota\circ \Phi_\iota)h]$, where
    \[
        a_\iota (t, -\iu \partial_t)[h](t)=\frac{1}{(2\pi)^{d-1}}\int_{\Rbb^{d-1}}\df \xi \int_{\Rbb^{d-1}} \df s\, a_\iota (t, \xi)\e^{\iu \xi \cdot (t-s)}h(s)
    \]
    is the pseudodifferential operator with symbol
    \begin{align*}
        a_\iota (t, \xi):=&\,\frac{\sqrt{\det \fff_\iota (t)}}{2\omega_d} \kappa_\iota (\Phi_\iota (t))\Fcal_{v\to \xi}\left[\frac{v\cdot \sff_\iota (t)v}{|v|_{\fff_\iota (t)}^{d}}\right](\xi) \\
        =&\,\frac{\kappa_\iota (\Phi_\iota (t))}{4}\left(\frac{\tr A_\iota (t)}{|\xi|_{\fff_\iota(t)^{-1}}}-\frac{\xi\cdot \fff_\iota(t)^{-1}A_\iota (t)\xi}{|\xi|_{\fff_\iota(t)^{-1}}^3}\right)
    \end{align*}
    for $(t, \xi)\in V_\iota \times (\Rbb^{d-1}\setminus \{0\})$ and $a_\iota (t, \xi)=0$ otherwise. 
    Here $|\xi|_{\fff_\iota (t)^{-1}}=\sqrt{\xi\cdot \fff_\iota (t)^{-1}\xi}$ and $A_\iota (t)=\sff_\iota (t)\fff_\iota (t)^{-1}$ is the representation matrix of the shape operator on $\partial D$ at $\Phi_\iota (t)$ associated with the basis $\partial_{t_1}, \ldots, \partial_{t_{d-1}}$ of the tangent space $T_{\Phi_\iota (t)}\partial D$. (See \cite{Miyanishi22} for three-dimensional case for example. Same calculation is valid for higher-dimensional cases.)

    Since the symbol $a_\iota (t, \xi)$ is infinitely smooth in $\xi\neq 0$ and admits the estimate
    \[
        \|\partial_\xi^\alpha a_\iota (\cdot, \xi)\|_{C^{0, \theta}} =O(|\xi|^{-1-|\alpha|}) \quad (|\xi|\to \infty)
    \]
    for any multi-index $\alpha\in \Nbb_0^{d-1}$, we can apply the boundedness of pseudodifferential operators on Besov spaces to obtain the following proposition (\cite{Marschall87}, see also \cite[Proposition 4.5]{Escher-Seiler08}):

    \begin{prop}\label{prop:psido:mapping}
        If $\partial D\subset \Rbb^d$ ($d\geq 2$) is $C^{2, \theta}$ for some $\theta\in (0, 1]$, then
        \[
            \Pcal_\iota \in \Bbb (B^{s-1, p} (\Rbb^{d-1}), B^{s, p} (\Rbb^{d-1}))
        \]
        for any $s\in (-\theta, \theta)$ and $p\in (1, \infty)$.
    \end{prop}

    Next, we investigate the remainder term $\Rcal_\iota$ defined in \eqref{eq:R:op}. We recall that $\Sigma^\theta_\delta$ is the class of weakly singular integral operators defined in Definition \ref{defi:wSIO}. 

    \begin{lemm}
        \label{lemm:R:SIO}
        If $\partial D$ is $C^{2, \theta}$ for some $\theta\in (0, 1]$, then $\Rcal_\iota\in \Sigma^{1+\theta}_{-1}$ and $\Rcal_\iota \partial_t \in \Sigma^\theta_0$.
    \end{lemm}

        \begin{proof}
            The operator $\Rcal_\iota$ is represented as
            \[
                \Rcal_\iota [h](t)=\sum_{a=1}^3\int_B R_a (t, s)h(s)\,\df s,
            \]
            where
        \begin{align*}
            R_1(t, s)
            &:= \frac{\nv_{\Phi_\iota (t)}\cdot (\Phi_\iota (t)-\Phi_\iota (s))+(t-s)\cdot \sff_\iota (t)(t-s)/2}{\omega_d|t-s|_{\fff_\iota (t)}^{d}} \\
            &\quad\times \kappa_\iota (\Phi_\iota (t))\chi_\iota (\Phi_\iota (s))\sqrt{\det \fff_\iota (s)}, \\
            R_2(t, s)&:=\frac{\nv_{\Phi_\iota (t)}\cdot (\Phi_\iota (t)-\Phi_\iota (s))\left(|\Phi_\iota (t)-\Phi_\iota (s)|^{d}-|t-s|_{\fff_\iota(t)}^{d}\right)}{2\omega_d|t-s|_{\fff_\iota(t)}^{d}|\Phi_\iota (t)-\Phi_\iota (s)|^{d}} \\
            &\quad \times \kappa_\iota (\Phi_\iota (t))\chi_\iota (\Phi_\iota (s))\sqrt{\det \fff_\iota (s)}
        \end{align*}
        and
        \begin{align*}
            R_3 (t, s):=-\frac{(t-s)\cdot \sff_\iota (t)(t-s)}{2\omega_d|t-s|_{\fff_\iota (t)}^{d}}\left(\sqrt{\det \fff_\iota (t)}-\sqrt{\det \fff_\iota (s)}\right)\kappa_\iota (\Phi_\iota (t))\chi_\iota (\Phi_\iota(s)).
        \end{align*}
        Since $\supp R_a\subset V_\iota \times V_\iota$, we extend them as $R_a=0$ outside $V_\iota \times V_\iota$. 

        Then, one can prove the following estimates for $a=1$, $2$, $3$:
            \begin{align}
                |R_a (t, s)|&\lesssim \frac{1}{|t-s|^{n-1-\theta}}, \label{eq:Rk:0}\\
                |R_a (t, s)-R_a (t^\prime, s)|&\lesssim
                \frac{|t-t^\prime|^\theta}{|t-s|^{n-1}} \quad (2|t-t^\prime|<|t-s|), \label{eq:Rk:1} \\
                |\partial_s R_a (t, s)|&\lesssim
                \frac{1}{|t-s|^{n-\theta}}, \label{eq:Rk:0k} \\
                |\partial_s R_a (t, s)-\partial_s R_a (t^\prime, s)|&\lesssim
                \frac{|t-t^\prime|^\theta}{|t-s|^n} \quad (2|t-t^\prime|<|t-s|), \label{eq:Rk:1k}
            \end{align}
            The estimates \eqref{eq:Rk:0} and \eqref{eq:Rk:1} together with the compactness of $\supp R_a$ ($a=1, 2, 3$) prove $\Rcal_\iota \in \Sigma^{1+\theta}_{-1}$. Since
            \[
                \Rcal_\iota [\partial_{t_j} h](t)=-\sum_{a=1}^3 \int_{\Rbb^n} \partial_{s_j} R_a (t, s)h(s)\,\df s
            \]
            by integration by parts, which is justified by virtue of \eqref{eq:Rk:0} and \eqref{eq:Rk:0k}, the estimates \eqref{eq:Rk:0k} and \eqref{eq:Rk:1k} prove $\Rcal_\iota \partial_t\in \Sigma^\theta_0$.
        \end{proof}

    \begin{lemm}
        \label{lemm:NP:decomp:2}
        If $\partial D$ is $C^{2, \theta}$ for some $\theta\in (0, 1]$, then 
        \[
            \Rcal_\iota \in \Bbb (H^{-1, p} (\Rbb^{d-1}), B^{\theta-\Ge, p}(\Rbb^{d-1}))
        \] 
        for any $p\in (1, \infty)$ and $\Ge>0$.
    \end{lemm}

    \begin{proof}
         By virtue of the inclusion relation $\Sigma^{1+\theta}_{-1}\subset \Sigma^\theta_0$ and Lemma \ref{lemm:R:SIO}, we can apply Theorem \ref{theo:SIO:bdd:Sobolev:0} to obtain $\Rcal_\iota, \Rcal_\iota \partial_{t_j}\in \Bbb (L^p (\Rbb^n), B^{\theta-\Ge, p} (\Rbb^n))$ for $p\in (1, \infty)$ and $\Ge>0$. We take their adjoints to obtain $\Rcal_\iota^*, \partial_{t_j}\Rcal_\iota^*\in \Bbb (B^{-\theta+\Ge, q} (\Rbb^n), L^q (\Rbb^n))$, where $q=p/(p-1)$ is the H\"older conjugate of $p$. Thus $\Rcal_\iota^* \in \Bbb (B^{-\theta+\Ge, q} (\Rbb^n), H^{1, q}(\Rbb^n))$, from which we have $\Rcal_\iota \in \Bbb (H^{-1, p} (\Rbb^n), B^{\theta-\Ge, p} (\Rbb^n))$.
    \end{proof}

    The preceding Proposition \ref{prop:psido:mapping} and Lemma \ref{lemm:R:SIO} immediately prove the following mapping property of $\Qcal_\iota$.

    \begin{prop}
        \label{prop:Sobolev:low:order}
        If $\partial D$ is $C^{2, \theta}$ for some $\theta\in (0, 1]$, then 
        \[
            \Qcal_\iota\in \Bbb (B^{s-1, p}(\Rbb^n), B^{s, p}(\Rbb^n))
        \] 
        for any $p\in (1, \infty)$ and $s\in [0, \theta)$.
    \end{prop}

\subsection{Commutator estimates}

We derive commutator estimates in order to derive mapping properties between higher-order Sobolev spaces from those between lower-order ones. We introduce the adjoint action $\ad_\Bcal \Acal:=\Bcal\Acal-\Acal\Bcal$ for linear operators $\Acal$ and $\Bcal$ and define
\[
    \ad_\partial^\alpha:=(\ad_{\partial_{t_1}})^{\alpha_1}\cdots (\ad_{\partial_{t_{d-1}}})^{\alpha_{d-1}}
\]
for any multi-index $\alpha=(\alpha_1, \ldots, \alpha_{d-1})$. We let $\ad_\partial^\alpha$ act on the operator $\Qcal_\iota$ defined in \eqref{eq:Q:defi}. If we formally apply the integration by parts, we obtain 
\[
    \ad_{\partial_{t_k}} \Qcal_\iota [h](t)
        =\frac{1}{\omega_d}\int_{V_\iota} [(\partial_{t_k}+\partial_{s_k}) Q_\iota (t, s)]h(s)\,\df s.
\]
In order to justify the above calculation, we need to prove that singularity of $(\partial_{t_k}+\partial_{s_k}) Q_\iota (t, s)$ on $t=s$ is weak. Iterating the above procedure, we expect that 
\begin{equation}\label{eq:ad:ker}
        \ad_\partial^\alpha \Qcal_\iota [h](t)
        =\frac{1}{\omega_d}\int_{V_\iota} [(\partial_t+\partial_s)^\alpha Q_\iota (t, s)]h(s)\,\df s.
    \end{equation}
    What is crucial is to investigate the singularity of $(\partial_t+\partial_s)^\alpha Q_\iota (t, s)$ on $t=s$. The result is summarized in the following lemma.

\begin{lemm}
    \label{lemm:Q:est}
    If $\partial D$ is $C^{k, \theta}$ for some $k\geq 1$ and $\theta\in (0, 1]$, we have $\ad_\partial^\alpha \Qcal_\iota \in \Sigma^\theta_0$ for any $|\alpha|\leq k-1$.

    In particular, $\ad_\partial^\alpha \Qcal_\iota \in \Bbb (L^p (\Rbb^n), B^{\theta-\Ge, p}(\Rbb^n))$ for any $p\in (1, \infty)$ and $\Ge>0$.
\end{lemm}

\begin{proof}

    We aim to prove that $(\partial_t+\partial_s)^\alpha Q_\iota (t, s)$ satisfies \ref{enum:wSIO:int} and \ref{enum:wSIO:der} in Definition \ref{defi:wSIO} with $n=d-1$. To do so, let
    \[
        Q^1_\iota (t, s):=\nv_{\Phi_\iota (t)}\cdot (\Phi_\iota (t)-\Phi_\iota (s))\kappa_\iota (\Phi_\iota (t))\chi_\iota (\Phi_\iota (s))\sqrt{\det \fff_\iota (s)}
    \]
    and
    \[
        Q^2_\iota (t, s):=|\Phi_\iota (t)-\Phi_\iota (s)|^d.
    \]
    Then
    \[
        Q_\iota (t, s)=\frac{Q^1_\iota (t, s)}{\omega_dQ^2_\iota (t, s)}.
    \]
    Since $\nv_{\Phi_\iota (t)}\cdot \partial_{t_j}\Phi_\iota (t)=0$,
    \begin{align*}
        &\nv_{\Phi_\iota (t)}\cdot (\Phi_\iota (t)-\Phi_\iota (s)) \\
        &=\sum_{j=1}^{d-1} (t_j-s_j)\nv_{\Phi_\iota (t)}\cdot \int_0^1 (\partial_{t_j}\Phi_\iota (t-\mu (t-s))-\partial_{t_j}\Phi_\iota (t))\,\df \mu.
    \end{align*}
    So, we can show by iteration that
    \begin{align*}
        |(\partial_t+\partial_s)^\alpha Q^1_\iota (t, s)|
        &\lesssim
        |t-s|^{1+\theta}, \\
        |(\partial_t+\partial_s)^\alpha Q^1_\iota (t, s)-(\partial_t+\partial_s)^\alpha Q^1_\iota (t^\prime, s)|
        &\lesssim |t-t^\prime|^\theta |t-s| \quad (2|t-t^\prime|<|t-s|)
    \end{align*}
    for any $|\alpha|\leq k-1$.

    On the other hand, the differentials of the reciprocal of $Q^2_\iota (t, v)$ is calculated as
    \begin{align*}
        (\partial_t+\partial_s)^\alpha \left(\frac{1}{Q^2_\iota (t, s)}\right)
        &=\sum_{l=1}^{|\alpha|} (-1)^l\left(\frac{d}{2}\right)_l \frac{1}{|\Phi_\iota (t)-\Phi_\iota (s)|^{d+2l}}\\
        &\quad\times \sum_{\substack{\alpha^{(1)}+\cdots+\alpha^{(l)}=\alpha \\ |\alpha^{(1)}|, \ldots, |\alpha^{(l)}|\geq 1}} \prod_{j=1}^l[(\partial_t+\partial_s)^{\alpha^{(j)}}|\Phi_\iota (t)-\Phi_\iota (s)|^2]
    \end{align*}
    for $|\alpha|\leq k-1$, where
    \[
        (x)_l=x(x+1)\cdots (x+l-1)
    \]
    is the Pochhammer symbol. Then, for $|\alpha|\leq k-1$ and $l=1, \ldots, |\alpha|$, we obtain the estimates
    \begin{align*}
    \prod_{j=1}^l\left|\partial_t^{\alpha^{(j)}}| \Phi_\iota (t)-\Phi_\iota (s)|^2\right|
        \lesssim \prod_{j=1}^l |t-s|^2 \lesssim |t-s|^{2l}
    \end{align*}
    and
    \begin{align*}
        &\left|\prod_{j=1}^l [(\partial_t+\partial_s)^{\alpha^{(j)}}|\Phi_\iota (t)-\Phi_\iota (s)|^2]-\prod_{j=1}^l [(\partial_{t^\prime}+\partial_s)^{\alpha^{(j)}}|\Phi_\iota (t^\prime)-\Phi_\iota (s)|^2]\right| \\
        &=\left| \sum_{m=1}^l \sum_{\beta\leq \alpha^{(m)}}\begin{pmatrix}
            \alpha^{(m)} \\
            \beta
        \end{pmatrix}
        (\partial_t^\beta\Phi_\iota (t)+\partial_t^\beta\Phi_\iota (t^\prime)-2\partial_t^\beta\Phi_\iota (s)) \right. \\
        &\quad\cdot (\partial_t^{\alpha^{(m)}-\beta}\Phi_\iota (t)-\partial_t^{\alpha^{(m)}-\beta}\Phi_\iota (t^\prime)) \\
        &\quad\left.\times \prod_{j<m} [(\partial_t+\partial_s)^{\alpha^{(j)}}|\Phi_\iota (t)-\Phi_\iota (s)|^2]\times \prod_{j>m} [(\partial_{t^\prime}+\partial_s)^{\alpha^{(j)}}|\Phi_\iota (t^\prime)-\Phi_\iota (s)|^2]\right| \\
        &\lesssim |t-t^\prime||t-s|^{2l-1} \quad (\text{if } 2|t-t^\prime|<|t-s|).
    \end{align*}
    Thus we obtain the estimates
    \begin{align*}
        \left|\partial_t^\alpha \left(\frac{1}{Q^2_\iota (t, s)}\right)\right|
        &\lesssim \sum_{l=1}^{|\alpha|}\frac{1}{|t-s|^{d+2l}}\times |t-s|^{2l}\lesssim \frac{1}{|t-s|^d}
    \end{align*}
    and
    \begin{align*}
        &\left|(\partial_t+\partial_s)^\alpha \left(\frac{1}{Q^2_\iota (t, s)}\right)-(\partial_{t^\prime}+\partial_s)^\alpha \left(\frac{1}{Q^2_\iota (t^\prime, s)}\right)\right| \\
        &\lesssim \sum_{l=1}^{|\alpha|}\left|\frac{1}{|\Phi_\iota (t)-\Phi_\iota (s)|^d}-\frac{1}{|\Phi_\iota (t^\prime)-\Phi_\iota (s)|^d}\right|\times |t-s|^{2l} \\
        &\quad + \sum_{l=1}^{|\alpha|} \frac{1}{|t-s|^d}\times |s||t-s|^{2l-1}
        \lesssim \frac{|t-t^\prime|}{|t-s|^{d-1}} \quad (2|t-t^\prime|<|t-s|).
    \end{align*}

    Combining these estimates by the Leibnitz rule, we obtain that $(\partial_t+\partial_s)^\alpha Q_\iota (t, s)$ satisfies \ref{enum:wSIO:int} and \ref{enum:wSIO:der} in Definition \ref{defi:wSIO} with $n=d-1$ and $\delta=0$.

    Thus the formula \eqref{eq:ad:ker} is justified for $\alpha\in \Nbb_0^{d-1}$ with $|\alpha|\leq k-1$. Hence $(\partial_t+\partial_s)^\alpha Q_\iota$ is actually the integral kernel of $\ad^\alpha_\partial \Qcal_\iota$. Since $\supp (\partial_t+\partial_s)^\alpha Q_\iota \subset V_\iota \times V_\iota$, we infer that $\ad_\partial^\alpha \Qcal_\iota\in \Sigma^\theta_0$. 

    Now the boundedness of $\ad_\partial^\alpha \Qcal_\iota: L^p (\Rbb^{d-1})\to  B^{\theta-\Ge, p}(\Rbb^{d-1})$ immediately follows from Theorem \ref{theo:SIO:bdd:Sobolev:0}.
\end{proof}

Now we are ready to prove the smoothing property of NP operators.

    \begin{proof}[Proof of Theorem \ref{theo:layer:Sobolev} \ref{enum:np:smoothing:1}]
        Let $\alpha\in \Nbb_0^{d-1}$ be a multi-index with $|\alpha|\leq k-1$. Then a direct calculation proves that
        \[
            \partial_t^\alpha \Qcal_\iota
            =\Qcal_\iota \partial_t^\alpha
            +\sum_{\substack{\beta \leq \alpha \\ \beta\neq \alpha}}
            \begin{pmatrix}
                \alpha \\
                \beta
            \end{pmatrix}
            (\ad_\partial^{\alpha-\beta}\Qcal_\iota) \partial_t^\beta.
        \]
        Let $\Ge\in (0, \theta)$ and take an arbitrary $h\in C^{k, \theta}(\Rbb^n)$ such that $\supp h\subset V_\iota$. Then, by Proposition \ref{prop:Sobolev:low:order} and Lemma \ref{lemm:Q:est}, we have
        \begin{align*}
            &\|\partial_t^\alpha \Qcal_\iota [h]\|_{B^{\theta-\Ge, p}(\Rbb^{d-1})} \\
            &\lesssim \|\Qcal_\iota [\partial_t^\alpha h]\|_{B^{\theta-\Ge, p}(\Rbb^{d-1})}+\sum_{\substack{\beta \leq \alpha \\ \beta\neq \alpha}} \|(\ad_\partial^{\alpha-\beta}\Qcal_\iota) [\partial_t^\beta h]\|_{B^{\theta-\Ge, p}(\Rbb^{d-1})} \\
            &\lesssim \|\partial_t^\alpha h\|_{B^{\theta-\Ge-1, p}(\Rbb^{d-1})}+\sum_{\substack{\beta \leq \alpha \\ \beta\neq \alpha}} \|\partial_t^\beta h\|_{L^p(\Rbb^{d-1})} \\
            &\lesssim \|h\|_{B^{k-2+\theta-\Ge}(\Rbb^{d-1})}=\|h\|_{B^{\overline{s}-2-\Ge, p}(\Rbb^{d-1})}.
        \end{align*}
        Here, the boundedness of $\ad_\partial^\beta \Qcal_\iota: L^p (\Rbb^{d-1})\to B^{\theta-\Ge, p}(\Rbb^{d-1})$ is still available even if $|\beta|< k-1$ since $C^{k, \theta}$-domain can be regarded as $C^{|\beta|+1, 1}$-domain when $|\beta|<k-1$. Thus the operator $\Qcal_\iota: B^{\overline{s}-2-\Ge, p} (\Rbb^{d-1})\to B^{\overline{s}-1-\Ge, p} (\Rbb^{d-1})$ is bounded for any $\Ge\in (0, \theta)$.

        By Lemma \ref{lemm:off:diag}, we obtain the boundedness of $\NP^*: B^{\overline{s}-2-\Ge, p} (\partial D)\to B^{\overline{s}-1-\Ge, p} (\partial D)$. Since $\NP=\SL\NP^*\SL^{-1}$ by \eqref{eq:Plemelj}, we infer from Theorem \ref{theo:layer:Sobolev} \ref{enum:sl:invertible} that 
        \[
            \NP : B^{\overline{s}-1-\Ge, p} (\partial D)\to  B^{\overline{s}-\Ge, p} (\partial D)
        \]
        is bounded. Here, when $d=2$, we may assume that $\capac_D>1$ by suitable dilation. Thus, by duality and interpolation, we see that $\NP^*: B^{s-1, p} (\partial D)\to B^{s, p} (\partial D)$ is bounded for $s\in (-\overline{s}+1, \overline{s}-1)$.
    \end{proof}

\bibliography{small_resistance_260317}

\begin{thebibliography}{10}

\bibitem{Ammari-Kang07}
H.~Ammari and H.~Kang.
\newblock {\em Polarization and moment tensors. {With} applications to inverse problems and effective medium theory}, volume 162 of {\em Applied Mathematical Sciences}.
\newblock Springer, New York, 2007.

\bibitem{AKLLL07}
H.~Ammari, H.~Kang, H.~Lee, J.~Lee, and M.~Lim.
\newblock Optimal estimates for the electric field in two dimensions.
\newblock {\em J. Math. Pures Appl. (9)}, 88(4):307--324, 2007.

\bibitem{BLY09}
E.~S. Bao, Y.~Y. Li, and B.~Yin.
\newblock Gradient estimates for the perfect conductivity problem.
\newblock {\em Arch. Ration. Mech. Anal.}, 193(1):195--226, 2009.

\bibitem{Benveniste-Miloh99}
Y.~Benveniste and T.~Miloh.
\newblock Neutral inhomogeneities in conduction phenomena.
\newblock {\em J. Mech. Phys. Solids}, 47(9):1873--1892, 1999.

\bibitem{Bergh-Loefstroem76}
J.~Bergh and J.~L\"{o}fstr\"{o}m.
\newblock {\em Interpolation spaces. {A}n introduction}.
\newblock Grundlehren der Mathematischen Wissenschaften, No. 223. Springer-Verlag, Berlin-New York, 1976.

\bibitem{BCF80}
H.~Br{\'e}zis, L.~A. Caffarelli, and A.~Friedman.
\newblock Reinforcement problems for elliptic equations and variational inequalities.
\newblock {\em Ann. Mat. Pura Appl. (4)}, 123:219--246, 1980.

\bibitem{CPR09}
L.~P. Castro, E.~Pesetskaya, and S.~V. Rogosin.
\newblock Effective conductivity of a composite material with non-ideal contact conditions.
\newblock {\em Complex Var. Elliptic Equ.}, 54(12):1085--1100, 2009.

\bibitem{Chang-Lee08}
T.~Chang and K.~Lee.
\newblock Spectral properties of the layer potentials on {L}ipschitz domains.
\newblock {\em Illinois J. Math.}, 52(2):463--472, 2008.

\bibitem{Dahlberg-Kenig87}
B.~E.~J. Dahlberg and C.~E. Kenig.
\newblock Hardy spaces and the {Neumann} problem in {{\(L^ p\)}} for {Laplace}'s equation in {Lipschitz} domains.
\newblock {\em Ann. Math. (2)}, 125:437--465, 1987.

\bibitem{DallaRiva-Musolino13}
M.~Dalla~Riva and P.~Musolino.
\newblock A singularly perturbed nonideal transmission problem and application to the effective conductivity of a periodic composite.
\newblock {\em SIAM J. Appl. Math.}, 73(1):24--46, 2013.

\bibitem{Dispa03}
S.~Dispa.
\newblock Intrinsic characterizations of {Besov} spaces on {Lipschitz} domains.
\newblock {\em Math. Nachr.}, 260:21--33, 2003.

\bibitem{Dondi-LdC17}
F.~Dondi and M.~Lanza~de Cristoforis.
\newblock Regularizing properties of the double layer potential of second order elliptic differential operators.
\newblock {\em Mem. Differ. Equ. Math. Phys.}, 71:69--110, 2017.

\bibitem{DLZ2510}
H.~Dong, H.~Li, and Y.~Zhao.
\newblock Optimal gradient estimates for conductivity problems with imperfect low-conductivity interfaces.
\newblock arXiv:2510.10615 [math.AP].

\bibitem{DLY24}
H.~Dong, Y.~Li, and Z.~Yang.
\newblock Gradient estimates for the insulated conductivity problem: the non-umbilical case.
\newblock {\em J. Math. Pures Appl.}, 189:103587, 2024.

\bibitem{DLY25}
H.~Dong, Y.~Li, and Z.~Yang.
\newblock Optimal gradient estimates of solutions to the insulated conductivity problem in dimension greater than two.
\newblock {\em J. Eur. Math. Soc.}, 27(8):3275--3296, 2025.

\bibitem{DYZ26}
H.~Dong, Z.~Yang, and H.~Zhu.
\newblock Gradient estimates for the conductivity problem with imperfect bonding interfaces.
\newblock {\em J. Reine Angew. Math.}, 830:101--139, 2026.

\bibitem{Escher-Seiler08}
J.~Escher and J.~Seiler.
\newblock Bounded {$H_\infty$}-calculus for pseudodifferential operators and applications to the {D}irichlet-{N}eumann operator.
\newblock {\em Trans. Amer. Math. Soc.}, 360(8):3945--3973, 2008.

\bibitem{FMM98}
E.~Fabes, O.~Mendez, and M.~Mitrea.
\newblock Boundary layers on {S}obolev-{B}esov spaces and {P}oisson's equation for the {L}aplacian in {L}ipschitz domains.
\newblock {\em J. Funct. Anal.}, 159(2):323--368, 1998.

\bibitem{Feppon-Ammari22}
F.~Feppon and H.~Ammari.
\newblock Modal decompositions and point scatterer approximations near the minnaert resonance frequencies.
\newblock {\em Stud. Appl. Math.}, 149(1):164--229, 2022.

\bibitem{FJKLfse}
S.~Fukushima, Y.-G. Ji, H.~Kang, and X.~Li.
\newblock Finiteness of the stress in presence of closely located inclusions with imperfect bonding.
\newblock {\em Math. Ann.}, 391(2):1753--1778, 2025.

\bibitem{Gilbarg-Trudinger01}
D.~Gilbarg and N.~S. Trudinger.
\newblock {\em Elliptic partial differential equations of second order}.
\newblock Classics in Mathematics. Springer-Verlag, Berlin, 2001.
\newblock Reprint of the 1998 edition.

\bibitem{Hashin01}
Z.~Hashin.
\newblock Thin interphase/imperfect interface in conduction.
\newblock {\em Journal of Applied Physics}, 89(4):2261--2267, 02 2001.

\bibitem{Hashin02}
Z.~Hashin.
\newblock Thin interphase/imperfect interface in elasticity with application to coated fiber composites.
\newblock {\em J. Mech. Phys. Solids}, 50(12):2509--2537, 2002.

\bibitem{Jerison-Kenig95}
D.~Jerison and C.~E. Kenig.
\newblock The inhomogeneous {Dirichlet} problem in {Lipschitz} domains.
\newblock {\em J. Funct. Anal.}, 130(1):161--219, 1995.

\bibitem{Ji-Kang23IMRN}
Y.-G. Ji and H.~Kang.
\newblock Spectrum of the {N}eumann-{P}oincar\'{e} operator and optimal estimates for transmission problems in the presence of two circular inclusions.
\newblock {\em Int. Math. Res. Not. IMRN}, 2023(9):7638--7685, 2023.

\bibitem{KKLSY16}
H.~Kang, K.~Kim, H.~Lee, J.~Shin, and S.~Yu.
\newblock Spectral properties of the {N}eumann-{P}oincar\'e{} operator and uniformity of estimates for the conductivity equation with complex coefficients.
\newblock {\em J. Lond. Math. Soc. (2)}, 93(2):519--545, 2016.

\bibitem{Kang-Li19}
H.~Kang and X.~Li.
\newblock Construction of weakly neutral inclusions of general shape by imperfect interfaces.
\newblock {\em SIAM J. Appl. Math.}, 79(1):396--414, 2019.

\bibitem{KPS07}
D.~Khavinson, M.~Putinar, and H.~S. Shapiro.
\newblock Poincar\'{e}'s variational problem in potential theory.
\newblock {\em Arch. Ration. Mech. Anal.}, 185(1):143--184, 2007.

\bibitem{Marschall87}
J.~Marschall.
\newblock Pseudodifferential operators with nonregular symbols of the class {$S^m_{\rho\delta}$}.
\newblock {\em Comm. Partial Differential Equations}, 12(8):921--965, 1987.

\bibitem{Mazya-Shaposhnikova05}
V.~Maz'ya and T.~Shaposhnikova.
\newblock Higher regularity in the layer potential theory for {L}ipschitz domains.
\newblock {\em Indiana Univ. Math. J.}, 54(1):99--142, 2005.

\bibitem{Milton23}
G.~W. Milton.
\newblock {\em The theory of composites}, volume~88 of {\em Class. Appl. Math.}
\newblock Philadelphia, PA: Society for Industrial {and} Applied Mathematics (SIAM), 2023.
\newblock Reprint of the 2002 edition.

\bibitem{Miranda70}
C.~Miranda.
\newblock {\em Partial differential equations of elliptic type}, volume Band 2 of {\em Ergebnisse der Mathematik und ihrer Grenzgebiete [Results in Mathematics and Related Areas]}.
\newblock Springer-Verlag, New York-Berlin, revised edition, 1970.

\bibitem{Mitrea97}
D.~Mitrea.
\newblock The method of layer potentials for non-smooth domains with arbitrary topology.
\newblock {\em Integral Equations Operator Theory}, 29(3):320--338, 1997.

\bibitem{MMP94}
D.~Mitrea, M.~Mitrea, and J.~Pipher.
\newblock Vector potential theory on nonsmooth domains in {$\mathbf{R}^3$} and applications to electromagnetic scattering.
\newblock {\em J. Fourier Anal. Appl.}, 3(2):131--192, 1997.

\bibitem{Mitrea94}
M.~Mitrea.
\newblock {\em Clifford wavelets, singular integrals, and {Hardy} spaces}, volume 1575 of {\em Lect. Notes Math.}
\newblock Berlin: Springer-Verlag, 1994.

\bibitem{Mitrea-Taylor00}
M.~Mitrea and M.~Taylor.
\newblock Potential theory on {Lipschitz} domains in {Riemannian} manifolds: {Sobolev}-{Besov} space results and the {Poisson} problem.
\newblock {\em J. Funct. Anal.}, 176(1):1--79, 2000.

\bibitem{Miyanishi22}
Y.~Miyanishi.
\newblock Weyl's law for the eigenvalues of the {N}eumann-{P}oincar\'{e} operators in three dimensions: {W}illmore energy and surface geometry.
\newblock {\em Adv. Math.}, 406:Paper No. 108547, 19, 2022.

\bibitem{Ouhabaz96}
E.-M. Ouhabaz.
\newblock Invariance of closed convex sets and domination criteria for semigroups.
\newblock {\em Potential Anal.}, 5(6):611--625, 1996.

\bibitem{Pernin99}
J.~N. Pernin.
\newblock Diffusion in composite solid: threshold phenomenon and homogenization.
\newblock {\em Internat. J. Engrg. Sci.}, 37(12):1597--1610, 1999.

\bibitem{Persson22}
B.~N.~J. Persson.
\newblock On the electric contact resistance.
\newblock {\em Tribol. Lett.}, 70(88):published online, 2022.

\bibitem{Smolic-Klajn21}
I.~Smolić and B.~Klajn.
\newblock Capacitance matrix revisited.
\newblock {\em Prog. Electromagn. Res. B}, 92:1--18, 2021.

\bibitem{Steinbach-Wendland01}
O.~Steinbach and W.~L. Wendland.
\newblock On {C}. {N}eumann's method for second-order elliptic systems in domains with non-smooth boundaries.
\newblock {\em J. Math. Anal. Appl.}, 262(2):733--748, 2001.

\bibitem{Taylor00}
M.~E. Taylor.
\newblock {\em Tools for {PDE}. {Pseudodifferential} operators, paradifferential operators, and layer potentials}, volume~81 of {\em Mathematical Surveys and Monographs}.
\newblock American Mathematical Society, Providence, RI, 2000.

\bibitem{TWSB19}
N.~O. Taylor, M.-T. Wei, H.~A. Stone, and C.~P. Brangwynne.
\newblock Quantifying dynamics in phase-separated condensates using fluorescence recovery after photobleaching.
\newblock {\em Biophys. J.}, 117(7):1285--1300, 2019.

\bibitem{terElst-Ouhabaz14}
A.~F.~M. ter Elst and E.~M. Ouhabaz.
\newblock Analysis of the heat kernel of the {Dirichlet}-to-{Neumann} operator.
\newblock {\em J. Funct. Anal.}, 267(11):4066--4109, 2014.

\bibitem{Torquato-Rintoul95}
S.~Torquato and M.~D. Rintoul.
\newblock Effect of the interface on the properties of composite media.
\newblock {\em Phys. Rev. Lett.}, 75:4067--4070, Nov 1995.

\bibitem{Triebel92}
H.~Triebel.
\newblock {\em Theory of function spaces {II}}, volume~84 of {\em Monogr. Math., Basel}.
\newblock Basel etc.: Birkh{\"a}user Verlag, 1992.

\bibitem{Verchota84}
G.~Verchota.
\newblock Layer potentials and regularity for the {D}irichlet problem for {L}aplace's equation in {L}ipschitz domains.
\newblock {\em J. Funct. Anal.}, 59(3):572--611, 1984.

\bibitem{VLA16}
L.~W. Votapka, C.~T. Lee, and R.~E. Amaro.
\newblock Two relations to estimate membrane permeability using milestoning.
\newblock {\em J. Phys. Chem. B}, 120(33):8606--8616, 2016.

\bibitem{Yosida74}
K.~Yosida.
\newblock {\em Functional analysis}, volume Band 123 of {\em Die Grundlehren der mathematischen Wissenschaften}.
\newblock Springer-Verlag, New York-Heidelberg, fourth edition, 1974.

\bibitem{Yun07}
K.~Yun.
\newblock Estimates for electric fields blown up between closely adjacent conductors with arbitrary shape.
\newblock {\em SIAM J. Appl. Math.}, 67(3):714--730, 2007.

\end{thebibliography}
\bibliographystyle{abbrv}

\end{document}